\documentclass[12pt]{amsart}
\usepackage{amssymb}
\usepackage{cite}
\usepackage{array}
\usepackage{booktabs}
\usepackage{mdwtab}
\usepackage{mathtools}
\usepackage{lmodern}
\usepackage[T1]{fontenc}
\usepackage[utf8]{inputenc}
\usepackage{hyphenat}
\usepackage{enumitem}
\usepackage{mathabx}
\usepackage{color}
\usepackage[pdftex]{graphicx}
\usepackage[pdftex,letterpaper,margin=1in]{geometry}
\usepackage{xr-hyper}
\PassOptionsToPackage{hyphens}{url}
\usepackage[bookmarks=true, bookmarksopen=true,%
    bookmarksopenlevel=1,%
    colorlinks=true,%
    linkcolor=blue,%
    citecolor=blue,%
    filecolor=blue,%
    menucolor=blue,%
    urlcolor=blue]{hyperref}
  \hypersetup{pdftitle={Elastic graphs}}
  \hypersetup{pdfauthor={Dylan P. Thurston}}
\usepackage{microtype}
\usepackage{tikz}
\usetikzlibrary{arrows}
\tikzstyle{every picture}=[> = to]
\tikzset{cdlabel/.style={execute at begin node=$\scriptstyle,execute at end node=$}}
\tikzset{implication/.style={double equal sign distance, -implies}}
\tikzset{biimplication/.style={double equal sign distance, implies-implies}}

\makeatletter
\newcommand\mi@kern[1]{%
  \settowidth\@tempdima{$\mi@obj^{#1}$}
  \kern-\@tempdima
  #1
  \settowidth\@tempdima{$\mi@obj$}
  \kern\@tempdima
}

\newtoks\mi@toksp
\newtoks\mi@toksb
\DeclareRobustCommand{\manyindices}[5]{
  \def\mi@obj{#5}
  \mi@toksp\expandafter{\mi@kern{#2}}
  \mi@toksb\expandafter{\mi@kern{#1}}
  \@mathmeasure4\textstyle{#5_{#1}^{#2}}
  \@mathmeasure6\textstyle{#5_{#3}^{#4}}
  \dimen0-\wd6 \advance\dimen0\wd4
  \@mathmeasure8\textstyle{\hphantom{{}_{#1}^{#2}}#5^{\the\mi@toksp#4}_{\the\mi@toksb#3}}
  \hbox to \dimen0{}{\kern-\dimen0\box8}
}
\makeatother 

\newread\testin

\def\mathcenter#1{\vcenter{\hbox{$#1$}}}
\def\grapha#1{\includegraphics{#1}}
\def\graphb#1{\includegraphics[trim=-1 -1 -1 -1]{#1}}
\def\mfig#1{\mathcenter{\grapha{#1}}}
\def\mfigb#1{\mathcenter{\graphb{#1}}}

\renewcommand{\colon}{\nobreak\mskip2mu\mathpunct{}\nonscript
  \mkern-\thinmuskip{:}\allowbreak\mskip6muplus1mu\relax}

\ifpdf
  
\else
  
\fi

\newcommand{\RR}{\mathbb R}
\newcommand{\CC}{\mathbb C}

\newcommand{\PP}{\mathbb P}

\newcommand{\cP}{\mathcal{P}}
\newcommand{\cS}{\mathcal{S}}
\newcommand{\cT}{\mathcal{T}}

\newcommand{\co}{\colon}

\renewcommand{\epsilon}{\varepsilon}
\newcommand{\abs}[1]{\lvert #1 \rvert}
\newcommand{\norm}[1]{\lVert #1 \rVert}

\newcommand{\bdy}{\partial}

\DeclareMathOperator*{\esssup}{ess\,sup}

\DeclareMathOperator{\Iter}{Iter}

\theoremstyle{plain}
\numberwithin{equation}{section}
\newtheorem{proposition}[equation]{Proposition}
\newtheorem{lemma}[equation]{Lemma}
\newtheorem{corollary}[equation]{Corollary}

\newtheorem{theorem}{Theorem}
\newtheorem{citethm}[equation]{Theorem}

\theoremstyle{definition}
\newtheorem{definition}[equation]{Definition}

\newtheorem{question}[equation]{Question}

\newtheorem{claim}[equation]{Claim}
\newtheorem{algorithm}[equation]{Algorithm}

\theoremstyle{remark}
\newtheorem{example}[equation]{Example}
\newtheorem{remark}[equation]{Remark}
\newtheorem{warning}[equation]{Warning}

\theoremstyle{plain}
\newtheorem{taggedthmx}{Theorem}
\newenvironment{taggedthm}[1]
 {\renewcommand\thetaggedthmx{#1}\taggedthmx}
 {\endtaggedthmx}

\DeclareMathOperator{\EL}{EL} 
\DeclareMathOperator{\SF}{SF} 
\DeclareMathOperator{\Fill}{Fill} 
\DeclareMathOperator{\Verts}{Vert}
\DeclareMathOperator{\Edge}{Edge}
\newcommand{\Edges}{\Edge}
\DeclareMathOperator{\Lip}{Lip} 
\DeclareMathOperator{\Dir}{Dir} 
\DeclareMathOperator{\Emb}{Emb} 
\DeclareMathOperator{\WR}{WR} 

\DeclareMathOperator{\Area}{Area}
\DeclareMathOperator{\Star}{Star} 
\DeclareMathOperator{\mincut}{mincut} 

\newcommand{\Wgt}{\mathcal{W}}
\newcommand{\Len}{\mathcal{L}}

\newcommand{\id}{\mathrm{id}}
\DeclareMathOperator{\Rlx}{Rlx} 
\newcommand{\DR}{\mathrm{DR}} 
\newcommand{\ER}{\mathrm{ER}} 
\newcommand{\NR}{\mathrm{NR}} 

\newcommand{\shortseq}[5]{#1 \overset{#2}{\longrightarrow} #3 \overset{#4}{\longrightarrow} #5}
\newcommand{\longseq}[7]{#1 \overset{#2}{\longrightarrow} #3 \overset{#4}{\longrightarrow} #5 \overset{#6}{\longrightarrow} #7}

\newcommand{\wt}[1]{\widetilde{#1}}

\newcommand{\pdual}{p^\vee}

\newcommand{\SFfrom}{\overrightarrow{\SF}}
\newcommand{\SFto}{\overleftarrow{\SF}}

\definecolor{dark-green}{rgb}{0,0.6,0}
\definecolor{dark-red}{rgb}{0.7,0,0}
\definecolor{dark-blue}{rgb}{0,0,0.8}

\graphicspath{{draws/}{figs/}}

\begin{document}
\title{Elastic Graphs}

\author[Thurston]{Dylan~P.~Thurston}
\address{Department of Mathematics\\
         Indiana University,
         Bloomington, Indiana 47405\\
         USA}
\email{dpthurst@indiana.edu}
\date{February 7, 2019}

\begin{abstract}
  An \emph{elastic graph} is a graph with an elasticity associated to
  each edge. It may be viewed as a network made out of ideal rubber
  bands. If the
  rubber bands are
  stretched on a target space there is an \emph{elastic energy}.  We
  characterize when a
  homotopy class of maps
  from one elastic graph to another is \emph{loosening}, i.e.,
  decreases this elastic energy for all possible targets. This fits
  into a more general framework of energies for maps between graphs.
\end{abstract}


\maketitle

\setcounter{tocdepth}{1}
\tableofcontents

\section{Introduction}
\label{sec:intro}

\subsection{Broader context}
\label{sec:broader-context}

We start with a general question.

\begin{question}\label{quest:looser}
  When is one network of elastic bands looser than another?
\end{question}

This question could have been asked long ago, as soon as Hooke's Law
was discovered. Informally, one network is looser than another if,
however they are stretched, the first network will always have lower
potential energy and force on the external vertices than the second.
(We will make this question precise shortly, but note for now that our
elastic bands have a resting length of~$0$, unlike real elastics or
springs.)

There is a closely related question for resistors:
\begin{question}\label{quest:looser-elec}
  When does one network of resistors dissipate more energy than another?
\end{question}
Elastic bands and resistors both have quadratic
responses to their inputs: the potential energy stored in an elastic
band is a quadratic function of how far it is stretched, while the
energy dissipated in a resistor is a quadratic function of the voltage
difference between the ends. (See Appendix~\ref{sec:electrical} for
more on this relationship.)

The electrical version of the question is much easier. Consider a
network of resistors with $k$ external \emph{nodes} where the voltage
can be controlled. Such a network has a symmetric $k \times k$
\emph{response matrix}, giving the vector of currents flowing out of
each node as a linear function of the vector of voltages. This same
matrix also gives a positive semi-definite quadratic form giving the
total energy dissipated, so Question~\ref{quest:looser-elec} amounts
to asking whether one quadratic form dominates another.

This answer does not translate to Question~\ref{quest:looser}. The key
difference between resistors and elastics is that electricity has a
direction of flow through a resistor, while elasticity just has an
unsigned tension. Thus in a resistor network the balancing condition
at a vertex (Kirchhoff's first law) is a linear equation, while in an
elastic network, stretched on a target space that is also a graph,
there are some combinatorics that play a role as well.
The difference between Questions~\ref{quest:looser-elec}
and~\ref{quest:looser} can be compared to the difference between
homology and homotopy, or between holomorphic differentials and
quadratic differentials.

In this paper we answer one interpretation of
Question~\ref{quest:looser}. The answers are different than for
Question~\ref{quest:looser-elec}; for instance,
the classical $Y$--$\Delta$ transform \cite{Kennelly99:TriangleStar},
giving an equivalence between different resistor networks, becomes an
inequality; cf.~\eqref{eq:Y-Delta-examp}. The answer is in terms of
minimizing a new ``embedding energy'' for maps between
graphs; cf.\ Equation~\eqref{eq:emb}.

There has been a great deal of work related to the electrical networks
version of the problem (or, equivalently, harmonic maps from a graph
to~$\RR$). Most prominently, since a harmonic $\RR$-valued function on
a simply-connected planar domain is the real part of an analytic
$\CC$-valued function, it can be used to get discrete analogues of
analytic functions. There is a 70-year history of work on the discrete
harmonic or analytic functions
\cite[\textit{inter alia}]{BSST40:DissectionRects, Isaacs41:Difference,
  Ferrand44:Preharmonic, Duffin68:Rhombic,PP93DiscreteMinimal,
  Mercat01:DiscreteRiemannIsing, Smirnov01:PercPlane,
  Lovasz04:DiscreteAnalytic, CS11:DiscreteIsoradial}. See Smirnov's
ICM lecture
\cite{Smirnov10:DiscreteComplexAnalysis} for some of the
history. These works rely crucially on the target being a vector
space. It is, for instance, not clear how to apply these
discretizations to compute the minimal dilatation map between Riemann
surfaces.

There is another line of work towards discrete approximations to
analytic functions, based on circle packings on the Riemann sphere, as
proposed by W.~Thurston \cite{Thurston86:Zippers} and developed by
others \cite{RS87:ConvergenceCirclePack, HS96:ConvergenceCirclePack,
  Stephenson05:IntroCirclePack}. Here the target of the map is the
sphere, not a
vector space, but it does use a complex projective
structure on the target to define the meaning of a ``circle''
\cite{KMT03:CircleProj}. The method also involves transcendental
equations.

In this paper, we are interested in more general targets. Concretely,
we work with maps to graphs rather than to $\RR$, $\CC$, or
$\CC\PP^1$. The equations that result have combinatorial data and,
once the combinatorics are fixed, are algebraic. For instance, for
harmonic maps from one graph to another,
the equations are linear once the combinatorial data is
fixed. Although we work
mainly with targets that are graphs, in fact
the main results apply to energies of harmonic maps with \emph{any}
possible target.

See Section~\ref{sec:prior-work} for more on
prior work.

Although in principle the results are quite general, the main
applications to date are in the setting of surfaces. We briefly
summarize the connection to surfaces here, although in the bulk of the
paper we will not refer to it.
A \emph{ribbon
  graph} is a graph together with a cyclic
ordering of the edges incident to each vertex; this is enough to give
a topological thickening of the graph to a surface. In the setting of
surfaces, it is traditional to look not at harmonic energy but at a
dual notion, \emph{extremal length}. This is based on homotopy classes
of loops mapped into the surface, rather than on maps from the surface
to some other space. (Definition~\ref{def:el-curve} has the
corresponding notion for graphs.) In the graph setting, the control we
get over when one network is looser than another also gives bounds on
the maximal ratio of extremal lengths, maximized over all homotopy
classes of loops. This turns out
to be quite useful, because
for closed surfaces, the
maximum ratio of extremal lengths over all homotopy classes gives the
minimal
dilation of a quasi-conformal map
between the surfaces \cite[Theorem~4]{Kerckhoff80:AsympTeich}, while
for surfaces with boundary, this ratio controls whether one surface
conformally embeds in another \cite{KPT15:EmbeddingEL}.

In turn, this control over when one surface conformally embeds in
another (plus a relation between extremal length on a graph and
extremal length on its thickening) lets us prove a new
characterization of when a topological branched self-cover of the
sphere is equivalent to a rational map
\cite{Thurston16:Characterize}. This gives a converse to an earlier
characterization by W. Thurston \cite{DH93:ThurstonChar}.

You can also use these techniques and the relation between graphs and
surfaces to approximate the
Teichmüller distance between two Riemann surfaces
\cite{Palmer15:harmonic-thesis}. In principle, that work does not
use the results of this paper, since we need a more flexible notion of
``looser'' than the one considered here. But the motivation and
techniques are similar.

\subsection{Definitions and results}
\label{sec:main-results}
We now turn to making Question~\ref{quest:looser} precise.

\begin{definition}
  In this paper, a \emph{graph} is a topologist's graph, a finite
  1-dimensional CW complex, i.e., with multiple edges and self-loops
  allowed. \emph{Maps} between graphs are continuous maps between the
  underlying topological spaces. In particular, they need not send
  vertices to vertices. For convenience, we will work with
  maps that are piecewise-linear (PL) with respect to a fixed linear
  structure on each
  edge.
\end{definition}

\begin{definition}\label{def:marked-graph}
  A \emph{marked graph} is a pair $(\Gamma, M)$,
  where $\Gamma$ is a graph and $M \subset \Verts(\Gamma)$ is a finite subset of marked
  points (possibly empty). A map $f \co (\Gamma_1, M_1) \to (\Gamma_2, M_2)$
  between marked
  graphs is required to send marked points to marked points (i.e.,
  $f(M_1) \subset M_2$). Homotopy is considered within the space of
  such maps. In particular, within a homotopy class the restriction of
  $f$ to a map from $M_1$
  to $M_2$ is fixed.
\end{definition}

\begin{definition}\label{def:length-graph}
  A \emph{length graph} $K = (\Gamma, \ell)$ is a graph in which each edge~$e$
  has a positive
  length $\ell(e)$. This gives a metric on~$\Gamma$ respecting the linear
  structure on the edges. If $f \co K_1 \to K_2$ is a PL map between
  length graphs, then $\abs{f'} \co K_1 \to \RR_{\ge 0}$ is the
  absolute value of the
  derivative with respect to the metrics. It is locally constant with
  jump discontinuities.
\end{definition}

\begin{definition}\label{def:elastic-graph}
  An \emph{elastic graph}~$G=(\Gamma,\alpha)$ is a graph~$\Gamma$ in which
  each edge~$e$ has
  a positive elastic constant $\alpha(e)$. We can take derivatives by
  interpreting $\alpha(e)$ as the length of~$e$ as above.
  For $G$ an elastic graph, $K$ a
  length graph, and $f \co G \to K$ a PL map, the
  \emph{Dirichlet energy} of~$f$ is
  \begin{align}
    \Dir(f) &\coloneqq \int_{x \in \Gamma} \abs{f'(x)}^2\,dx.\label{eq:dir}\\
  \intertext{The Dirichlet energy of a (marked) homotopy
    class~$[f]$ is defined to be}
    \Dir[f] &\coloneqq \inf_{g \in [f]} \Dir(g).\label{eq:dir-homotopy}
  \end{align}
\end{definition}

It is easy to see that a minimizer~$g$ for $\Dir[f]$ must be
\emph{constant-derivative}: $\abs{g'}$ is constant on each edge of~$g$
(Definition~\ref{def:const-deriv}). If $g$ is constant-derivative, we
then have
\begin{equation}\label{eq:dir-1}
\Dir(g) = \sum_{e \in \Edge(\Gamma)} \frac{\ell(g(e))^2}{\alpha(e)},
\end{equation}
where $\ell(g(e))$ is the length of the image of~$e$ in a
natural sense (Definition~\ref{def:length}). This is
the Hooke's law energy of~$g$, where each edge is an ideal spring
with resting length~$0$ and spring constant given by~$\alpha$.
Physically, you could think about stretching a rubber band graph
shaped like~$G$ over
a system of pipes shaped like~$K$.

\emph{Harmonic} maps are maps that minimize the energy
\eqref{eq:dir} or~\eqref{eq:dir-1} within their homotopy class.
\begin{figure}
  \[
  \begin{tikzpicture}
    \node (A0) at (0,0) {$\mfigb{vertex-tt-1}$};
    \node (A1) at (4,0) {$\mfigb{vertex-tt-0}$};
    \draw[->] (A0) to (A1);

    \node (B0) at (0,-2.5) {$\mfig{vertex-tt-11}$};
    \node (B1) at (4,-2.5) {$\mfig{vertex-tt-10}$};
    \draw[->] (B0) to (B1);
  \end{tikzpicture}
  \]
  \caption{Local schematic pictures of
    the equilibrium condition for harmonic
    maps. Top: A vertex maps
    to an edge. Bottom: A vertex maps to a vertex.}
  \label{fig:equil-examp}
\end{figure}
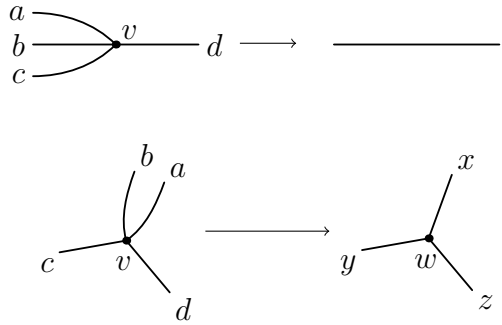
Intuitively, a map is harmonic if it is constant-derivative and each
vertex is at equilibrium,
in the sense that the force pulling in one direction is less
than or equal to the total force pulling it in other directions.
Physically, the force pulling a vertex in the direction of
an incident edge~$e$ is given by the derivative of
Equation~\eqref{eq:dir-1} with respect to $\ell(g(e))$, i.e.,
$2\ell(g(e))/\alpha(e) = 2\abs{f'(e)}$. We will drop the irrelevant
factor of~$2$ and refer to $\abs{f'(e)}$ as the \emph{tension} in the
edge (Definition~\ref{def:harmonic}).

For instance, if a vertex~$v$ of~$G$
maps to the middle of an edge of~$K$ with three edges on the left and
one edge on the right as
on top of Figure~\ref{fig:equil-examp}, the net
force to the left must equal the net force to the right:
\[
\abs{f'(a)} + \abs{f'(b)} + \abs{f'(c)} = \abs{f'(d)}.
\]
On the other hand, if a vertex~$v$ of~$G$ maps to a vertex~$w$
of~$K$ as on the bottom of Figure~\ref{fig:equil-examp}, we no longer
have an equality. Instead, the net force
pulling~$v$ into any one of the three edges incident to~$w$ can't be
greater than the total force pulling it into the other two
edges. This gives three triangle inequalities:
\begin{equation}\label{eq:triangle-examp}
\begin{aligned}
  \abs{f'(a)} + \abs{f'(b)} &\le \abs{f'(c)} + \abs{f'(d)}\\
  \abs{f'(c)} &\le \abs{f'(a)} + \abs{f'(b)} + \abs{f'(d)}\\
  \abs{f'(d)} &\le \abs{f'(a)} + \abs{f'(b)} + \abs{f'(c)}.
\end{aligned}
\end{equation}
(See Corollary~\ref{cor:harmonic-triangle} for the general statement.)
In any homotopy class of maps to a
target length graph there is a harmonic map
(Theorem~\ref{thm:harmonic-min}).

\begin{example}\label{examp:harmonic}
  Consider an elastic graph~$G$ and length graph~$K$, both tripods
  with their ends marked, with
  elastic constants and lengths given by
  \begin{align*}
    G &= \mfigb{energies-107} & K &= \mfigb{energies-1},
  \end{align*}
  where $x$ varies. If $x=3$, the map $f$ minimizing
  Dirichlet energy takes
  the vertex of~$G$ to a point $1/5$ of the way along an edge:
  \[
  \mfigb{energies-103} \overset{f}{\longrightarrow} \mfigb{energies-10}.
  \]
  The net force pulling the vertex of~$G$ upwards is
  $4/5$, the same as the
  net force pulling it downwards.
  On the other hand, if $0 < x \le 2$, then the harmonic representative
  has the central vertex of~$G$ mapping to the central vertex of~$K$.
\end{example}

We can think of Dirichlet energy as a function of the target lengths (or more
precisely the target lengths and the homotopy class). We next compare
these functions.

\begin{definition}\label{def:loosening}
  Given a homotopy class $[\phi] \co G_1 \to G_2$ of maps
  between marked elastic graphs, we say that $[\phi]$ is
  \emph{loosening} if, for all marked length graphs~$K$ and marked maps
  $f \co G_2 \to K$,
  \[
  \Dir[f \circ \phi] \le \Dir[f].
  \]
  If the inequality is always strict, we say that $[\phi]$ is
  \emph{strictly loosening}.
  For a finer invariant, define the \emph{Dirichlet stretch factor} to be
  \begin{equation}\label{eq:sf-dir}
  \SF_{\Dir}[\phi] \coloneqq \sup_{[f] \co G_2 \to K} \frac{\Dir[f \circ \phi]} {\Dir[f]},
  \end{equation}
  where the supremum runs over all marked length graphs~$K$ and all
  marked homotopy classes~$[f]$, so that $[\phi]$ is loosening
  iff $\SF_{\Dir}[\phi] \le 1$. It is less obvious but also true that $[\phi]$ is
  strictly loosening iff $\SF_{\Dir}[\phi] < 1$.
\end{definition}

\begin{definition}\label{def:embedding}
  For $\phi \co G_1 \to G_2$ a PL map between marked
  elastic graphs,
  the \emph{embedding energy} of~$\phi$ is
  \begin{equation}\label{eq:emb}
    \Emb(\phi) \coloneqq \esssup_{y \in G_2} \sum_{x \in \phi^{-1}(y)} \abs{\phi'(x)}.
  \end{equation}
  We are taking the essential supremum over points $y \in G_2$,
  ignoring in particular
  vertices of~$G_2$, images of vertices
  of~$G_1$, and points where
  $\phi^{-1}(y)$ is infinite, all of
  which have measure zero. For a homotopy class~$[\phi]$, define
  \begin{equation}\label{eq:emb-homotopy}
  \Emb[\phi] \coloneqq \inf_{\psi \in [\phi]} \Emb(\psi).
  \end{equation}
\end{definition}

\begin{example}\label{examp:emb}
  If $G_1$ and $G_2$ are both tripods with marked ends, with elastic
  constants $(1,3,3)$ and $(1,1,1)$, then the minimizer for
  $\Emb(\phi)$ is not the map from Example~\ref{examp:harmonic}, but
  instead sends the vertex of~$G_1$ to a point $2/5$ of the way along
  an edge of~$G_2$, with $\Emb(\phi) = 3/5$:
  \[
  \mfigb{energies-104} \overset{\phi}{\longrightarrow} \mfigb{energies-11}.
  \]
\end{example}

Our answer to Question~\ref{quest:looser} says that
$\Emb[\phi] = \SF_{\Dir}[\phi]$, and in particular the homotopy class $[\phi]$ is
loosening exactly when $\Emb[\phi] \le 1$. Before stating the theorem,
we way to measure the tightness of an elastic network~$G$, using
maps to~$G$
rather than maps from~$G$.

\begin{definition}\label{def:curve}
  A \emph{marked one-manifold} is a marked graph in which every marked
  point is a 1-valent vertex and every other vertex is 2-valent;
  equivalently, it is a (not necessarily connected)
  one\hyp manifold~$C$ with boundary, where the set of marked points
  is equal to $\bdy C$.
  A \emph{marked multi-curve} on a marked graph~$\Gamma$ is a marked
  one\hyp manifold $C$ and a marked map
  $c \co C \to \Gamma$.  Thus, it is a union of loops on~$\Gamma$ and arcs
  between marked points of~$\Gamma$. It is a \emph{marked curve} if
  $C$ is connected.
  A marked
  multi-curve is \emph{reduced} if it is PL and has no
  backtracking: on each component of~$C$, $c$ is either
  constant or has a perturbation that is locally injective. In each
  homotopy class of multi-curves
  there is an essentially unique reduced representative. If $(C,c)$
  is a marked multi-curve on~$\Gamma$, then
  for $y \in \Gamma$, 
  define $n_c(y)$ to be the size of
  $c^{-1}(y)$. If $c$ is reduced and $e$ is an edge of~$\Gamma$, then
  $n_c(y)$ is constant almost everywhere on~$e$, so we may
  write $n_c(e)$.

  A \emph{marked weighted multi-curve} is a marked multi-curve
  in which each component $C_i$ has a positive weight~$w_i$. For
  weighted multi-curves,
  $n_c(y)$ is the weighted
  count of $c^{-1}(y)$.
\end{definition}

\begin{definition}\label{def:el-curve}
  The \emph{extremal length} of a reduced multi-curve~$c$ on an
  elastic graph~$G = (\Gamma, \alpha)$ is
  \begin{equation}\label{eq:el-curve}
    \EL(c) \coloneqq \sum_{e \in \Edge(G)} \alpha(e) n_c(e)^2.
  \end{equation}
  This is also the extremal length of the homotopy class~$[c]$; see
  Equation~\eqref{eq:el-weight-def} for extremal length of non-reduced
  multi-curves.
\end{definition}

See \cite[Section~5.2]{Thurston16:RubberBands} for motivation on why
this is called extremal length.

\begin{definition}
  For $[\phi] \co G_1 \to G_2$ a homotopy class of maps
  between elastic marked graphs, the \emph{extremal length
    stretch factor} is
  \begin{equation}\label{eq:sf-el}
    \SF_{\EL}[\phi] \coloneqq \sup_{[c]\colon C \to G_1} \frac{\EL[\phi \circ c]}{\EL[c]}
  \end{equation}
  where the supremum runs over all homotopy classes of marked
  multi-curves $(C,c)$ on~$G_1$.
\end{definition}

We are now ready to state the main result of this paper.
\begin{theorem}\label{thm:emb-sf}
  For $[\phi] \co G_1 \to G_2$ a homotopy class of maps
  between marked elastic graphs,
  \begin{equation}\label{eq:emb-sf}
  \Emb[\phi] = \SF_{\Dir}[\phi] = \SF_{\EL}[\phi].
  \end{equation}
  Furthermore, there is
  \begin{itemize}
  \item a PL map $\psi\in[\phi]$ realizing $\Emb[\phi]$;
  \item a marked length graph~$K$ and a PL map $f \co G_2 \to K$ realizing
    $\SF_{\Dir}$, in the sense that $f$ and $f \circ \phi$ are both
    harmonic and
    \[
    \SF_{\Dir}[\phi]
      = \frac{\Dir(f \circ \psi)}{\Dir(f)};
    \]
    and
  \item a marked weighted multi-curve~$(C,c)$ on $G_1$
    realizing $\SF_{\EL}$, in the sense that $c$ and $\psi \circ c$
    are both reduced and
    \[
    \SF_{\EL}[\phi] =
    \frac{\EL(\psi \circ c)}{\EL(c)}.
    \]
  \end{itemize}
\end{theorem}

Note that $\Emb[\phi]$ is defined as an infimum (over the
homotopy class), while $\SF_{\Dir}[\phi]$ and $\SF_{\EL}[\phi]$ are
defined as suprema (over all possible targets or
multi-curves). As such,
Equation~\eqref{eq:emb-sf} helps us compute these quantities, by
sandwiching the target value.

We also give an algorithm that produces,
simultaneously, the map $\psi \in [\phi]$ minimizing $\Emb$, the
pair $(K,f)$ maximizing the ratio of Dirichlet energies, and the
reduced multi-curve
$(C,c)$ maximizing the ratio of extremal lengths.

\begin{example}
  Example~\ref{examp:emb} fits into the following
  sequence of maps realizing $\SF_{\EL}$ and $\SF_{\Dir}$:
  \begin{gather*}
  C \overset{c}{\longrightarrow} G_1
    \overset{\phi}{\longrightarrow} G_2
    \overset{g}{\longrightarrow} K\\
  \mfigb{energies-4}
    \overset{c}{\longrightarrow} \mfigb{energies-103}
    \overset{\phi}{\longrightarrow} \mfigb{energies-101}
    \overset{g}{\longrightarrow} \mfigb{energies-5}
  \end{gather*}
  Here $c$ is the evident map, $\phi$ is the map from
  Example~\ref{examp:emb}, and $g$ is the map that takes the vertex
  of~$G_2$ to the vertex of~$K$. Specifically, we have
  \[
  \frac{\EL[\phi \circ c]}{\EL[c]} = \frac{4+1+1}{4+3+3}
    = \Emb(\phi) = 3/5
    = \frac{\Dir[g \circ \phi]}{\Dir[g]} = \frac{36/25 + 27/25+27/25}{4+1+1}.
  \]
\end{example}

\medskip

We can generalize the target spaces in Theorem~\ref{thm:emb-sf}
considerably.

\begin{definition}\label{def:dir-general}
  Let $G$ be a marked elastic graph and let $X$ be a marked length
  space.%
  \footnote{A \emph{length space} is a metric space where the
    distance is given by the infimum of lengths of paths.
    \emph{Marked} means that there is a distinguished finite set of points.}
  Let $f \co G \to X$ be a Lipschitz map taking marked points to
  marked points. Define the
  Dirichlet energy of~$f$ by
  \begin{align*}
    \Dir(f) &\coloneqq \int_{x \in G} \abs{f'(x)}^2\,dx\\
  \intertext{where $\abs{f'(x)}$ is the best local Lipschitz constant of~$f$ in a
  neighborhood of~$x$.
  For a homotopy class of marked maps, define}
    \Dir[f] &\coloneqq \inf_{g \in [f]} \Dir(g).
  \end{align*}
\end{definition}
In this generality, minimizers for
$\Dir[\phi]$ need not exist. (See
\cite{KS93:SobolevHarmonic,EF01:HarmonicPolyhedra} for some cases
where minimizers do exist.)

\begin{theorem}\label{thm:emb-sf-gen}
  For $[\phi] \co G_1 \to G_2$ a homotopy class of maps
  between marked elastic graphs,
  \[
  \Emb[\phi] = \sup_{[f]\co G_2 \to X} \frac{\Dir[f \circ \phi]}{\Dir[f]},
  \]
  where the supremum runs over all marked length spaces~$X$ and all homotopy
  classes of marked maps $f \co G_2 \to X$ with $\Dir[f] > 0$.
\end{theorem}

We next put Theorem~\ref{thm:emb-sf} in the broader context of
a family of energies
of maps between weighted graphs, elastic graphs, and length graphs:
\begin{equation}
\mathcenter{\begin{tikzpicture}[node distance=3.5cm,
  every text node part/.style={align=center},align=center]
  \node (weight) {Weighted\\graph $W$};
  \node (elastic) [right of=weight] {Elastic\\graph $G$};
  \node (length) [right of=elastic] {Length\\graph $K$};
  \draw[->] (weight) to node[above,cdlabel]{$\sqrt{\EL}$} (elastic);
  \draw[->] (elastic) to node[above,cdlabel]{$\sqrt{\Dir}$} (length);
  \draw[->, bend left=28] (weight) to node[above,cdlabel]{$\ell$} (length);
  \draw[->, loop below, min distance=0.6cm, out=-105, in=-75]
    (weight) to node[below,cdlabel]{$\WR$} (weight);
  \draw[->, loop below, min distance=0.6cm, out=-105, in=-75]
    (elastic) to node[below,cdlabel]{$\sqrt{\Emb}$} (elastic);
  \draw[->, loop below, min distance=0.6cm, out=-105, in=-75]
    (length) to node[below,cdlabel]{$\Lip$} (length);
\end{tikzpicture}}
\label{eq:energies-graph}
\end{equation}
(Weighted graphs, Definition~\ref{def:weighted-graph}, are a mild
generalization of weighted multi-curves.)
The label on an arrow gives the appropriate energy of a map between
the given types of graphs.
$\EL(f)$, $\Emb(f)$ and $\Dir(f)$ are as defined above, except that we
take the square root for reasons to be explained shortly, and $\EL$
has been extended from multi-curves to maps from general weighted graphs in a
natural way (Equation~\eqref{eq:el-weight-def}).
We make some additional definitions.
\begin{itemize}
\item For $f \co W \to K$ a PL map from a weighted graph to a length graph,
  $\ell(f)$ is the weighted length of the image of~$f$:
  \begin{equation}\label{eq:ell-def}
    \ell(f) \coloneqq \int_{x \in W} w(x) \abs{f'(x)}\,dx.
  \end{equation}
\item For $f \co K_1 \to K_2$ a PL map between length graphs,
  $\Lip(f)$ is the best global Lipschitz constant for~$f$:
  \begin{equation}\label{eq:lip-def}
  \Lip(f) \coloneqq \esssup_{x \in K_1} \abs{f'(x)}.
  \end{equation}
\item For $f \co W_1 \to W_2$ a PL map between weighted
  graphs, the \emph{weight ratio} is the maximum ratio of weights:
  \begin{equation}\label{eq:WR-def}
    \WR(c) \coloneqq \esssup_{y \in W_2} \frac{\sum_{x \in f^{-1}(y)}
    w(x)}{w(y)}.
  \end{equation}
\end{itemize}
A key fact is that these energies are \emph{submultiplicative}, in the
sense that the energy of a composition of two maps is less than or
equal to the product of the energies of the two pieces. To state this
uniformly, we make some further definitions.

\begin{definition}\label{def:12inf-graphs}
  For $p \in \{1,2,\infty\}$, a \emph{$p$-conformal graph} $G^p$ is
  \begin{itemize}
  \item for $p=1$, a weighted graph;
  \item for $p=2$, an elastic graph; and
  \item for $p=\infty$, a length graph.
  \end{itemize}
  Let $p,q \in \{1,2,\infty\}$ with $p \le q$. If we have a
  $p$-conformal graph~$G_1$, a $q$-conformal graph~$G_2$, and a PL map
  $f \co G_1^p \to G_2^q$ between them, define
  $E^p_q(f)$ by
  \begin{align*}
    E^1_1(f) &\coloneqq \WR(f)  & E^1_2(f) &\coloneqq \sqrt{\EL(f)} &  E^1_\infty(f) &\coloneqq \ell(f)\\
            &         & E^2_2(f) &\coloneqq \sqrt{\Emb(f)} & E^2_\infty(f) &\coloneqq \sqrt{\Dir(f)}\\
            &         &          &                 & E^\infty_\infty(f) &\coloneqq\Lip(f).\vphantom{\sqrt{f}}
  \end{align*}
  In each case, for a homotopy class $[f]$ of maps, define 
  \[
  E^p_q[f] \coloneqq \inf_{g \in [f]} E^p_q(g).
  \]
\end{definition}
Let $p,q,r \in \{1,2,\infty\}$ with
$p \le q \le r$. Suppose we have a sequence of maps
\[\shortseq{G_1^p}{f}{G_2^q}{g}{G_3^r}\]
between marked graphs of the respective types.
Then the energies are submultiplicative:
\begin{align}
  E^p_r(g \circ f) &\le E^p_q(f) E^q_r(g)\label{eq:submul}\\
  E^p_r[g \circ f] &\le E^p_q[f] E^q_r[g].\label{eq:submul-homotop}
\end{align}
See Proposition~\ref{prop:sub-mult} for a more detailed statement,
spelling out the cases.

More generally, $p$-conformal graphs and energies $E^p_q$ can be defined
uniformly for real numbers $p, q$ with $1 \le p \le q \le \infty$. See
Appendix~\ref{sec:energies}.

Theorem~\ref{thm:emb-sf} can be interpreted as saying that some cases of
Equation~\eqref{eq:submul-homotop} are tight, as follows.

\begin{definition}\label{def:tight}
  Let $p,q \in \{1,2,\infty\}$ with $p \le q$. If
  $f \co G_1^p \to G_2^q$ is a PL map between marked
  graphs of the respective conformal type and
  \[
  E^p_q(f) = E^p_q[f],
  \]
  we say that $f$ is an \emph{energy minimizer} (in its homotopy class).

  Let $p,q,r \in \{1,2,\infty\}$ with $p \le q \le r$. If we have a
  sequence of maps
  \[
  \shortseq{G_1^p}{f}{G_2^q}{g}{G_3^r}
  \]
  of PL maps between graphs of the respective conformal type, we say
  that this sequence is \emph{tight} if
  $g \circ f$ is an energy minimizer and Equation~\eqref{eq:submul} is sharp:
  \[
    E^p_r[g \circ f] = E^p_r(g \circ f) = E^p_q(f) E^q_r(g).
  \]
  We may similarly say that a longer sequence of maps is tight.
\end{definition}

\begin{lemma}\label{lem:tight}
  Let $\shortseq{G_1^p}{f}{G_2^q}{g}{G_3^r}$ be a sequence of maps. If
  the sequence is tight, then $f$ and $g$ are also energy
  minimizers. Conversely, if $f$ and $g$ are energy minimizers and
  $E^p_r[g \circ f] = E^p_q[f]E^q_r[g]$, then the sequence is tight.
\end{lemma}

\begin{proof}
  In any sequence of maps $\shortseq{G_1^p}{f}{G_2^q}{g}{G_3^r}$
  (tight or not), we have the following inequalities:
  \begin{gather*}
      E^p_r[g \circ f] \le E^p_r(g \circ f) \le E^p_q(f)E^q_r(g)\\
      E^p_r[g \circ f] \le E^p_q[f]E^q_r[g] \le E^p_q(f)E^q_r(g).
    \end{gather*}
    (In both lines, we use Proposition~\ref{prop:sub-mult} and the
    definition of $E^p_q[f]$ as an infimum.)
    The hypothesis of the first (resp.\ second) claim in the statement
    is that the inequalities in
    the top (resp.\ bottom) line are equalities. In either case, the
    inequalities in the other line must be equalities as well.
\end{proof}

In this language, Theorem~\ref{thm:emb-sf} says that, for every
homotopy class of maps $[\phi] \co G_1 \to G_2$ between marked elastic
graphs, there is a corresponding tight sequence
\[
\longseq{C}{c}{G_1}{\psi}{G_2}{f}{K}
\]
with $\psi \in [\phi]$, $(C,c)$ a weighted multi-curve on~$G_1$, and $K$ a
length graph.

\subsection{Structure of the paper}
\label{sec:structure}

The paper is organized by the type of energy under consideration.
In the course of the paper, we also prove several other tightness
results.
\begin{itemize}
\item In Section~\ref{sec:taut}, we prove
  Theorem~\ref{thm:maxflow-mincut}, which gives
  a version of the max-flow/min-cut theorem in the context of
  maps between graphs. This gives a characterization of which maps
  minimize $E^1_p$ for any~$p$. This is surprisingly
  involved, and feels like it should be standard, but we have been
  unable to find a reference.
\item In Section~\ref{sec:lipschitz}, we prove
  Theorem~\ref{thm:lipschitz-stretch}, which shows that for any
  homotopy class $[\phi] \co K_1 \to K_2$ of maps between marked length
  graphs, there is a tight sequence
  \[
  \shortseq{C}{c}{K_1}{\psi}{K_2}
  \]
  where $\psi \in [\phi]$. This is a theorem of
  White~\cite{Bestvina11:Bers}: the minimal Lipschitz
  stretch factor between metric graphs equals the maximal ratio by
  which lengths of multi-curves are changed.
  We also extend the graphs we are looking at to allow \emph{weak
    graphs}, where lengths are allowed to be zero, and prove the
  corresponding tightness result
  (Theorem~\ref{thm:weak-stretch}).
\item In Section~\ref{sec:harmonic}, we prove
  Theorem~\ref{thm:harmonic-min}, which shows that for any homotopy
  class $[f] \co G \to K$ of maps from a marked elastic graph to a
  marked length graph, there is a tight sequence
  \[
  \shortseq{C}{c}{G}{g}{K}
  \]
  where $g \in [f]$ is harmonic. Again, we
  extend the results to the setting of weak graphs
  (Theorem~\ref{thm:harmonic-min-weak}).
\item Section~\ref{sec:filling} has the proof of the main results,
  Theorems~\ref{thm:emb-sf} and~\ref{thm:emb-sf-gen}, characterizing minima
  $\Emb(\phi)$ in a homotopy class.
\item Appendix~\ref{sec:energies} defines a family of energies
  $E^p_q(f)$ for any $1 \le p \le q \le \infty$ and any PL map $f$
  between metric graphs.
  Theorem~\ref{thm:energy-sf} gives the most general
  statement: for any $1 \le p \le q \le \infty$ and homotopy class of
  maps $[\phi] \co G \to H$ between marked graphs with the appropriate
  extra structure, we can find $\psi \in [\phi]$ and a tight sequence
  \[
  \longseq{C}{c}{G}{\psi}{H}{f}{K}.
  \]
  (The proof of Theorem~\ref{thm:energy-sf} is only sketched.)
\item Finally, Appendix~\ref{sec:electrical}
relates the elastic networks of this paper and the more
well-studied subject of resistor networks and electrical equivalence.
\end{itemize}

\begin{table}
  \caption{Structures on graphs and maps in this paper. See also
    Definition~\ref{def:12inf-graphs}.}
  \centering
  \begin{tabular}{@{}llcl@{}}
    \toprule
    Concept & Summary & Letters & Reference \\ \midrule
    \multicolumn{2}{l}{Discrete structures on graphs} & $\Gamma$ \\
    \quad Marked graph & With special marked points && Def.\ \ref{def:marked-graph}\\
    \quad Multi-curve & $1$-manifold & $C, c$ & Def.\ \ref{def:curve} \\
    \quad Weak graph & Null edges of $0$ length allowed && Def.\ \ref{def:weak-graph}\\
    \quad Train track & Edges partitioned into gates & $T$ & Def.\ \ref{def:tt}\\[4pt]
    \multicolumn{2}{l}{Metric or numeric structures on graphs} \\
    \quad Weighted & With multiplicity,  $1$-conformal & $W, w, \Wgt$ & Def.\ \ref{def:weighted-graph} \\
    \qquad Balanced & \multicolumn{2}{l}{Weights satisfying $\Delta$ inequalities} & Def.\ \ref{def:tt} \\
    \quad Length & Metric graph, $\infty$-conformal & $K, \ell, \Len$ & Def.\ \ref{def:length-graph}\\
    \quad Elastic & Stretchable edges, $2$-conformal & $G, \alpha$ & Def.\ \ref{def:elastic-graph} \\
    \quad Strip & Weighted \& length structure & $S,(w,\ell)$ & Def.\ \ref{def:strip-graph}\\
    \quad $p$-conformal & Generalization of above & $G,\alpha$ & Def.\ \ref{def:p-graph}\\[4pt]
    \multicolumn{2}{l}{Types of maps between graphs}\\
    \quad Edge-reduced & \multicolumn{2}{l}{No backtracking on individual edges} & Def.\ \ref{def:edge-reduced}\\
    \quad Reduced & No backtracking && Def.\ \ref{def:reduced}\\
    \quad Taut & Minimizing $E^1_p$ for all $p$ & & Def.\ \ref{def:taut}\\
    \quad Strongly reduced & Taut for some choice of weights & & Def.\ \ref{def:strongly-reduced} \\
    \quad Constant-derivative & $\abs{f'}$ constant on edges & & Def.\ \ref{def:const-deriv} \\
    \quad Harmonic & Minimizing $E^2_\infty$ & $f$ & Def.\ \ref{def:harmonic}\\
    \quad (Partially) $\lambda$-filling & Minimizing $E^2_2$ & $\phi$ & Defs.\ \ref{def:lambda-filling}, \ref{def:partial-filling}\\
    \quad Relaxed & Only metric information & $r$ & Def.\ \ref{def:relaxed}\\
    \bottomrule
  \end{tabular}
  \label{tab:structures-graphs}
\end{table}
See Table~\ref{tab:structures-graphs} for the various structures we
use on graphs and on maps between them.

Throughout the paper, we develop the theory behind these concepts,
more than is necessary for the proof of the main
Theorem~\ref{thm:emb-sf}. For instance, the theory of taut maps in
Section~\ref{sec:taut} is more general than needed for the maps that
actually arise in Sections~\ref{sec:harmonic} and~\ref{sec:filling}.

Many of the definitions and proofs are made more lengthy by the need
to deal with edges of weight or length~$0$. (These arise when dealing
with PL maps that are constant on some edges.) To get an overall
view, these extra complications are unnecessary. Thus we recommend
skipping the following sections on first reading:
\begin{itemize}
\item Assume that all reduced maps are locally injective in
  Definition~\ref{def:reduced}. This makes much of
  Section~\ref{sec:taut} trivial.
\item The notion of \emph{weak graphs} and \emph{weak maps} in
  Section~\ref{sec:lipschitz}. With that simplification, the older
  Theorem~\ref{thm:lipschitz-stretch} suffices (not the more
  complicated Theorem~\ref{thm:weak-stretch}), and almost the entire
  section can be skipped.
\item Section~\ref{sec:harmonic-weak}.
\item Section~\ref{sec:partial-lambda-filling}.
\end{itemize}

In addition, for readers familiar with Riemann surfaces and especially
flip-translation surfaces, the discussion about rectangle tilings at the
end of Appendix~\ref{sec:electrical} may be helpful.

\subsection{Prior and related work}
\label{sec:prior-work}

In addition to the work on discrete harmonic functions surveyed at the
beginning of the introduction,
harmonic maps between manifolds and singular spaces have been studied
for a long time, and there is a great deal of work on various cases
\cite[\textit{inter alia}]{GS92:Harmonic,KS93:SobolevHarmonic,Wolf95:JSHarmonic,Picard05:MartingalesTrees}.
In
particular, Eells and Fuglede proved that in every homotopy class of
maps between suitable (marked) Riemannian polyhedra there is a harmonic
representative \cite[Theorem 11.1]{EF01:HarmonicPolyhedra}, which is a
large part of our Theorem~\ref{thm:harmonic-min}. Their theory is
much more general. We reprove these results in our
more elementary setting, while also providing further connections to,
e.g., the theory of extremal length.

There have also been various notion of extremal length on graphs.
Duffin~\cite{Duffin62:ELNetwork} was one of the first, defining
extremal length between two nodes in a planar graph, using definitions
very similar to ours. Since they only considered two nodes, this is
related to maps to $\RR$ and the case of electrical networks, as
explained in their paper. A somewhat different notion of extremal
length was considered by Cannon, Floyd, and Perry
\cite{CFP94:SquaringRectangles}, where the contributions to the
``length'' part of ``extremal length'' come from passing through
vertices rather than over edges. Thus, even though they only consider
rectangles and cylinders (both related to maps to $\RR$), their
theories are not quite equivalent to electrical networks, although
they also get square tilings. Their
vertex-based extremal length is coarsely related to the edge-based
extremal length in the current paper, with constants depending on the
maximum valence.

Question~\ref{quest:looser} appears to be new.

The definition of embedding energy in Definition~\ref{def:embedding}
has not appeared in this form, although it is dual
to Jeremy
Kahn's notion of domination of weighted arc diagrams
\cite{Kahn06:BoundsI}. See
\cite[Prop.~1.17]{Thurston16:Characterize} for the
precise statement.

Theorem~\ref{thm:emb-sf} should be thought of as analogous to
Teichmüller's theorem, that in any homotopy class of homeomorphism
$[f] \co S_1 \to S_2$ between closed Riemann surfaces there are
``best'' metrics on~$S_1$ and~$S_2$ that reflect the minimal
quasi-conformal stretching. (The length graph~$K$ in the
statement of Theorem~\ref{thm:emb-sf} turns out to be $G_2$ with a
different metric.)

Most of the general graph energies in Section~\ref{sec:energies}
appear to be new, but the energy $E^1_p[c]$ is a power of
the $p$-modulus of the homotopy
class, which exists in much greater
generality~\cite{Fuglede57:ELFunctComplete}.

Many of the results of this paper were announced as part of an earlier
research report \cite{Thurston16:RubberBands}, which also contains
many related open problems.

\subsection{Acknowledgements}

This project grew out of joint work with Kevin Pilgrim, and owes a great
deal to conversations with him and with Jeremy Kahn.
There were many additional helpful conversations with
Maxime Fortier Bourque and
Steven Gortler.
The referees provided detailed and helpful feedback.
This work was partially supported by NSF grants DMS-1358638 and DMS-1507244.


\section{Basic notions and sub-multiplicativity}
\label{sec:basic}

As mentioned already, all graphs in this paper have a linear
structure on the edges,
metrics are assumed to be linear with respect to this structure,
weights are assumed to be piecewise-constant, and maps between graphs
are assumed to be PL. All the energies extend naturally to a wider
class of maps between graphs; however, the exact wider class of maps
depends on the energy, and it is more convenient to stick to PL
maps. The downside of working with PL maps is that the existence of
minimizers of the energies
is not obvious, since the space of PL maps is not locally compact in
any reasonable sense. We will prove existence of energy minimizers in
each case of interest.

We start by revisiting our definitions of graphs and energies.

\begin{definition}\label{def:regular}
If $f \co \Gamma_1 \to \Gamma_2$ is a PL map, the \emph{regular
  points} of~$f$ are the points in~$\Gamma_1$ that are in the interior
of a segment on which $f$ is linear with non-zero derivative, and the
\emph{singular points} are all
other points, namely vertices of~$\Gamma_1$, preimages of vertices
of~$\Gamma_2$, points where the derivative changes, and segments on
which~$f$ is constant. Similarly, the
\emph{singular values} in~$\Gamma_2$ are the images of singular
points, and the \emph{regular values} are the rest
of~$\Gamma_2$. There are only finitely many singular values. If
$y \in \Gamma_2$ is any point in the domain (regular or singular), a
\emph{regular neighborhood} of~$y$ is any neighborhood that has no
singular values, except possibly at~$y$.
\end{definition}

\begin{definition}\label{def:length}
  As in the introduction, a length graph~$K$ is a pair
  $(\Gamma, \ell)$ of a graph and an assignment of a length
  $\ell(e) > 0$ to edge~$e$ of~$\Gamma$; we call $\ell$ a
  \emph{metric} on~$\Gamma$. In a \emph{weak length graph}, we weaken
  the conditions to $\ell(e) \ge 0$. The space of all weak
  metrics on~$\Gamma$ is denoted $\Len(\Gamma)$, and the subspace of
  metrics is denoted $\Len^+(\Gamma)$. See
  Definition~\ref{def:weak-graph} for more on weak metrics.
  Given a PL map
  $f \co \Gamma_1 \to (\Gamma_2, \ell)$ to a length graph, we can
  pull-back the metric on~$\Gamma_2$ to a weak metric $f^*\ell$
  on~$\Gamma_1$, assigning to each edge $e \in \Edges(\Gamma_1)$ the
  total length traced out by $f(e)$. We also write $\ell(f(e))$ for
  $f^*\ell(e)$.
\end{definition}

\begin{definition}\label{def:weighted-graph}
  A \emph{weighted graph} is a graph~$W=(\Gamma,w)$ with a piecewise-constant
  weight function $w \co \Gamma \to \RR_{\ge 0}$. If $w$
  is constant on an edge~$e$, we write $w(e)$ for the weight
  on~$e$. The space of all weights on~$\Gamma$ that are constant on
  each edge is
  denoted $\Wgt(\Gamma)$, and $\Wgt^+(\Gamma)$ is the subspace where the
  weights are positive.
\end{definition}

\begin{definition}\label{def:multiplicity}
  For $W$ a weighted graph, $\Gamma$ a graph, and $c \co W \to
  \Gamma$ a PL map, the \emph{multiplicity function}
  $n_c \co \Gamma \to \RR$ is defined at regular values of~$c$ by
  \[
  n_c(y) \coloneqq \sum_{x \in c^{-1}(y)} w(x).
  \]
  This may also be written $n_c^w$ or $c_* w$ if we need to make the
  dependence on~$w$ explicit.
  We will never care about the value of $n_c$ at non-regular values.

  If in addition $\Gamma$ has more structure we have an energy.
  \begin{itemize}
  \item If $\Gamma$ is also a weighted graph with its own weights~$w$, the
    \emph{weight ratio} was defined in Equation~\eqref{eq:WR-def}; it
    is the $\ell^\infty$ norm of~$n_c$.
  \item If $\Gamma$ is an elastic graph with elastic
    constants~$\alpha$, the \emph{extremal length}
    of~$c$ is
    \begin{equation}\label{eq:el-weight-def}
    \EL(c) \coloneqq \int_{y \in \Gamma} n_c(y)^2\,dy,
    \end{equation}
    where we view each edge~$e$ of~$G$ as having a measure
    of total mass $\alpha(e)$. This generalizes
    Definition~\ref{def:el-curve}; $\EL(c)$ is the $\ell^2$ norm
    of~$n_c$, squared.
  \item If $\Gamma$ is a length graph with lengths~$\ell$, the
    \emph{weighted length}
    of~$c$ was defined in Equation~\eqref{eq:ell-def}; by
    Lemma~\ref{lem:length-defs} below, it is the $\ell^1$ norm of~$n_c$.
  \end{itemize}
\end{definition}

In Section~\ref{sec:taut} we will show that each homotopy class has a
\emph{taut map} that simultaneously minimizes $n_c(y)$
for all regular values~$y\in\Gamma$. Taut maps are automatically
minimizers of $\WR$, $\EL$, and~$\ell$, independent of the structure
on~$\Gamma$. If $c$ is taut, $n_c$ is constant on each edge
of~$\Gamma$ and
Equation~\eqref{eq:el-weight-def} reduces to
\begin{align*}
  \EL(c) &= \sum_{e \in \Edges(G)} n_c(e)^2\cdot \alpha(e)
\end{align*}

In addition to these energies, the energies $\Emb$, $\Dir$, and $\Lip$
were defined in Section~\ref{sec:intro}.
For all of these energies, if we wish to make explicit the dependence
of the energy on the geometric structure, we
use superscripts for the structure on the domain and subscripts for
the structure on the range. Thus we write $\WR^{w_1}_{w_2}$,
$\EL^w_\alpha$, $\ell^w_\ell$, $\Emb^{\alpha_1}_{\alpha_2}$,
$\Dir^\alpha_\ell$, or $\Lip^{\ell_1}_{\ell_2}$.

To work with these energies, we give some elementary lemmas, starting
with another formula for Dirichlet energy.

\begin{definition}
  For $G$ an elastic graph, $\Gamma$ either a
  length graph or an elastic graph, and $f \co G \to \Gamma$ a
  PL map, the \emph{filling function} $\Fill_f
  \co \Gamma \to \RR$ is defined at regular values of~$f$ by
  \begin{equation}
    \Fill_f(y) \coloneqq \sum_{x \in f^{-1}(y)} \abs{f'(x)}.
  \end{equation}
  In particular, Equation~\eqref{eq:emb} says that $\Emb(f) = \esssup_y
  \Fill_f(y)$.
\end{definition}

\begin{lemma}\label{lem:dirichlet-defs}
  For $f \co G
  \to K$ a PL map from an elastic graph to a length graph,
  \begin{equation}\label{eq:dirichlet-defs}
  \Dir(f) = \int_{x \in \Gamma} \abs{f'}^2\,dx = \int_{y \in K} \Fill_f(y)\,dy.
  \end{equation}
\end{lemma}


There is a corresponding formula for~$\ell$.
\begin{lemma}\label{lem:length-defs}
 For $c \co W \to K$ a PL map from a weighted graph to a
 length graph,
  \begin{equation}\label{eq:length-defs}
    \ell(c) = \int_{x \in W} w(x)\abs{c'(x)}\,dx = \int_{y \in K} n_c(y)\,dy.
  \end{equation}
\end{lemma}

\begin{proof}[Proof of Lemmas~\ref{lem:dirichlet-defs} and~\ref{lem:length-defs}]
  Change of variables.
\end{proof}

\begin{lemma}\label{lem:fill-ineq}
  Let $\shortseq{G_1}{\phi}{G_2}{\psi}{\Gamma}$ be a sequence of
  PL maps between graphs, where $G_1$
  and~$G_2$ are elastic graphs and $\Gamma$ is a length graph
  or an elastic graph. Then for almost every $z \in \Gamma$,
    \begin{equation*}
    \Fill_{\psi\circ\phi}(z) = \sum_{y\in\psi^{-1}(z)} \abs{\psi'(y)} \Fill_\phi(y)
      \le \Emb(\phi)\,\Fill_\psi(z).
  \end{equation*}
\end{lemma}

\begin{lemma}\label{lem:nf-W}
  Let $\shortseq{W_1}{\phi}{W_2}{c}{\Gamma}$ be a sequence of PL
  maps, with $W_1$ and $W_2$ weighted graphs. Then for almost every $z \in \Gamma$,
  \[
  n_{c \circ \phi}(z) \le \WR(\phi) n_c(z).
  \]
\end{lemma}

\begin{lemma}\label{lem:deriv-Lip}
  Let $\shortseq{\Gamma}{f}{K_1}{\phi}{K_2}$ be a sequence of PL
  maps, with $K_1$ and $K_2$ length graphs. Then for almost every
  $x \in \Gamma$,
  \[
  \bigl\lvert(\phi \circ f)'(x)\bigr\rvert \le \abs{f'(x)} \Lip(\phi).
  \]
\end{lemma}

\begin{proof}[Proof of Lemmas~\ref{lem:fill-ineq}, \ref{lem:nf-W}, and \ref{lem:deriv-Lip}]
  Immediate from the definitions.
\end{proof}

We now turn to sub-multiplicativity, Equations~\eqref{eq:submul}
and~\eqref{eq:submul-homotop}. For concreteness, we list the 10
cases individually.

\begin{proposition}\label{prop:sub-mult}
  The energies from~\eqref{eq:energies-graph} are sub-multiplicative, in the following
  sense. For $i \in \{1,2,3\}$, let $W_i$ be marked weighted graphs, $G_i$
  be marked elastic graphs, and $K_i$ be marked length graphs. Then,
  if we are given marked PL
  maps between these spaces as specified on each line, we have the
  given inequality.
  \begin{align}
    \label{eq:sm-CCC}
    \shortseq{W_1}{\phi}{\,&W_2}{\psi}{W_3}\co&
      \WR(\psi \circ \phi) &\le \WR(\phi) \WR(\psi)\\
    \label{eq:sm-CCG}
    \shortseq{W_1}{\phi}{\,&W_2}{c}{G}\co&
      \EL(c \circ \phi) &\le \WR(\phi)^2 \EL(c)\\
    \label{eq:sm-CCK}
    \shortseq{W_1}{\phi}{\,&W_2}{f}{K}\co&
      \ell(f \circ \phi) &\le \WR(\phi) \ell(f) \displaybreak[0]\\
    \label{eq:sm-CGG}
    \shortseq{W}{c}{\,&G_1}{\phi}{G_2}\co&
      \EL(\phi \circ c) &\le \EL(c) \Emb(\phi)\\
    \label{eq:sm-CGK}
    \shortseq{W}{c}{\,&G}{f}{K}\co&
      \ell(f \circ c)^2 &\le \EL(c) \Dir(f)\\
    \label{eq:sm-CKK}
    \shortseq{W}{c}{\,&K_1}{\phi}{K_2}\co&
      \ell(\phi \circ c) &\le \ell(c)\Lip(\phi) \displaybreak[0]\\
    \label{eq:sm-GGG}
    \shortseq{G_1}{\phi}{\,&G_2}{\psi}{G_3}\co&
      \Emb(\psi \circ \phi) &\le \Emb(\phi) \Emb(\psi)\\
    \label{eq:sm-GGK}
    \shortseq{G_1}{\phi}{\,&G_2}{f}{K}\co&
      \Dir(f \circ \phi) &\le \Emb(\phi) \Dir(f)\\
    \label{eq:sm-GKK}
    \shortseq{G}{f}{\,&K_1}{\phi}{K_2}\co&
      \Dir(\phi \circ f) &\le \Dir(f)\Lip(\phi)^2 \\
    \label{eq:sm-KKK}
    \shortseq{K_1}{\phi}{\,&K_2}{\psi}{K_3}\co&
      \Lip(\psi \circ \phi) &\le \Lip(\phi) \Lip(\psi).
  \end{align}
  (When there is only one graph of a particular type, we omit the subscript.)
  Each inequality still holds if we take homotopy classes
  on both sides.
\end{proposition}

\begin{proof}
  Equations~\eqref{eq:sm-CCC}, \eqref{eq:sm-CCG}, and
  \eqref{eq:sm-CCK} follow from Lemma~\ref{lem:nf-W} and
  Equations \eqref{eq:WR-def}, \eqref{eq:el-weight-def},
  and~\eqref{eq:length-defs}, respectively.
  Equations~\eqref{eq:sm-CKK}, \eqref{eq:sm-GKK},
  and~\eqref{eq:sm-KKK} follow from Lemma~\ref{lem:deriv-Lip} and
  Equations~\eqref{eq:length-defs}, \eqref{eq:dir},
  and~\eqref{eq:lip-def}, respectively.

  For Equation~\eqref{eq:sm-CGG}, we have
  \begin{align*}
    \EL(\phi \circ c)
      &= \int_{z \in G_2} n_{\phi \circ c}(y)^2\,dz\\
      &= \int_{z\in G_2} \biggl(\sum_{\phi(y)=z} n_c(y)\biggr)^2\,dz\\
      &\le \int_{z\in G_2} \biggl(\sum_{\phi(y)=z} n_c(y)^2/\abs{\phi'(y)}\biggr)
          \biggl(\sum_{\phi(y)=z} \abs{\phi'(y)}\biggr)\,dz\\
      &\le \Emb(\phi) \int_{y\in G_1} n_c(y)^2\,dy\\
      &= \Emb(\phi) \EL(c),
  \end{align*}
  using
  the definition of $\EL$;
  the equality $n_{\phi\circ c}(y) = \sum_{\phi(x)=y} n_c(x)$ at regular values;
  the Cauchy-Schwarz inequality;
  the definition of $\Emb(\phi)$ and a change of
    variables between $G_2$ and~$G_1$; and
  the definition of $\EL$ again.

  For Equation~\eqref{eq:sm-CGK}, we have
  \begin{align*}
    \ell(f \circ c)^2 &= \left(\int_{z \in K} n_{f \circ c}(z)\,dz\right)^2\\
      &= \left(\int_{y\in G} n_c(y)\abs{f'(y)}\,dy\right)^2\\
      &\le \left(\int_{y \in G} n_c(y)^2\,dy\right)
          \left(\int_{y\in G} \abs{f'(y)}^2\,dy\right)\\
      &= \EL(c)\Dir(f),
  \end{align*}
  using Lemma~\ref{lem:length-defs},
  a change of variables from~$K$ to~$G$, the Cauchy-Schwarz
  inequality, and the definitions of $\EL$ and $\Dir$.

  Equation~\eqref{eq:sm-GGG} follows from
  Lemma~\ref{lem:fill-ineq}.

  For Equation~\eqref{eq:sm-GGK}, we have
  \begin{align*}
    \Dir(f\circ\phi)
      &= \int_{z \in K} \Fill_{f\circ\phi}(z)\,dz\\
      &\le \int_{z \in K} \Fill_{f}(z)\Bigl(\max_{f(y)=z}
          \Fill_\phi(y)\Big)\,dz\\
       &\le \int_{z\in K} \Fill_f(z)\Emb(\phi)\,dz\\
       &= \Emb(\phi) \Dir(f),
  \end{align*}
  using Lemma~\ref{lem:dirichlet-defs},
  Lemma~\ref{lem:fill-ineq}, the definition of $\Emb(\phi)$, and
  Lemma~\ref{lem:dirichlet-defs} again.

  For each of these equations, to replace concrete functions by homotopy classes,
  take representatives $f,g$ of the two homotopy classes whose energy
  is within a factor of $\epsilon$ of the optimal value, in the sense
  that $E^p_q[f] \le
  E^p_q(f) \le E^p_q[f](1 + \epsilon)$, and similarly for~$g$. Then
  \[
  E^p_r[g \circ f] \le E^p_r(g \circ f) \le E^p_q(f)E^q_r(g) \le
    E^p_q[f]E^q_r[g](1+\epsilon)^2.
  \]
  Since $\epsilon$ can be chosen as small as desired, we are done.
\end{proof}

We now have one direction of Theorem~\ref{thm:emb-sf}.
\begin{corollary}\label{cor:sf-bound}
  For any homotopy class $[\phi] \co G_1 \to G_2$ of maps
  between marked elastic graphs, $\SF_{\Dir}[\phi] \le \Emb[\phi]$ and
  $\SF_{\EL}[\phi] \le \Emb[\phi]$.
\end{corollary}

\begin{proof}
  For any marked length graph~$K$ and
  homotopy class $[f] \co G_2 \to K$, by the homotopy version of
  Equation~\eqref{eq:sm-GGK} we have
  \[
  \frac{\Dir[f \circ \phi]}{\Dir[f]} \le
  \frac{\Dir[f]\Emb[\phi]}{\Dir[f]}  = \Emb[\phi].
  \]
  Since $K$ and~$[f]$ were arbitrary, it follows that
  $\SF_{\Dir}[\phi] \le \Emb[\phi]$. The argument that
  $\SF_{\EL}[\phi] \le \Emb[\phi]$ is exactly parallel.
\end{proof}


\section{Reduced and taut maps}
\label{sec:taut}

We next turn to notions of efficiency of maps between graphs. There is
a weak notion depending on no extra structure
(``reduced''), and then a more powerful notion that depends on a weight
structure (``taut'').

\subsection{Reduced maps}
\label{sec:reduced}

We work by analogy with the standard notion of a reduced (cyclic) word in a
group.

\begin{definition}\label{def:edge-reduced}
  A map $f \co \Gamma_1 \to \Gamma_2$ between marked graphs is
  \emph{edge-reduced} if, on each edge~$e$ of~$\Gamma_1$, $f$ is
  either constant or has a perturbation that is locally injective.
\end{definition}

\begin{definition}\label{def:direction}
  For a graph~$\Gamma$ and point $x\in \Gamma$, a \emph{direction}~$d$
  at~$x$ is a germ of a PL map from~$\RR_{\ge 0}$ to~$\Gamma$ starting
  at~$x$, considered up to PL reparametrization. Explicitly, there is
  a \emph{zero}
  direction, the germ of a constant map. Points on edges of~$\Gamma$
  have two non-zero directions, and vertices have as many non-zero
  directions as their valence. If $f \co \Gamma_1 \to \Gamma_2$ is a
  PL map and $d$ is a direction at $x \in \Gamma_1$, then $f(d)$ is a
  direction at~$f(x)$.
\end{definition}

\begin{definition}\label{def:reduced}
  A locally injective PL map $f \co \Gamma_1 \to \Gamma_2$ between marked graphs is
  \emph{reduced} if, at each unmarked vertex~$v$
  of~$\Gamma_1$, there are directions $d_1$ and~$d_2$ at~$v$ so that
  $f(d_1)$ and $f(d_2)$ are distinct and non-zero.

  More generally, pick $y \in \Gamma_2$ and let
  $Z \subset \Gamma_1$ be a connected component of $f^{-1}(y)$. A
  \emph{direction} from~$Z$ is a point $x \in Z$
  and a direction~$d$ from $x$ that points out of~$Z$ (so that $f(d)$
  is non-zero).
  Then $Z$ is a \emph{dead end} for~$f$ if
  $Z$ has no marked point and there is exactly one
  direction in
  \[
  D(Z) \coloneqq \{\,f(d)\mid \text{$d$ a direction from $Z$}\,\}.
  \]
  (If $D(Z) = \emptyset$, then $Z$ is an
  entire connected component of~$\Gamma_1$. We do not count this as a
  dead end.) We say that $f$ is
  \emph{reduced} if it has no dead ends. If $f$ is not reduced, there
  is a natural \emph{reduction} operation at a dead end~$Z$, where we
  modify~$f$ by pulling the image of~$Z$ in the direction $D(Z)$ until
  it hits a vertex of~$\Gamma_2$ or the
  boundary of a domain of linearity.
\end{definition}

\begin{proposition}\label{prop:reduced}
  If $f \co \Gamma_1 \to \Gamma_2$ is any map between marked graphs,
  then repeated reduction makes $f$ into a reduced map. In
  particular, 
  there is a reduced map in every homotopy class.
\end{proposition}

\begin{proof}
  Proceed by induction on the number of linear segments of~$f$ with
  non-zero derivative. This is reduced by each reduction operation.
\end{proof}

\begin{proposition}\label{prop:min-reduced}
  For $p,q \in \{1,2,\infty\}$ with $1 \le p \le q \le \infty$,
  reduction does not increase $E^p_q$, and strictly decreases $E^p_q$
  if $p < q$.
\end{proposition}

\begin{proof}
  Clear from the definitions.
\end{proof}

As a result of Propositions~\ref{prop:reduced}
and~\ref{prop:min-reduced}, when looking for energy minimizers we can
restrict our attention to reduced maps.

\subsection{Taut maps and flows}
\label{sec:tautness-definitions}

We now add some more structure, and correspondingly get
more restrictive conditions on energy\hyp minimizing maps. We will
consider maps
from weighted graphs, and in particular minimize $E^1_q$ for
any~$q$.

\begin{definition}\label{def:taut}
  Let $c \co W \to \Gamma$ be a PL map from a marked weighted
  graph~$W$ to a marked graph~$\Gamma$. We defined the multiplicity
  $n_c \co \Gamma \to \RR_{\ge 0}$ in
  Definition~\ref{def:multiplicity}.
  For a homotopy class~$[c]$ and $y$ in the interior of an edge of~$\Gamma$
  set
  \[
  n_{[c]}(y) \coloneqq \inf_{d \in [c]} n_d(y).
  \]
  The infimum is taken over maps~$d$ for which $n_d(y)$ is
  defined. Then it is easy to see that $n_{[c]}(y)$ depends only on the edge
  containing~$y$. We say that $c$ is \emph{taut} if
  $n_c = n_{[c]}$ almost everywhere. We say that $c$ is \emph{locally taut}
  if for every regular value $y \in \Gamma$ and regular
  neighborhood~$N$ of~$y$, the quantity~$n_c$
  cannot be reduced by homotopy of~$c$ on $c^{-1}(N)$.
\end{definition}

It will require work to show that taut maps exist, but if
they do exist they have good properties.

\begin{lemma}\label{lem:taut-reduced}
  A taut map from a positive weighted graph is reduced.
\end{lemma}

\begin{proof}
  Reduction strictly reduces $n_c$.
\end{proof}

\begin{proposition}\label{prop:taut-minimal}
  Let $c \co W \to \Gamma$ be a taut map from a marked weighted
  graph~$W$ to a
  marked graph~$\Gamma$. If $\Gamma$ is weighted, then $\WR(c)$ is
  minimal in $[c]$. If $\Gamma$ has an elastic structure, then
  $\EL(c)$ is minimal in $[c]$. If $\Gamma$ has a length
  structure, then $\ell(c)$ is minimal in $[c]$.
\end{proposition}

\begin{proof}
  Clear from the definitions, since the energies are monotonic
  in~$n_c$. When $\Gamma$ is a length graph, we use
  Lemma~\ref{lem:length-defs}.
\end{proof}

See also Proposition~\ref{prop:taut-balanced} below.

\begin{example}\label{examp:mincut}
  Let $W$ be a marked weighted graph with two marked vertices $s$
  and~$t$. Let $I$ be the interval $[-1,1]$ with the endpoints
  marked, and let $f \co W \to I$ be a map with $f(s) = -1$ and
  $f(t)=1$. Then $f$ is taut iff each edge is mapped monotonically
  and,
  for each regular value $y \in I$, the set of edges containing
  $f^{-1}(y)$ is a \emph{minimal cut-set} for~$W$, a set of edges of~$W$
  that separates $s$ from~$t$ and has minimal weight among all such
  sets.
\end{example}

The max-flow/min-cut theorem says that in the
context of Example~\ref{examp:mincut} the total weight of a minimal
cut-set
for~$W$ is equal to the maximum flow from $s$ to~$t$ through the edges
of~$W$. We will show that taut maps exist, and that they satisfy a
generalization of the max-flow/min-cut theorem; see
Theorem~\ref{thm:maxflow-mincut} and
Corollary~\ref{cor:maxflow-mincut-std} below. To give the strongest
statement, we make some definitions and preliminary lemmas first.

\begin{lemma}\label{lem:taut-curve}
  A  marked weighted multi-curve $c \co C \to \Gamma$ on a marked
  graph~$\Gamma$, with weights that are positive and constant on the
  components of~$C$, is taut iff it is reduced.
\end{lemma}

\begin{proof}
  One direction is Lemma~\ref{lem:taut-reduced}. For
  the other direction, if $c$ is reduced and non-constant, perturb it
  to a locally injective map
  and use the standard fact that, given
  a non-trivial free
  homotopy class of maps from $S^1$ to~$\Gamma$, a locally injective
  representative is unique up
  to reparametrization of the domain (which does not change~$n_c$).
  The same is true for homotopy
  classes of maps from $I$ to~$\Gamma$ relative to the endpoints.
\end{proof}

\begin{definition}\label{def:carries}
  If $c \co W_1 \to W_2$ is a PL map between weight graphs, we say that
  $W_2$ \emph{carries} $(W_1,c)$ if $\WR(c) \le 1$. If $W_2$ carries
  $(W_1,c)$, then we say that a point~$y$ of~$W_2$ is \emph{saturated}
  if $n_c(y) = w(y)$. Similarly, a subset of a weighted graph is saturated
  if almost every point in it is saturated.
\end{definition}

One notion of a ``flow'' on a weighted graph~$W$ is a
taut weighted multi-curve carried by~$W$. These multi-curves are a
little awkward to work with directly. As such, we will also define and
work with a more general type of flow from \emph{train tracks}.

\begin{definition}\label{def:triangle}
  A sequence of non-negative numbers $(x_i)_{i=1}^k$ is said to satisfy the
  \emph{triangle inequalities} if, for each $i$, $x_i$ is no larger than the
  sum of the remaining numbers:
  \begin{equation}\label{eq:triang-ineq}
  x_i \le \sum_{\substack{1 \le j \le k\\i \ne j}} x_j.
  \end{equation}
  This implies that $k \ne 1$, and if $k=2$ then $x_1 = x_2$.
  If $k \ge 3$ and one of these inequalities is an equality, then there
  is exactly one~$i$ so that $x_i$ is equal to the sum of
  the remainder.
  That $x_i$ is said to
  \emph{dominate} the rest.
\end{definition}

\begin{definition}\label{def:tt}
  A \emph{train-track structure}~$\tau$ on a graph~$\Gamma$ is, for
  each point~$x$ of~$\Gamma$, a partition
  of the non-zero directions from~$x$ into equivalence
  classes, called the \emph{gates} at~$x$, with at least two gates at
  each unmarked point. (This only gives additional structure at the
  vertices of~$\Gamma$.) A \emph{train track} is a graph with a
  train-track structure.

  If the graph~$\Gamma$ of a train track is weighted,
  the \emph{weight} of a gate is the sum of the weights of the edges
  corresponding to the directions
  that make it up. A set of weights on a train track is
  \emph{balanced} if, at each unmarked vertex, the weights of the gates
  satisfy the triangle inequalities.
  A \emph{balanced train track} is a train track together with
  balanced weights.
  If the weight of one gate~$g$ at a vertex of a balanced train track
  dominates the others,
  we can \emph{smooth} the vertex, changing the train-track structure
  so that there are only two gates, one with the directions from~$g$
  and one with all the
  other directions.

  A graph~$\Gamma$ with no unmarked univalent ends has a
  \emph{discrete} train-track structure $\delta_\Gamma$, in which two
  different directions are never equivalent. By default we use the
  discrete train-track structure on a graph. A (marked) weighted graph for
  which the discrete train-track structure satisfies the triangle
  inequalities (at unmarked vertices) is said to be \emph{balanced}.

  If $f \co (\Gamma_1,\tau_1) \to (\Gamma_2,\tau_2)$ is a map between
  train tracks that is locally injective on the edges
  of~$\Gamma_1$, we say that $f$ is a \emph{train-track map} if, for
  each vertex $v$ of~$\Gamma_1$ and directions $d_1$ and $d_2$ at~$v$,
  \begin{equation}\label{eq:tt-condition}
  \bigl(d_1 \sim_{\tau_1} d_2\bigr) \Longleftrightarrow
    \bigl(f(d_1) \sim_{\tau_2} f(d_2)\bigr).
  \end{equation}
  More generally,
  we say
  that $f$ is a train-track map if it has arbitrarily small
  perturbations $f_\epsilon$ so that $f_\epsilon$ is locally
  injective on the edges of~$\Gamma_1$ and satisfies
  Equation~\eqref{eq:tt-condition}. 
  In particular, a multi-curve $c \co C \to \Gamma$ is a train-track map iff
  at each point it passes
  through the incoming and outgoing directions are in different
  gates.

  If $f \co \Gamma_1 \to \Gamma_2$ is a locally injective, reduced
  map, the \emph{train track of~$f$} is
  the unique train-track structure~$\tau(f)$ on $\Gamma_1$ so that $f$
  is a train-track map with respect to the train track structures
  $\tau(f)$ on~$\Gamma_1$ and $\delta_{\Gamma_2}$
  on~$\Gamma_2$. Concretely, we have
  \[
  \bigl(d_1 \sim_{\tau(f)} d_2\bigr) \Longleftrightarrow
   \bigl(f(d_1) = f(d_2)\bigr).
  \]
\end{definition}

\begin{remark}
  Train tracks were first introduced in the context of surfaces and
  the theory of pseudo-Anosov
  diffeomorphisms~\cite{PH92:CombTrainTracks}, and a slight different
  notion of train track maps
  between graphs is used to study automorphisms of free
  groups~\cite{BH92:TrainTracks}.
\end{remark}

\begin{lemma}\label{lem:compose-tt}
  A composition of train-track maps is a train-track map.
\end{lemma}
\begin{proof}
  This is obvious if both maps are locally injective. In general,
  given train track maps $f$ and~$g$ with locally injective
  train-track perturbations $f_\epsilon$ and~$g_\epsilon$, the
  composition $g_\epsilon \circ f_\epsilon$ is a locally injective
  train-track perturbation of $g \circ f$.
\end{proof}
\begin{proposition}\label{prop:taut-balanced}
  Let $f \co W \to \Gamma$ be a taut map from a marked weighted graph.
  Then $n_f(\cdot)$ gives $\Gamma$ the structure of a balanced
  weighted graph.
\end{proposition}
\begin{proof}
  If $n_f$ (as a function on the edges of~$\Gamma$) is not balanced at
  an unmarked vertex~$v$, with one incident edge $e$ dominating the
  others, we may homotop $f$ to move the image of $f^{-1}(v)$
  along~$e$. This decreases $n_f$, contradicting the assumption that
  $f$ was taut.
\end{proof}

We can now state the main goal of this section.

\begin{theorem}\label{thm:maxflow-mincut}
  Let $f \co W \to \Gamma$ be a PL map from a marked weighted graph to
  a marked graph. Then there is a taut map in $[f]$. Furthermore, the
  following conditions are equivalent.
  \begin{enumerate}
  \item\label{item:mf-taut} The map~$f$ is taut.
  \item\label{item:mf-locally-taut} The map~$f$ is locally taut.
  \item\label{item:mf-tt} The graph~$W$ carries a marked balanced
    train track $t \co T \to W$, so that $f \circ t$ is
    a train-track map (with respect to the discrete train track
    structure on~$\Gamma$) and $n_{f \circ t} = n_t$. We may choose $(T,t)$
    so that $T$ is
    a subgraph of~$W$.
  \item\label{item:mf-curve} The graph~$W$ carries a marked weighted
    multi-curve $c \co C \to W$ so that $f \circ c$ is taut and
    $n_{f \circ c} = n_f$.
  \end{enumerate}
  Furthermore, if these conditions are satisfied, $t$
  in~(\ref{item:mf-tt}) and $c$ in~(\ref{item:mf-curve}) saturate
  every
  edge of~$W$ on which $f$ is not constant, and $c$ can be chosen to
  factor through~$t$: there is a multi-curve $c' \co C \to T$  so that $c = t \circ c'$.
\end{theorem}

As a preliminary step towards Theorem~\ref{thm:maxflow-mincut}, we
relate train tracks and multi-curves.

\begin{proposition}\label{prop:tt-curve}
  Any balanced train track $T = (\Gamma, w, \tau)$ carries a marked
  weighted multi-curve $(C,c)$ that saturates~$T$ and so that
  $c \co C \to T$ is a train-track map. We can choose
  $(C,c)$ so that each component of $C$ runs over each edge
  of~$\Gamma$ at most twice.
\end{proposition}

\begin{proof}
  Let $T' \subset T$ be the sub-train-track of edges of non-zero
  weight. Let $T''$ be $T'$ modified by
  smoothing all vertices in which one gate dominates the others. Let
  $\mathop{\mathrm{Yard}}(T'') \subset \Verts(T'')$ be the set of marked vertices of
  $T''$ with at least three gates. We will proceed by induction on
  $\abs{\Edges(T'')} + \abs{\mathop{\mathrm{Yard}}(T'')}$.

  In the base case, $T'$ and $T''$ are empty.

  Otherwise, choose any oriented edge
  $\vec e_0$ of $T''$, and find a path (forward and backwards)
  from~$\vec e_0$ within~$T''$, always making
  turns between different gates.
  Since there are at least two gates at each unmarked vertex, we can
  always
  find a successor edge for $\vec e_i$, unless we hit
  a marked vertex.
  Since there are only finitely many oriented edges in~$T$, we
  will eventually either find a path of edges between marked points or see a
  repeat and find a cyclic loop of edges.

  Consider the marked curve $(C_1,c_1)$ that runs
  over the cycle or path. For $\epsilon > 0$, let
  $w_\epsilon(e) \coloneqq w(e) - \epsilon n_{c_1}(e)$.
  Then for $\epsilon$ sufficiently small, $w_\epsilon$ gives a balanced
  train-track structure on~$T''$:
  \begin{itemize}
  \item On an edge~$e$, $w(e) > 0$ by construction of $T'$
    so $w_\epsilon(e) \ge 0$; and
  \item At a vertex~$v$, the construction of $T''$
    ensures that the triangle
    inequalities at~$v$ continue to hold in $w_{\epsilon}$.
  \end{itemize}
  Let $\epsilon_1$ be the maximum value of~$\epsilon$ so that
  $(\Gamma, w_{\epsilon_1},\tau)$  is a balanced train track. Let
  $w_1 = w_{\epsilon_1}$ and
  $T_1 = (\Gamma, w_1, \tau)$. As before, we construct derived train tracks
  $T_1'$ by deleting edges of weight~$0$ from~$T_1$ and $T_1''$ by smoothing
  dominating vertices in~$T_1'$. By the choice of $\epsilon_1$, there is
  either
  \begin{itemize}
  \item an edge~$e$ of $T''$ with $w_1(e) = 0$ but $w(e) \ne 0$, so
    $\abs{\Edge(T_1'')} < \abs{\Edge(T'')}$; or
  \item a vertex~$v$ of~$\Gamma$ with a gate~$g$ so that $w_1(g)$ dominates the
    other gates but $w(g)$ does not, so $\abs{\mathrm{Yard}(T_1'')} < \abs{\mathrm{Yard}(T'')}$.
  \end{itemize}
  In either case,
  by induction $T_1$ carries a marked weighted multi-curve $(C_2,c_2)$ that
  saturates~$T_1$. Then $(C_1 \sqcup C_2, c_1 \sqcup c_2, \epsilon_1
  \sqcup w_2)$ is
  the desired multi-curve from the statement.

  If, in the construction above, we make a cyclic loop as soon as we
  see a repeated oriented edge,
  the components of~$C$ run over each (unoriented) edge at
  most twice.
\end{proof}

\begin{lemma}\label{lem:carry-curve-taut}
  If $f \co W \to \Gamma$ is a map from a marked weighted graph to a
  marked graph and $(C,c)$ is a marked weighted multi-curve carried by $W$
  so that $f \circ c$ is taut and $n_{f \circ c} = n_f$, then $f$ is taut.
\end{lemma}

\begin{proof}
  If $g \co W \to \Gamma$ is any other map in $[f]$, then
  \[
  n_g \ge n_{g \circ c} \ge n_{[f \circ c]} = n_{f \circ c} = n_f,
  \]
  using the fact that $W$ carries~$c$ and Lemma~\ref{lem:nf-W}; the
  definition of $n_{[f \circ c]}$ as an infimum over the homotopy
  class; and the hypotheses.
\end{proof}

\begin{corollary}\label{cor:tt-taut}
  A train-track map from a marked balanced train
  track is taut.
\end{corollary}

\begin{proof}
  Let $f \co T \to \Gamma$ be the train-track map. Let $c \co C \to T$
  be the multi-curve from Proposition~\ref{prop:tt-curve}. Then, by
  Lemma~\ref{lem:compose-tt}, $f \circ c$ is taut and, since $c$
  saturates~$T$, $n_{f \circ c} = n_f$. Thus by
  Lemma~\ref{lem:carry-curve-taut}, $f$ is taut.
\end{proof}

\subsection{Local models}
\label{sec:local-models}

To prove Theorem~\ref{thm:maxflow-mincut}, we first analyze the
situation locally in a
regular neighborhood of a singular value. This reduces to
studying maps from a graph with $k$
marked points to a $k$-leg star graph.

Let $\Star_k$ be the star graph with $k$ legs, with a central vertex
$s_*$, marked endpoints $s_1,\dots,s_k$, and edges $[s_i,s_*]$. A
\emph{$k$-marked} graph is a graph with $k$ marked vertices
$(v_i)_{i=1}^k$ (in order).
There is a canonical homotopy class of marked maps from a $k$-marked
graph to $\Star_k$,
taking $v_i$ to $s_i$. We prove an analogue of
Theorem~\ref{thm:maxflow-mincut} in this context.

\begin{proposition}\label{prop:maxflow-mincut-local}
  Let $W$ be a $k$-marked weighted graph. Then there
  is a taut map in the canonical homotopy class of maps to
  $\Star_k$. Furthermore, the
  following conditions are equivalent.
  \begin{enumerate}
  \item\label{item:local-taut} The map~$f$ is taut.
  \item\label{item:local-tt} The graph~$W$ carries a marked balanced
    train track $(T,t)$ so
    that $f \circ t$ is a train-track map and $n_{f \circ t} = n_f$.
  \item\label{item:local-curve} The graph~$W$ carries a marked
    weighted multi-curve $(C,c)$ so that
    $f \circ c$ is taut and $n_{f \circ c} = n_f$.
  \end{enumerate}
\end{proposition}

\begin{definition}
  A \emph{cut}~$S$ of a graph~$\Gamma$ is a partition of the vertices of the graph
  into two disjoint subsets: $\Verts(\Gamma) = S \sqcup \overline{S}$. The
  corresponding \emph{cut-set} $c(S) = c(\overline{S})$ is the set of edges that
  connect $S$ to~$\overline{S}$. If $\Gamma$ has weights~$w$, the \emph{weight} of the
  cut is $w(S) \coloneqq \sum_{e \in c(S)} w(e)$.

  Two cuts $S_1$ and~$S_2$
  are \emph{nested} if they are disjoint or one is contained in the
  other : $S_1 \cap S_2 = \emptyset$, $S_1 \subset S_2$, or
  $S_2 \subset S_1$.  The definition is unchanged if we replace an
  $S_i$ with~$\overline{S_i}$. A set of cuts is nested if each pair is nested.

  Let $W$ be a $k$-marked weighted graph. A
  \emph{$v_i$-cut} is a subset $S_i \subset \Verts(G)$ so that
  $S_i\cap \{v_1,\dots,v_k\} = \{v_i\}$. A \emph{vertex cut} is
  a $v_i$-cut
  for some~$i$. A
  \emph{minimal} $v_i$-cut is one with minimal weight. Let
  $\mincut_i(w)$ be the weight of a
  minimal $v_i$-cut.
\end{definition}

\begin{lemma}\label{lem:submodular}
  If $S_1$ and~$S_2$ are
  two cuts on the same weighted graph, then
  $w(S_1) + w(S_2) \ge w(S_1 \cap S_2) + w(S_1 \cup S_2)$.
\end{lemma}

\begin{proof}
  This is the standard sub-modular property of cuts, and is easy to prove.
\end{proof}

\begin{lemma}\label{lem:nest-one-end}
  If $W$ is a $k$-marked weighted graph and
  $S_i$ and $S_i'$ are two minimal $v_i$-cuts, then $S_i \cap S_i'$
  and $S_i \cup S_i'$ are also minimal $v_i$-cuts.
\end{lemma}

\begin{proof}
  The sets $S_i \cap S_i'$ and $S_i \cup S_i'$ are $v_i$-cuts, so by
  minimality they both have weight at least as large as
  $\mincut_i(w) = w(S_i) = w(S_i')$.  Lemma~\ref{lem:submodular} gives
  the other inequality.
\end{proof}

\begin{lemma}\label{lem:nest-two-ends}
  If $W$ is a $k$-marked weighted graph,
  $S_i$ is a minimal $v_i$-cut, and $S_j$ is a minimal $v_j$ cut for
  $j \ne i$, then $S_i \setminus S_j$ is also a minimal $v_i$-cut and $S_j
  \setminus S_i$ is also a minimal $v_j$-cut.
\end{lemma}

\begin{proof}
  $S_i \setminus S_j$ is a $v_i$-cut and $S_j \setminus S_i$ is a
  $v_j$-cut. By minimality, their weights are at least as large as
  $\mincut_i(w)$ and $\mincut_j(w)$, respectively. Applying
  Lemma~\ref{lem:submodular} to $S_i$ and $\overline{S_j}$
  gives the other inequalities.
\end{proof}

\begin{definition}
  If $W$ is a $k$-marked weighted graph,
  we say that an edge of~$W$ is \emph{slack} if it has non-zero
  weight and is not contained in any minimal vertex cut.
\end{definition}

See
Figure~\ref{fig:mincut} for the next two lemmas.

\begin{lemma}\label{lem:minimal-weights}
  If $W = (\Gamma, w)$ is a weighted graph with $k$ marked vertices,
  then there is a set of weights $w_0 \le w$ on~$\Gamma$ so that
  $W_0 = (\Gamma, w_0)$ has no slack edges and so that for all~$i$,
  $\mincut_i(w) = \mincut_i(w_0)$.
\end{lemma}

\begin{proof}
  We proceed by induction on the number of slack edges. If there is
  one, pick any slack edge~$e_0$ of~$W$. For $0 \le k \le w(e_0)$,
  define a modified set of weights by
  \[
  w\{e_0\mapsto k\}(e) \coloneqq
  \begin{cases}
    k & e = e_0 \\
    w(e) &\text{otherwise}.
  \end{cases}
  \]
  Let $k_0$ be minimal value so that, for all~$i$,
  $\mincut_i(w\{\hbox{$e_0\mapsto k_0$}\}) = \mincut_i(w)$. Then $k <
  w(e_0)$ and $e$ is not slack with respect to
  $w\{e_0 \mapsto k_0\}$. By induction we can find
  weights~$w_0 \le w\{e_0 \mapsto k_0\} \le w$ on~$\Gamma$ so that
  $(\Gamma, w_0)$ has no slack edges.
\end{proof}

The weights~$w_0$ in Lemma~\ref{lem:minimal-weights} is usually not unique.

\begin{lemma}\label{lem:complete-nested}
  Let $W = (\Gamma, w)$ be a $k$-marked weighted graph with no slack edges
  and let $\cS \subset \cP(\cP(\Verts(W)))$ be a nested set of minimal
  vertex cuts. Then there is a nested set of minimal vertex cuts $\cT
  \supset \cS$ so that every edge of~$W$ with non-zero weight is
  contained in $c(S)$ for some $S \in \cT$.
\end{lemma}

Here $\cP(X)$ is the power set of~$X$.

\begin{figure}
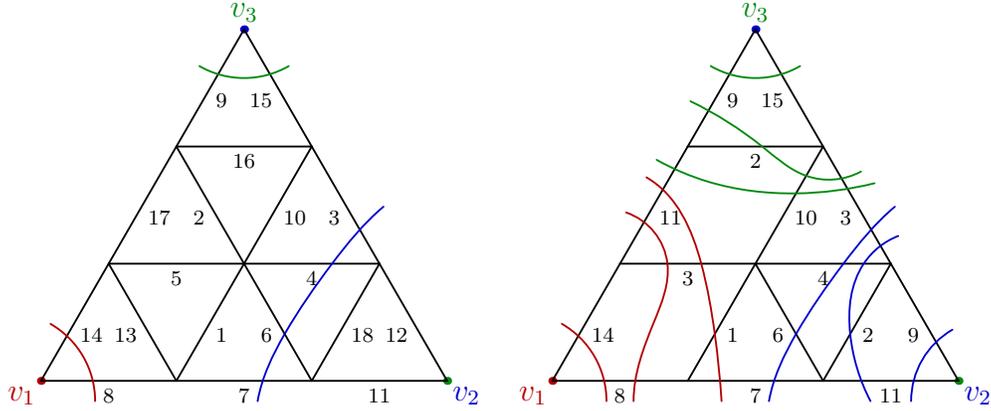

  \[
  \mfigb{cuts-1}
  \quad\mfigb{cuts-2}
  \]
  \caption{Finding minimum cuts and maximum flows near a vertex. Left:
    a weighted graph~$W$ with three marked vertices, with the weights
    indicated. Minimal cuts that separate the vertices are
    marked; these are used to construct a taut map from~$W$ to
    $\Star_3$. Right: a weighted graph~$W_0$ carried by~$W$, with no
    slack edges and the same $\mincut_i$ for each~$i$. The minimal
    cuts on~$W$
    have been extended to a complete nested family of minimal
    cuts. These are used to construct a train track carried by~$W$.
  }
  \label{fig:mincut}
\end{figure}

\begin{proof}
    We proceed by induction on the size of
    $\Edges(W) \setminus \bigcup_{S \in \cS} c(S)$. If there is at
    least one edge~$e_1$ in this set with non-zero weight, it is
    contained in some minimal vertex cut-set $c(S_1)$ since $W$ has no
    slack edges. Then, by repeatedly applying
    Lemmas~\ref{lem:nest-one-end} and~\ref{lem:nest-two-ends}, we can
    find another minimal vertex cut $S_1'$ so that $e \in c(S_1')$ and
    $S_1'$ is
    nested with respect to all $S \in \cS$. By
    induction, $\cS \cup \{S_1'\}$ can be completed to the desired
    set~$\cT$.
\end{proof}

\begin{proof}[Proof of Proposition~\ref{prop:maxflow-mincut-local}]
  We first prove the existence of a taut map.
  For each $i$, pick a minimal $v_i$-cut~$S_i$, and let $S_i' = S_i
  \setminus \bigcup_{j \ne i} S_j$. By
  Lemma~\ref{lem:nest-two-ends}, $\{S_i'\}$ is a nested collection
  of minimal $v_i$-cuts. Define $f \co W \to \Star_k$ by
  mapping all vertices in $S_i'$ to~$s_i$, all vertices not in any
  $S_i'$ to $s_*$, and all edges to
  reduced paths connecting their endpoints. (Thus $f$ is
  constant on any edge
  not in any $c(S_i')$.) Then
  $f$ is taut
  since the $S_i'$ are minimal.

  Proposition~\ref{prop:tt-curve} and Lemma~\ref{lem:taut-curve} tell
  us that (\ref{item:local-tt})
  implies~(\ref{item:local-curve}), and
  Lemma~\ref{lem:carry-curve-taut} tells us that
  (\ref{item:local-curve}) implies (\ref{item:local-taut}), so it
  remains to prove that~(\ref{item:local-taut})
  implies~(\ref{item:local-tt}).

  Suppose that $f$ is taut. For each regular value
  $y\in [s_*, s_i]\subset \Star_k$, consider the $v_i$-cut
  $S_y = \{v \in \Verts(\Gamma) \mid f(v) \in [s_i,y]\}$. Then $S_y$
  is a minimal vertex cut, since $f$ is taut.
  Furthermore, if $z$ is
  another regular value of~$f$ on any edge of $\Star_k$, then $S_y$
  and $S_z$ are nested. Let
  $\cS$ be the set of all $S_y$ for all regular~$y$, together with the
  singleton endpoints $\{v_i\}_{i=1}^k$. This is
  a nested set of minimal cuts on~$W$.

  Use Lemma~\ref{lem:minimal-weights} to
  find weights $w_0 \le w$ so that $(\Gamma, w_0)$ has no slack edges.
  Let $\Gamma_0 \subset \Gamma$ consist of the edges~$e$ with
  $w_0(e) \ne 0$, let $W_0 = (\Gamma_0,w_0)$, and let $t \co W_0
  \hookrightarrow W$ be the inclusion.
  By
  Lemma~\ref{lem:complete-nested}, $\cS$ can be completed to a nested
  system of cuts~$\cT$ on~$W_0$ so that every edge is in at least
  one cut-set. For one example of $W_0$ and~$\cT$, see the right of
  Figure~\ref{fig:mincut}.

  For each $i$, let the distinct
  $v_i$-cuts in $\cT$ be
  \[
  \{v_i\} = S_{i,0} \subsetneq S_{i,1} \subsetneq \dots \subsetneq S_{i,n_i}.
  \]
  For each $i$ and for $0 \le j \le n_i$, pick distinct points
  $x_{i,j} \in [s_i,s_*) \subset \Star_k$ with $x_{i,0} = s_i$ and with the $x_{i,j}$ in
  order. Then define a map $g \co W_0
  \to \Star_k$ on vertices by
  \[
  g(v) =
  \begin{cases}
    s_i & v = v_i\\
    x_{i,j} & v \in S_{i,j} \setminus S_{i,j-1}\\
    s_* & v \in \Verts(W_0) \setminus \bigcup_{S \in \cT} S.
  \end{cases}
  \]
  (These cases are exclusive since $\cT$ is nested.)
  On an edge~$e$ of~$W_0$, define $g(e)$ to be a reduced path
  connecting its
  endpoints. Since $\cT$ is a complete set of cuts, $g$ is not
  constant on any edge. Let $\tau_0 = \tau(g)$. Then $(W_0,\tau_0)$ is a
  balanced train track: the triangle
  inequalities follow since
  the $S_{i,j}$ are minimal cuts.

  By an appropriate
  choice of the $x_{i,j}$, the train-track map~$g$ can be made to be
  an arbitrarily small
  perturbation of~$f \circ t$, so $f \circ t$ is also a train-track
  map. If $e$ is an edge on which the original map~$f$ is not
  constant, then $e$ is not slack and $w(e) = w_0(e)$,
  so $n_{f \circ t} = n_f$ as desired.
\end{proof}

\subsection{General flows}
\label{sec:general-flows}

We now use these local models to prove Theorem~\ref{thm:maxflow-mincut}.

\begin{definition}\label{def:local-model}
  Let $f \co W \to \Gamma$ be a PL map from a marked weighted
  graph~$W$ to a marked graph~$\Gamma$. Pick $y \in \Gamma$ and let $N \subset \Gamma$
  be a closed regular neighborhood of~$y$. Then the
  \emph{local model for $f$ at~$y$} is the map
  \[
  f^{-1}(N)/\mathord{\sim} \longrightarrow N,
  \]
  where
  \begin{itemize}
  \item $N$ is considered as a marked graph (equivalent to $\Star_k$)
    with marked points at $\bdy N$;
  \item $\sim$ is the equivalence relation that identifies two points
    in $f^{-1}(\bdy N)$ if they map to the same point in $\bdy N$; and
  \item $f^{-1}(N)/\mathord{\sim}$ is considered as a $k$-marked weighted
      graph, with weights inherited from~$W$.
  \end{itemize}
\end{definition}

\begin{lemma}\label{lem:exist-local-taut}
  Let $[f] \co W \to \Gamma$ be a homotopy class of maps from a
  marked weighted graph~$W$ to a marked graph~$\Gamma$. Then there is
  a locally taut map in~$[f]$.
\end{lemma}

\begin{proof}
  Suppose~$f$ is any PL representative for its homotopy class. We can
  modify~$f$ to send vertices to vertices without increasing
  $n_f$, as follows. For each edge~$e$ of~$\Gamma$, look at the
  division of $e$ into intervals according to the value of~$n_f$. Pick
  one of these intervals on which $n_f$ is minimal (among the values
  of $n_f$ that appear), and spread out this interval by a homotopy
  until it covers~$e$. This gives us an initial map~$f_0$.

  If $f_0$ is not locally taut, then there is some vertex~$v$
  of~$\Gamma$ so that the local model of~$f_0$ at~$v$ is not
  taut. By Proposition~\ref{prop:maxflow-mincut-local}, there is a taut map in
  the homotopy class of the local model. Let $f_0'$ be $f_0$ with the
  map replaced by its taut model near~$v$. There is at
  least one point~$x$ near~$v$, regular for both maps, with
  $n_{f_0'}(x) < n_{f_0}(x)$. Homotop
  $f_0'$ as above, spreading out segments of minimal multiplicity,
  to construct a map $f_1$ that sends vertices to vertices with $n_{f_1} \le n_{f_0}$
  everywhere and $n_{f_1}(x) < n_{f_0}(x)$.

  If $f_1$ is not locally taut, repeat the process, with $f_1$ in
  place of $f_0$. Our
  initial representative $f_0$ gives an upper bound on $n_{f_i}(e)$ for
  all~$e$. There are only finitely many
  linear combinations of the non-zero weights of edges of~$W$ to get a
  value less
  than this bound. At each step $n_f$
  strictly decreases on at least one edge, so the process terminates
  in finite time.
\end{proof}

\begin{proof}[Proof of Theorem~\ref{thm:maxflow-mincut}]
  We first prove the equivalence of the four conditions on taut
  maps. (\ref{item:mf-taut}) implies~(\ref{item:mf-locally-taut}) is
  obvious.

  To see that (\ref{item:mf-locally-taut}) implies~(\ref{item:mf-tt}),
  suppose $f$ is locally taut.
  Then for each singular value $y \in \Gamma$, the local model for $f$
  at~$y$ carries a balanced train track compatible with~$f$ by
  Proposition~\ref{prop:maxflow-mincut-local}. Define a
  balanced train track~$T$ and map $t \co T \to W$ by
  assembling these local models following the pattern of~$W$,
  leaving $W$ unchanged outside of the local models. Then $(T,t)$ is
  the desired train track carried by~$W$, with $T$ a subgraph of~$W$.

  Proposition~\ref{prop:tt-curve} shows that (\ref{item:mf-tt})
  implies~(\ref{item:mf-curve}).

  Lemma~\ref{lem:carry-curve-taut} shows that (\ref{item:mf-curve})
  implies~(\ref{item:mf-taut}).

  Now if $[f]$ is any homotopy class, by
  Lemma~\ref{lem:exist-local-taut} there is a locally taut element
  of~$[f]$, which by the above equivalences is also globally taut.

  The statements about saturation follow immediately, and the
  multi-curve~$c$ factors through~$t$.
\end{proof}

\subsection{Connection to max-flow/min-cut}
\label{sec:conn-maxflow-mincut}
For some intuition,
we can connect Theorem~\ref{thm:maxflow-mincut} and
Proposition~\ref{prop:maxflow-mincut-local} to a statement that looks
more like the classical max-flow/min-cut problem
(Example~\ref{examp:mincut}). For simplicity, we
restrict to the local setting of
Proposition~\ref{prop:maxflow-mincut-local}. Given a graph~$\Gamma$ with
$k$ marked points $\{v_i\}_{i=1}^k$, define a \emph{flow} between
$v_i$ and~$v_j$ to be a marked weighted multi-curve $(C,c)$ on~$\Gamma$, with
each component a marked interval with one endpoint mapping to~$v_i$
and the other
mapping to~$v_j$. Such a flow has multiplicities $n_c$ on each edge as
usual, as well as a total weight $w(c)$, the sum of the weights of all
components of~$C$.

\begin{corollary}\label{cor:maxflow-mincut-std}
  Let $W$ be a weighted graph with $k$ marked points
  $\{v_i\}_{i=1}^k$. Then we can find
  \begin{itemize}
  \item for $1 \le i \le k$, a $v_i$-cut $S_i$, and
  \item for $1 \le i < j \le k$, a flow $c_{ij} = c_{ji}$ between $v_i$
    and~$v_j$
  \end{itemize}
  so that
  \begin{itemize}
  \item the collection of all flows $c_{ij}$ is carried by~$W$: for
    each edge~$e$ of~$\Gamma$,
    \begin{equation}\label{eq:flow-carry}
    \sum_{i < j} n_{c_{ij}}(e) \le w(e),
    \end{equation}
  \item the total flow into a vertex equals the weight of the cut: for
    each $i$,
    \begin{equation}\label{eq:flow-cut}
    \sum_{j \ne i} w(c_{ij}) = w(S_i).
    \end{equation}
  \end{itemize}
  In this situation, for each~$i$, $w(S_i)$ is minimal and
  $\sum_{j \ne i} w(c_{ij})$ is maximal.
\end{corollary}

\begin{proof}
  Let $f$ be a taut map in the canonical homotopy class of maps
  from~$W$ to $\Star_k$, as given by
  Proposition~\ref{prop:maxflow-mincut-local}. By that proposition,
  $W$ carries a marked weighted multi-curve $(C,c)$ so that $f \circ c$ is
  taut and $n_{f\circ c} = n_{f}$. The only non-trivial homotopy
  classes of marked multi-curves on $\Star_k$ are given by paths connecting
  $s_i$ to $s_j$ for some $i \ne j$. Let the flow $c_{ij}$ be the flow
  given by those components of $(C,c)$ that connect $v_i$
  to~$v_j$. Then Equation~\eqref{eq:flow-carry} says that $W$ carries
  $(C,c)$.

  On the other hand, for each~$i$, let $y_i$ be a regular value on
  $[s_*,s_i]$, and let $S_i$ be $S_{y_i}$ as defined in the proof of
  Proposition~\ref{prop:maxflow-mincut-local}. Then the equality
  $n_{f \circ c}(y) = n_f(y)$ is equivalent
  to Equation~\eqref{eq:flow-cut}. Minimality of $w(S_i)$ and
  maximality of $\sum_{j \ne i} w(c_{ij})$ both follow from tautness,
  and indeed are easy to deduce from Equation~\eqref{eq:flow-cut}.
\end{proof}

\subsection{Strongly reduced maps}
\label{sec:strong-red}

For later finiteness statements, we introduce a stronger notion of
``reduced'' for maps between
general marked graphs without a weight structure.

\begin{definition}\label{def:strongly-reduced}
  A map $f \co \Gamma_1 \to \Gamma_2$ between marked graphs is
  \emph{strongly reduced} if it is taut for some choice of positive
  weights on~$\Gamma_1$. See Figure~\ref{fig:reduced} for a non-example.
\end{definition}

\begin{figure}
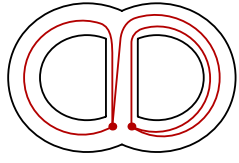

  \[
  \mfigb{reduced-3}
  \]
  \caption{A map from an eyeglass graph to a theta graph that is
    reduced but not strongly reduced.}
  \label{fig:reduced}
\end{figure}

\begin{lemma}\label{lem:strong-red-curve}
  Let $f \co \Gamma_1 \to \Gamma_2$ be an edge-reduced map. Then
  $f$ is strongly reduced iff for each
  edge~$e$ of~$\Gamma_1$, there is a
  reduced curve $(C_e,c_e)$ on~$\Gamma_1$ that
  runs over~$e$ and so that $f \circ c$ is also
  reduced. Furthermore, the curves can be chosen to run
  over each edge of~$\Gamma_1$ at most twice.
\end{lemma}

\begin{proof}
  If $f$ is strongly reduced, then let $w \in \Wgt^+(\Gamma_1)$ be a set
  of weights for which $f$ is taut. Then
  the connected components of the multi-curve from
  condition~(\ref{item:mf-curve}) of Theorem~\ref{thm:maxflow-mincut}
  is the desired curve for each edge, which by
  Proposition~\ref{prop:tt-curve} may be chosen to run
  over each edge at most twice.

  Conversely, if there are curves $(C_e,c_e)$ as in the statement, let
  $W = \bigsqcup_e C_e$ be their union with the evident map
  $c \co W \to \Gamma_1$ and with constant weight~$\mathbf{1}$ on each
  component. Then $f$ is taut with respect to $n_c(\mathbf{1})$.
\end{proof}

\begin{definition}\label{def:comb-type}
  Given an edge-reduced PL map $f \co \Gamma_1 \to \Gamma_2$ between
  marked graphs, the \emph{combinatorial type} of~$f$ consists of
  the following discrete data.
  \begin{enumerate}
  \item For each vertex~$v$ of~$\Gamma_1$, record whether $f(v)$ is on a
    vertex or in the interior of an edge of~$\Gamma_2$, and which vertex
    or edge it lies on.
  \item For each oriented edge~$\vec e$ of~$\Gamma_1$, record the
    \emph{edge-path} of $f(\vec e)$, the reduced
    sequence of oriented edges of~$\Gamma_2$ that $f(\vec e)$ travels
    over. This edge-path may start and/or end with a partial edge,
    depending on whether the endpoints of $\vec e$ map to vertices or
    edges. There are degenerate cases when $\vec e$ maps to a single
    vertex or stays within the interior of a single edge.
  \end{enumerate}
\end{definition}
The combinatorial type of~$f$ determines the homotopy class
of~$[f]$.

\begin{proposition}\label{prop:strong-red-finite}
  Let $[f] \co \Gamma_1 \to \Gamma_2$ be a homotopy class of maps
  between marked graphs.  Then there are only finitely many
  combinatorial types of strongly reduced maps
  in~$[f]$.
\end{proposition}

\begin{proof}
  Let $g \in [f]$ be strongly reduced, and let~$\vec e$ be an oriented
  edge
  of~$\Gamma_1$. We must show that there is a finite set of possible
  edge-paths for $f(\vec e)$. For this, let $(C,c)$ be the curve
  running over~$e$ as in
  Lemma~\ref{lem:strong-red-curve}. Since
  $(C,c)$ runs over each edge of~$\Gamma_1$ at most twice, there are
  only finitely many possibilities for it. Furthermore, $[f \circ c]$
  is determined by $[f]$ and $[c]$, so there are only finitely many
  edge-paths in~$\Gamma_2$ arising from $[f \circ c]$. The edge-path
  of~$f(\vec e)$
  must be a sub-path of the edge-path of $f\circ c$, and again there
  are only finitely many possibilities.
\end{proof}

\begin{remark}
  Proposition~\ref{prop:strong-red-finite} is false if we look at
  reduced maps rather than strongly reduced maps. (The map $f$ can be
  ``spun'' around a taut cycle of~$\Gamma_1$, as in Figure~\ref{fig:reduced}.)
\end{remark}

\subsection{Restricting the domain and range}
\label{sec:taut-restricting}

\begin{lemma}\label{lem:taut-subdomain}
  Let $f \co W \to \Gamma$ be a map from a marked weighted graph to a
  marked graph. Let $W_0 \subset W$ be the closure of the subset of~$W$ where
  $w(x) \ne 0$. Then $f$ is taut iff the
  restriction $f|_{W_0} \co W_0 \to \Gamma$ is taut.
\end{lemma}

\begin{proof}
  Segments of weight~$0$ do not effect $n_f$.
\end{proof}

\begin{definition}
  A homotopy class $[f] \co \Gamma_1 \to \Gamma_2$ of marked maps
  is \emph{essentially surjective} if every element of~$[f]$ is
  surjective.
\end{definition}

\begin{remark}
  If $\Gamma_1$ and $\Gamma_2$ are connected with no 1-valent vertices and
  $[\phi] \co \Gamma_1 \to \Gamma_2$ is $\pi_1$-surjective, then it is
  essentially surjective.
\end{remark}

If~$f$ is not (essentially) surjective, then tautness of~$f$ is
equivalent to tautness of the map with restricted range.

\begin{lemma}\label{lem:taut-subrange}
  Let $f \co W \to \Gamma$ be a map from a marked weighted graph to a
  marked graph. Suppose the image of~$f$ is contained in a subgraph
  $\Gamma_0 \subset \Gamma$. Let $\Gamma^*$ be
  $\Gamma$ with all edges not in~$\Gamma_0$ collapsed, with
  $\kappa\co \Gamma \to \Gamma^*$ the collapsing map. Then the
  following are equivalent:
  \begin{enumerate}
  \item $f$ is taut,\label{item:subrange-taut}
  \item the map with restricted range $f \co W \to \Gamma_0$ is
    taut,\label{item:subrange-restr} and
  \item the map $\kappa \circ f \co W \to \Gamma^*$ is taut.\label{item:subrange-collapse}
  \end{enumerate}
\end{lemma}

\begin{proof}
  Recall from Theorem~\ref{thm:maxflow-mincut} that being taut is
  equivalent to being locally taut. If you replace ``taut'' with
  ``locally taut'' in conditions
  \eqref{item:subrange-taut}--\eqref{item:subrange-collapse} above,
  they are easily
  seen to be equivalent.
\end{proof}

\subsection{Continuity}
\label{sec:taut-continuity}

Let $[f] \co (\Gamma_1,w) \to \Gamma_2$ be a
homotopy class of maps between marked graphs. We are interested in how
$n^w_{[f]}$ varies as $w$ varies.
Define $N_{[f]} \co \Wgt(\Gamma_1) \to \Wgt(\Gamma_2)$ by
$N_{[f]}(w) \coloneqq n_{[f]}^w$. (There is usually no
single~$g \in [f]$ that is taut for all~$w$.)

\begin{proposition}\label{prop:nf-cont-pl}
  Let $[f] \co \Gamma_1 \to \Gamma_2$ be a homotopy class of maps
  between marked graphs. Then
  $N_{[f]} \co \Wgt(\Gamma_1) \to \Wgt(\Gamma_2)$ is continuous and piecewise-linear.
\end{proposition}

\begin{proof}
  Pick $w_0 \in \Wgt(\Gamma_1)$ and $e_0 \in
  \Edges(\Gamma_2)$, and let $K=n_{[f]}^{w_0}(e_0)$.
  Suppose that $w \in \Wgt(\Gamma_1)$ is some set of weights with
  $\abs{w(e) - w_0(e)} < \epsilon$ for all~$e$. To show continuity of
  $N_{[f]}$
  near~$w_0$, we will give upper and lower bounds on $n^w_{[f]}(e_0)$
  in terms of~$K$ and~$\epsilon$.

  To get an upper bound, let $f_0 \in [f]$
  be a taut map from $(\Gamma_1,w_0)$ to $\Gamma_2$. Pick a regular
  value $y \in e_0$ of~$f_0$, and let $k = \abs{f_0^{-1}(y)}$.
  Then since $f_0 \in [f]$,
  \[
  n_{[f]}^w(e_0) \le n_{f_0}^w(y) \le n_{f_0}^{w_0}(y) + k\epsilon
    = K + k\epsilon,
  \]
  as desired.

  To get a lower bound, let
  $E_1 = \{\,e\in\Edges(\Gamma_1) \mid w_0(e) \ne 0\,\}$ and $\nu
  = \min_{e \in E_1} w_0(e)$. Make sure
  $\epsilon < \nu$, and let
  $M = K/(\nu-\epsilon)$. Let $g \co
  (\Gamma_1,w) \to \Gamma_2$ be a taut map in $[f]$. Now
  $n_g^w(e_0)$ can be written as an integer linear combination of
  weights of edges:
  \[
  n_g^w(e_0) = \sum_{e \in \Edges(\Gamma_1)} a_e w(e).
  \]
  If $n_g^w(e_0) \ge K$, we are done. Otherwise, we must have
  $\sum_{e\in E_1} a_e \le M$, and so
  \begin{align*}
    n_g^w(e_0)
      &\ge \sum_{e \in E_1} a_e w(e) \\
      &\ge -M\epsilon + \sum_{e \in \Edges(\Gamma_1)} a_e w_0(e)\\
      &\ge -M\epsilon + K.
  \end{align*}
  The last inequality uses $n_g^{w_0}(e_0) \ge K$, which
  follows from the definition of~$K$.

  To show that $N_{[f]}$ is piecewise-linear, by continuity we may
  restrict attention to
  positive weights $w \in \Wgt^+(\Gamma_1)$. Then
  $n^w_{[f]}(e)$ can be computed by taking a taut
  representative $g \in [f]$ and finding $g^{-1}(x)$
  for $x$ a generic point on~$e$ close to one of the endpoints. 
  The combinatorial type
  of~$g$ determines $g^{-1}(x)$, and by
  Proposition~\ref{prop:strong-red-finite} there are only
  finitely many possibilities for $g^{-1}(x)$. For each possibility,
  $n_g(e)$ is an integer linear combination of weights on~$\Gamma_1$.
  By definition $n^w_{[f]}(e)$ is the minimum of
  these possibilities, so $N_{[f]}$ is a minimum of a finite set of
  linear functions.
\end{proof}


\section{Weak graphs and minimizing Lipschitz energy}
\label{sec:lipschitz}

We start by recalling White's result that the minimal Lipschitz
stretch factor of a map between length graphs is given by the maximal
ratio of lengths of curves.

\begin{theorem}[White]\label{thm:lipschitz-stretch}
  Let $[\phi] \co K_1 \to K_2$ be a homotopy class of maps between
  marked length graphs. Then there is a representative $\psi \in
  [\phi]$ and a marked curve $c \co C \to K_1$ so that the
  sequence
  \[
  \shortseq{C}{c}{K_1}{\psi}{K_2}
  \]
  is tight. In particular,
  \[
  \Lip[\phi] = \Lip(\psi) = \frac{\ell[\psi \circ c]}{\ell[c]}
= \sup_{c \co C \to K_1}\frac{\ell[\psi \circ c]}{\ell[c]},
  \]
  where the supremum runs over all non-trivial curves. We get the same
  quantity if we take the supremum over multi-curves or over curves
  that cross each edge
  of~$K_1$ at most twice.
\end{theorem}

A version of Theorem~\ref{thm:lipschitz-stretch} appears in papers by Bestvina \cite[Proposition
  2.1]{Bestvina11:Bers} and Francaviglia-Martino
  \cite[Proposition 3.11]{FM11:MetricOutSpace}, and is attributed to
  Tad White (unpublished). Both
  papers work in the context of outer space and so assume that
  $\phi$ is a homotopy equivalence, although that assumption is never
  used in the proof. The extension to marked graphs is also
  immediate; we omit the proof.

We need a less trivial extension of
Theorem~\ref{thm:lipschitz-stretch}: we will need to understand
the behavior of energies when lengths degenerate to~$0$. This
leads us to \emph{weak graphs}, in which edges
may have length~$0$. The definition of maps
between weak graphs requires a little care, since we need to remember
homotopy information that may not be present at the level of the graph
obtained by collapsing all $0$-length edges.

\begin{definition}\label{def:weak-graph}
  A \emph{weak graph} is a graph~$\Gamma$ in which certain edges are
  designated as \emph{null edges}; the union of the null edges forms
  the \emph{null graph}. The \emph{collapsed
    graph}~$\Gamma^*$ of~$\Gamma$ is the graph obtained by identifying
  each null edge to a single point. There is a natural
  \emph{collapsing map} $\kappa \co \Gamma \to \Gamma^*$. If $\Gamma$
  is a marked graph, then $\Gamma^*$ is also marked,
  with marked points the image of the marked points of~$\Gamma$.
  Observe that
  $\kappa$ is a homotopy equivalence iff the null graph is a forest,
  with at most one marked point in each component of the null graph.

  If $(\Gamma, \ell)$ is a weak length graph, we consider
  $\Gamma$ to be a weak graph, where the null edges are the edges of
  length~$0$. The lengths $\ell(e)$ determine
  a pseudo-metric on~$\Gamma$ and a metric
  on~$\Gamma^*$.

  If $\Gamma_1$ and~$\Gamma_2$ are weak graphs, a PL \emph{weak map}
  from $\Gamma_1$ to~$\Gamma_2$ is a pair $(\wt f,f)$ of a PL map
  $f \co \Gamma_1^* \to \Gamma_2^*$ between the collapsed graphs,
  together with a map
  $\wt f \co \Gamma_1 \to \Gamma_2$ that is a local homotopy lift of~$f$,
  in the following sense. Pick disjoint regular neighborhoods $N_v^*$ of
  each vertex~$v$ of $\Gamma_2^*$.
  Let $N_2^* = \bigcup_v N_v^* \subset\Gamma_2^*$, $N_1^* = f^{-1}(N_2^*)$, $N_2 = \kappa_2^{-1}(N_2^*)$, and
  $N_1 = \kappa_1^{-1}(N_1^*)$. Then we require that $N_1 = \wt
  f^{-1}(N_2)$ and that the diagram
  \begin{equation}\label{eq:weak-map}
  \mathcenter{\begin{tikzpicture}[x=4em,y=40pt]
    \node(K1) at (0,0) {$\Gamma_1$};
    \node (K2) at (1,0) {$\Gamma_2$};
    \node (K1*) at (0,-1) {$\Gamma_1^*$};
    \node (K2*) at (1,-1) {$\Gamma_2^*$};
    \draw[->] (K1) to node[above,cdlabel] {\wt f} (K2);
    \draw[->] (K1*) to node[above,cdlabel] {f} (K2*);
    \draw[->] (K1) to node[right,cdlabel] {\kappa_1} (K1*);
    \draw[->] (K2) to node[right,cdlabel] {\kappa_2} (K2*);
  \end{tikzpicture}}
  \end{equation}
  commutes up to a homotopy that is the identity on
  $\Gamma_1 \setminus N_1$. We also call $f$ the \emph{shadow} of $\wt f$.

  If the $\Gamma_i$ are marked, we also require that $f$
  and~$\wt f$ be marked maps.
\end{definition}

We do not require
that $\wt f$ be an exact lift of~$f$, as those do not always exist; see
Figure~\ref{fig:weak-lift}.

\begin{figure}
  \[
  \begin{tikzpicture}
    \node (G1) at (0,-0.75) {\graphb{weak-0}};
      \node at (G1.180) [left=2pt] {$\Gamma_1$};
    \node (G2) at (3,0) {\graphb{weak-1}};
      \node at (G2.0) [right=2pt] {$\Gamma_2$};
    \node (G2s) at (3,-1.5) {\graphb{weak-2}};
      \node at (G2s.0) [right=2pt] {$\Gamma_2^*$};
    \draw[->] (G1) -- node[above,cdlabel]{\wt f} (G2.180);
    \draw[->] (G1) -- node[below,cdlabel]{f} (G2s.180);
    \draw[->] (G2) -- node[right,cdlabel]{\kappa} (G2s);
  \end{tikzpicture}
  \]
  \caption{A weak map where an exact lift is impossible. On the
    right is a weak graph~$\Gamma_2$ with a null edge $e_3$,
    and its collapsed graph $\Gamma_2^*$. The map $f \co \Gamma_1 \to
    \Gamma_2^*$ maps $e_1$ forward and backward
    $e_2^*$, so that the inverse image of the
    vertex~$v$ is a single point. The local lift $\wt f \co
    \Gamma_1 \to \Gamma_2$ maps $e_1$ along $e_2$, around $e_3$, and then
    back along $e_2$. There is no exact lift of~$f$.}
  \label{fig:weak-lift}
\end{figure}
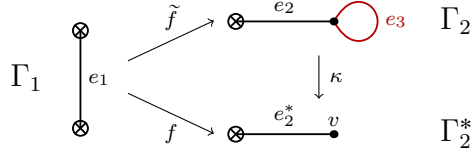

By the \emph{energy} of a weak map $(\wt f, f)$
(with additional weight, elastic, or length structure on the graphs, as
appropriate), we
mean the energy of the map~$f$ between collapsed graphs,
using the same analytic expression as for maps between ordinary
graphs. This applies to all of
the energies in Definition~\ref{def:12inf-graphs}. 

\begin{proposition}\label{prop:weak-homotopy}
  Let $\Gamma_1, \Gamma_2$ be weak graphs, and let
  $[g] \co \Gamma_1 \to \Gamma_2$ be a homotopy class of maps
  between the uncollapsed graphs. Then there is a weak map $(\wt f,f)$ between
  $\Gamma_1$ and $\Gamma_2$, with $\wt f \in [g]$, iff
  \begin{itemize}
  \item there is no null cycle in $\Gamma_1$ that maps to a non-null
    cycle in~$\Gamma_2$; and
  \item there is no null path in $\Gamma_1$ between marked points that
    maps to a non-null path in $\Gamma_2$.
  \end{itemize}
\end{proposition}
Here a \emph{null cycle} or \emph{null path} means one in which every
edge is null.
\begin{proof}
  The ``only if'' direction is immediate. The other direction is an
  application of Theorem~\ref{thm:lipschitz-stretch},
  applied to a metric on the $\Gamma_i$ in which all non-null edges
  have length~$1$ and all null edges have sufficiently small
  length.
\end{proof}

In light of Proposition~\ref{prop:weak-homotopy}, we make the
following definition.
\begin{definition}
  Let $\Gamma_1, \Gamma_2$ be marked weak graphs. A \emph{homotopy
    class} of maps between them is a homotopy class of maps between
  the uncollapsed graphs that contains at least one weak map
  $(f,\wt f)$, i.e., so
  that (up to homotopy) the image of every null cycle (resp.\ path) is a null
  cycle (resp.\ path). Note that this is the homotopy class of
  $\wt f$; the homotopy class of $f$ typically has less information.
\end{definition}

We next generalize the notions of reduced maps.
\begin{definition}\label{def:weak-reduced}
  Let $(\wt f, f) \co (\Gamma_1, \Gamma_1^*) \to (\Gamma_2,\Gamma_2^*)$ be
  a weak marked map. Suppose we have
  point $y \in \Gamma_2^*$, a regular neighborhood $N_y^*$ of~$y$, and
  a component~$Z$ of $\kappa_1^{-1}(f^{-1}(N_y^*))$ that is not an
  entire component of~$\Gamma_1$. Then we say that $Z$
  is a \emph{dead end} if $\wt f$ can be changed
  by a local homotopy so that $\wt f(Z)$ does not intersect
  $\kappa_2^{-1}(y)$. (Compare Definition~\ref{def:reduced}.)
  Concretely, we have the following, where
  $N_y = \kappa_2^{-1}(N_y^*)$.
  \begin{enumerate}
  \item If $Z$ contains a marked point, then $Z$ is not a dead end.
  \item If there are two points $x_1, x_2 \in \bdy Z$ so that
    $\wt f(x_1)$ and $\wt f(x_2)$ are distinct points of
    $\bdy N_y$, then $Z$ is not a dead end.
  \item If there are points $x_1, x_2 \in \bdy Z$ (possibly identical)
    and a path $\gamma$ in~$Z$ between them so that
    $\wt f(x_1) = \wt f(x_2)$ and $\wt f(\gamma)$
    is non-trivial in $\pi_1(N_y,\wt f(x_1))$, then $Z$ is not a dead
    end.
  \item Otherwise, $Z$ is a dead end.
  \end{enumerate}
  If $(\wt f, f)$ has no dead ends, we say that it is \emph{reduced}.
  As in Definition~\ref{def:reduced}, we can \emph{reduce} dead ends:
  if $Z$ is a
  dead end mapping
  to $y \in \Gamma_2^*$, modify $f$
  so that $f(\kappa_2(Z))$ misses~$y$ and continue to pull until you
  reach another singular value of~$f$.
\end{definition}

\begin{proposition}\label{prop:weak-reduced}
  Every weak map is homotopic to a reduced map. Reduction does not
  increase $E^p_q$ for any $p \le q$.
\end{proposition}

\begin{proof}
  As in Proposition~\ref{prop:min-reduced}, energies are not increased
  as you reduce dead ends. To find the reduced map, as in
  Proposition~\ref{prop:reduced}, repeatedly reduce dead ends.
\end{proof}

Furthermore, for a weak curves (a weak map from a marked 1-manifold to a
weak graph), reduced maps minimize all
energies $E^1_p$, as in Lemma~\ref{lem:taut-curve} and
Proposition~\ref{prop:taut-minimal}.

\begin{taggedthm}{4\/$'$}
  \label{thm:weak-stretch}
  Let $[\phi] \co K_1 \to K_2$ be a homotopy class of maps between
  weak marked length graphs, so that $K_1$ has at least one non-null
  edge. Then there is a
  weak map $\psi \in
  [\phi]$ and a marked curve $c \co C \to K_1$ with
  $\ell[c] > 0$ so that the
  sequence
  \[
  \shortseq{C}{c}{K_1}{\psi}{K_2}
  \]
  is tight. In particular,
  \[
  \Lip[\phi] = \Lip(\psi) = \frac{\ell[\psi \circ c]}{\ell[c]}
= \sup_{\substack{c'\co C \to K_1\\\ell[c] > 0}}\frac{\ell[\psi \circ c]}{\ell[c]},
  \]
  where the supremum runs over all curves~$c'$ of positive length in~$K_1$.
  We get the same
  quantity if we take the supremum over multi-curves or over
  curves that
  cross each edge
  of~$K_1$ at most twice.
\end{taggedthm}

To prove Theorem~\ref{thm:weak-stretch}, we restrict the maps
we consider, first for ordinary (non-weak)
maps.

\begin{definition}\label{def:const-deriv}
  A PL map $f \co K_1 \to K_2$ between marked length graphs
  is \emph{constant-derivative} if it is edge-reduced and $\abs{f'}$
  is constant on each edge
  of~$K_1$. A constant-derivative map is determined by its
  combinatorial type (Definition~\ref{def:comb-type}) plus a bounded
  amount of continuous data: for
  each vertex~$v$ of~$K_1$ that maps to the interior of an
  edge~$e$ of~$K_2$, record where $f(v)$ is along~$e$.
\end{definition}

\begin{lemma}\label{lem:deriv-bound}
  For any $D>0$, there are only finitely many
  combinatorial types of constant-derivative maps among maps with
  Lipschitz constant $\le D$.
\end{lemma}

\begin{proof}
  The bound on derivatives gives a bound
  on how many edges of~$K_2$ the image of an edge of~$K_1$ can
  cross.
\end{proof}

Since Lipschitz energy minimizers are constant-derivative, the above
lemma reduces us to finitely many cases to consider. We can give
parallel definitions in the case of weak maps.

\begin{definition}\label{def:weak-comb-type-const-deriv}
  A weak map $(\wt f, f) \co \Gamma_1 \to \Gamma_2$ is
  \emph{edge-reduced} if $\wt f$ is edge-reduced in the sense of
  Definition~\ref{def:edge-reduced}.
  The \emph{combinatorial type} of an
  edge-reduced weak map is the combinatorial type
  of~$f$.%
  \footnote{The map $f$ need not itself be edge-reduced; see
  Figure~\ref{fig:weak-lift} for an example. There is still a
  well-defined sequence of oriented edges that the image passes over,
  so the extension of Definition~\ref{def:comb-type} is clear.}
  If in
  addition $\Gamma_1$ and~$\Gamma_2$ have
  length structures, we say that $(\wt f, f)$ is
  \emph{constant-derivative} if $f$ is constant-derivative.
\end{definition}

\begin{remark}
  Definition~\ref{def:weak-comb-type-const-deriv} depends only on the shadow~$f$
  of the weak map. As a result, unlike for
  ordinary maps, the combinatorial type of a weak map $(\wt f, f)$ does not
  determine its homotopy class~$[\wt f]$. Furthermore, for a
  constant-derivative weak map $(\wt f, f)$, if we fix the
  combinatorial type of~$f$, the homotopy class~$[\wt f]$, and the
  continuous data of the position $f(v)$ for $v \in \Verts(\Gamma_1)$,
  we still may not determine $(\wt f, f)$ up to local
  homotopy. Nevertheless, since the energies we consider
  only depend on $f$, we are able to work with this definition.
\end{remark}

\begin{proof}[Proof of Theorem~\ref{thm:weak-stretch}]
  We follow the previous proofs
  \cite{Bestvina11:Bers,FM11:MetricOutSpace}, adapted to weak
  maps.

  First, we show there is a map that minimizes
  Lipschitz energy. The condition that
  no null cycle of~$K_1$ maps to a non-null cycle of~$K_2$ guarantees
  that there is some constant-derivative weak map
  $(\wt \psi_0, \psi_0)$ of
  finite Lipschitz energy. Then any optimizer $(\wt \psi, \psi)$ must
  have Lipschitz energy less that $\Lip(\psi_0)$.
  By Lemma~\ref{lem:deriv-bound} there are only finitely many
  combinatorial types of weak maps $(\wt \psi, \psi)$ to consider.
  For each constant-derivative maps of a fixed combinatorial type,
  $\Lip(\psi)$ is a continuous function
  of the position of the endpoints along the target edges, and so
  achieves its minimum among maps with that type.

  Let $(\wt \psi, \psi)$
  be a constant-derivative map realizing this minimum.
  We next prove the existence of a marked curve~$c$
  exhibiting the Lipschitz stretch of~$\psi$. For each non-null edge~$e$
  of~$K_1$, say that it is
  \begin{itemize}
  \item a \emph{tension edge} if $\abs{\psi'(e)}$ is the maximal
    value, $\Lip(\psi)$, and
  \item a \emph{slack edge} if $\abs{\psi'(e)} < \Lip(\psi)$.
  \end{itemize}
  Assume that the
  set of tension edges of~$\psi$ is minimal among maps with minimal $\Lip(\psi)$.

  To find the desired curve, we will find a suitable sequence of oriented
  edges $(\vec e_i)_{i=1}^N$ of~$K_1$ passing
  only over tension and null edges.

  \begin{lemma}\label{lem:tension-chain}
    Let $(\wt \psi, \psi)$ be a constant-derivative map realizing the minimal
    Lipschitz factor with minimal set of tension edges.
    Let $\vec e_0$ be an oriented tension edge
    of~$K_1$. Then we can find a reduced
    edge-path
    $\vec e_1,\dots,\vec e_k$ on~$K_1$ so that either
    \begin{itemize}
    \item the $\vec e_i$ for $i=1,\dots,k$ are null edges and
      $\vec e_k$ ends at a marked point, or
    \item the $\vec e_i$ for $i=1,\dots,k-1$ are null edges and $\vec
      e_k$ is a tension edge.
    \end{itemize}
    In the first case, we allow for $k=0$, if $\vec e_0$ already ends
    at a marked point.
    Furthermore, if $c_0 \co [0,1] \to K_1$ is the map passing over
    these edges in order, then $(\wt \psi \circ c_0, \psi \circ \kappa
    \circ c_0)$ is reduced as a weak map, fixing the endpoints.
  \end{lemma}

  \begin{proof}
    Let $v_1$ be the end of $\vec e_0$. If there is no null edge
    touching $v_1$, then either $v_1$ is marked or there is another
    tension edge meeting $v_1$ and mapping by $\psi$ to a different
    direction at $\psi(v_1)$; in either case we are done. Otherwise,
    let $C_1$ be the connected component of the null graph
    of~$K_1$ that contains the end of $\vec e_0$ and let
    $\{x\} = \kappa_1(C_1)$.
    Let $y = \psi(x_1)$ and $C_2$ be the relevant component of
    $\kappa_2^{-1}(y)$.
    $C_2$ may be a point on an edge of $K_2$, a vertex
    of~$K_2$, or a connected null subgraph of~$K_2$. We
    proceed by cases on $C_1$ and~$C_2$, parallel to the cases in
    Definition~\ref{def:weak-reduced}.
    \begin{enumerate}
    \item\label{item:tension-marked} If $C_1$ has a marked point~$x$,
      connect $v_1$ to~$x$ through null edges.
    \item\label{item:two-tension} If there is another tension edge $\vec e'$ of $K_1$
      incident to $C_1$ so that $\vec e'$ and the reverse of $\vec e_0$ map
      by $\psi \circ \kappa$ to distinct directions in $K_2^*$,
      connect $\vec v_1$ to the start of $\vec e'$ by a reduced edge-path in $K_1$.
    \item\label{item:tension-loop} If $\pi_1(C_1)$ maps non-trivially to $\pi_1(C_2)$, connect
      $v_1$ to itself by a reduced loop in $C_1$ that maps
      to a non-trivial loop in $C_2$.
    \end{enumerate}
    If none of the above cases apply, we get a contradiction, as
      follows. Partition the non-null edges
      incident to~$C_1$ according to which direction they map to
      from~$y$. (Put each edge with $\abs{\psi'} = 0$ into its own
      part.) Since we are not in cases~(\ref{item:tension-marked})
      or~(\ref{item:tension-loop}), we
      can displace $y \in K_2^*$ in any direction, reducing
      $\abs{\psi'}$ on one group of edges at the cost of
      increasing it on all other groups.
      Since we are not in
      case~(\ref{item:two-tension}), we can reduce
      $\abs{\psi'(\vec e_0)}$ so $e_0$ is no longer a tension edge
      without creating another tension edge, contradicting the
      assumption that the set of tension edges of $\psi$ was minimal.

    In any of cases
    (\ref{item:tension-marked})--(\ref{item:tension-loop}), by
    construction the sequence of edges in $K_1$ is reduced,
    and the path in $K_2$ falls into one of the cases of
    Definition~\ref{def:weak-reduced} and so is reduced as a weak map.
  \end{proof}

  To complete the proof of
  Theorem~\ref{thm:weak-stretch}, start with any tension
  edge~$\vec e_0$ of~$K_1$. (There is at least one since not all edges of $K_1$
  are null.) Use Lemma~\ref{lem:tension-chain} to construct a
  non-backtracking chain of tension and null edges forward and
  backward from $\vec e_0$. Either the chain ends in a marked point
  both forward and backward, or there is a repeated oriented edge. In
  either case we can extract a marked
  curve~$c$ (either an arc or a loop) so that both $c$ and
  $\psi \circ c$ are reduced (as weak maps). Thus
  $\Lip(\psi) = \ell(\psi \circ c)/\ell(c) = \ell[\psi \circ
  c]/\ell[c]$. If we make the null paths
  as efficient as possible and close off to make a cycle the
  first time that an oriented edge is repeated, then $c$
  will cross each (unoriented) edge of~$K_1$ at most twice.
\end{proof}

\begin{remark}
  The proof of Theorem~\ref{thm:weak-stretch} is similar to the
  proof of Proposition~\ref{prop:tt-curve}. The presence of
  null edges introduces extra complications in
  Theorem~\ref{thm:weak-stretch}.
\end{remark}


\section{Minimizing Dirichlet energy: Harmonic maps\label{sec:harmonic}}

\subsection{Definition and existence\label{sec:harmonic-def}}

We give a concrete description of harmonic maps and
prove their existence,
making the intuitive description of harmonic maps from the introduction
more precise.

\begin{definition}\label{def:harmonic}
  Let $G = (\Gamma, \alpha)$ be a marked elastic graph and let $K$ be a
  marked length graph.
  To any constant\hyp derivative map~$f$, we can associate the
  \emph{tension\hyp weighted graph} $W_f = (\Gamma, \abs{f'})$; that
  is, the weight of
  an edge~$e$ is $\abs{f'(e)}$.
  Then $f$~is \emph{harmonic} if it is constant-derivative
  and the associated map $f \co W_f \to K$ is taut.
\end{definition}

For general weighted graphs, it is not immediately obvious when a map
is taut. We can
simplify the condition in Definition~\ref{def:harmonic} to give
the triangle inequalities at vertices from the introduction.

\begin{proposition}\label{prop:harmonic-tt}
  Let $f \co G \to K$ be a constant-derivative map from a marked
  elastic graph to a marked length graph. Let $G_0 \subset G$ be the
  subgraph of~$G$ on which $\abs{f'} \ne 0$, and let $f_0$ be the
  restriction of~$f$ to~$G_0$.
  Then $f$ is harmonic iff
  $(W_{f_0},\tau(f_0))$ forms a marked balanced train track
  in the sense of Definition~\ref{def:tt}.
\end{proposition}

\begin{proof}
  By Lemma~\ref{lem:taut-subdomain}, $f$~is harmonic iff $f_0 \co
  W_{f_0} \to K$ is taut.
  If $f_0$ is taut, the train-track in condition~(\ref{item:mf-tt}) of
  Theorem~\ref{thm:maxflow-mincut} must be $(W_{f_0}, \tau(f_0))$.
\end{proof}

To be more explicit about the triangle inequalities, fix a
constant-derivative map $f \co G \to K$ from an elastic graph to a
length graph. For $v \in G$ and $d_K$ a direction in~$K$ from $f(v)$,
the \emph{tension} of $d_K$ is
\[
  T(d_K) = \sum \bigl\{\, \abs{f'(d_G)} \mid
    d_G\text{ direction in~$G$ at~$v$},
    f(d_G) = d_K\,\}.
\]
We say that $v$ is \emph{balanced} if
$(\,T(d_K) \mid d_K\text{ direction in $K$ at $f(v)$}\,)$ satisfies
the triangle inequalities.

\begin{corollary}\label{cor:harmonic-triangle}
  Let $f \co G \to K$ be a constant-derivative map from a marked
  elastic graph to a marked length graph. Then $f$~is harmonic iff
  each unmarked vertex of~$G$ is balanced.
\end{corollary}

\begin{proof}
  This is a restatement of
  Proposition~\ref{prop:harmonic-tt}.
\end{proof}

\begin{theorem}\label{thm:harmonic-min}
  Let $[f] \co G \to K$ be a homotopy class of maps from a marked
  elastic graph to a marked length graph. Then there is a harmonic map
  in~$[f]$. Furthermore, the following conditions are equivalent.
  \begin{enumerate}
  \item\label{item:harmonic-global} The map $f$ is a global minimum for~$\Dir$.
  \item\label{item:harmonic-local} The map $f$ is a local minimum for~$\Dir$.
  \item\label{item:harmonic} The map $f$ is harmonic.
  \item\label{item:harmonic-tt} The natural map $\iota \co W_f \to G$ is
    part of a tight sequence
    \[
    \shortseq{W_f}{\iota}{G}{f}{K}.
    \]
  \item\label{item:harmonic-curve} There is a weighted multi-curve
    $(C,c)$ on $G$ that forms a
    tight sequence
    \[
    \shortseq{C}{c}{G}{f}{K}.
    \]
  \end{enumerate}
\end{theorem}

\begin{proof}
  We start with the equivalences. By definition, (\ref{item:harmonic-global})
  implies~(\ref{item:harmonic-local}).

  To show that (\ref{item:harmonic-local}) implies~(\ref{item:harmonic}),
  let $f$ be a local
  minimizer for $\Dir(f)$ within $[f]$; we wish to show $f$ is
  harmonic. One of the first results in calculus of variations is that
  $f$ is constant\hyp derivative. So by
  Proposition~\ref{prop:harmonic-tt} we only need to show the triangle
  inequalities. Let
  $v$ be a unmarked vertex of~$G$ of valence~$k$, and let
  $d_1,\dots,d_k$ be the non-zero
  directions of~$K$ at~$f(v)$. For small~$\epsilon>0$, let
  $f_{i,\epsilon}$ be $f$~modified by moving $f(v)$ a
  distance~$\epsilon$ in the direction $d_i$, extended to the edges
  so that $f_{i,\epsilon}$ is still constant-derivative.
  By hypothesis $\Dir(f) \le
  \Dir(f_{i,\epsilon})$. We have
  \begin{equation*}
  \Dir(f_{i,\epsilon}) = \Dir(f)
    + \sum_{d \in f^{-1}(d_i)} -2\epsilon\abs{f'(d)}
     + \sum_{\mathstrut j \ne i} \,\sum_{d \in f^{-1}(d_j)}
      2\epsilon\abs{f'(d)}
    + \!\sum_{d\text{ direction at }v}\! \epsilon^2/\alpha(d).
  \end{equation*}
  To be a local minimum, we must have for each~$i$
  \[
  \frac{d}{d\epsilon}\bigl(\Dir(f_{i,\epsilon})\bigr)\Big\vert_{\epsilon=0} \ge 0.
  \]
  This gives the
  $i$'th triangle inequality at~$v$, so $f$ is harmonic.

  To see that (\ref{item:harmonic})
  implies~(\ref{item:harmonic-tt}), suppose that $f$ is harmonic. Then
  $f \circ \iota$ is taut and therefore
  energy-minimizing. Furthermore,
  \begin{align*}
    \EL(\iota) &= \sum_{e \in \Edges(\Gamma)} \alpha(e) \abs{f'(e)}^2
               = \Dir(f)\\
    \ell(f \circ \iota) &= \sum_{e \in \Edges(\Gamma)}\abs{f'(e)}\,\ell(f(e))
                      = \Dir(f),
  \end{align*}
  so the energies are multiplicative, as desired.

  Proposition~\ref{prop:tt-curve} tells us that
  (\ref{item:harmonic-tt}) implies~(\ref{item:harmonic-curve}).

  Lemma~\ref{lem:tight} tells us that (\ref{item:harmonic-tt}) or
  (\ref{item:harmonic-curve}) implies~(\ref{item:harmonic-global}).

  To prove existence, let
  $g$ be any constant-derivative map. Then
  $\Dir[f] \le \Dir(g)$. For each edge~$e$, this gives an upper bound
  on possible values of $\abs{f'(e)}$ for any minimizer. Then, as in the proof of
  Theorem~\ref{thm:weak-stretch}, there are only finitely many
  combinatorial types below this bound, and for each combinatorial type
  there is a compact space of piecewise-linear maps. There is therefore a global
  minimum for $\Dir(f)$, which is harmonic.
\end{proof}

\subsection{Harmonic maps to weak graphs}
\label{sec:harmonic-weak}

As with Lipschitz energy, we need a generalization of
Theorem~\ref{thm:harmonic-min} to allow the target to be a weak
length graph.

\begin{definition}\label{def:weak-taut}
  Let $(\wt f, f) \co \Gamma_1 \to \Gamma_2$ be weak map between
  marked weak graphs, and suppose there is a weight structure $w_1$
  on~$\Gamma_1$. Then we say that $(\wt f, f)$ is \emph{taut} if there is a
  lift $\wt f'$ locally homotopic to $\wt f$ so that $\wt f'$ is taut.
\end{definition}

\begin{definition}\label{def:weak-harmonic}
  Let $G = (\Gamma, \alpha)$ be a marked elastic graph and let $K$ be
  a marked weak length graph. We say that a weak map
  $(\wt f, f) \co G \to K$ is \emph{harmonic} if it is
  constant\hyp derivative and the weak map $W_f \to K$ from the
  tension-weighted graph is taut.
\end{definition}

We do not allow $G$ to
be a weak graph. The definition could be extended to this case with
some more work (since when $e$ is a null
edge, $\abs{f'(e)}$ is not defined), but we will
not need it.

We do not know a concise version of
Corollary~\ref{cor:harmonic-triangle} in this setting.

\begin{taggedthm}{5\/$'$}
\label{thm:harmonic-min-weak}
  Let $[f] \co G \to K$ be a homotopy class of maps from a marked
  elastic graph to a marked weak length graph. Then there is a
  harmonic weak map
  in~$[f]$. Furthermore, the following conditions are equivalent.
  \begin{enumerate}
  \item\label{item:weak-harm-global} The weak map $f$ is a global minimum for~$\Dir$.
  \item\label{item:weak-harm-local} The weak map $f$ is a local minimum for~$\Dir$.
  \item\label{item:weak-harm} The weak map $f$ is harmonic.
  \item\label{item:weak-harm-curve} There is a weighted multi-curve
    $(C,c)$ on~$G$ that forms a
    tight sequence
    \[
    \shortseq{C}{c}{G}{f}{K}.
    \]
  \end{enumerate}
\end{taggedthm}

\begin{proof}
  The most significant change from the proof of
  Theorem~\ref{thm:harmonic-min} is the
  proof that (\ref{item:weak-harm-local})
  implies (\ref{item:weak-harm}), which we now do.
  Let
  $(\wt g, g)$ be a constant-derivative PL map, and
  suppose that $g$ is not harmonic in the sense of
  Definition~\ref{def:weak-harmonic}.  
  Then we will find a local
  modification that reduces $\Dir(g)$. The possible local
  modifications are more complicated than in the proof of
  Theorem~\ref{thm:harmonic-min}; we
  use the technology of taut maps from
  Section~\ref{sec:taut}.

  Fix $\epsilon$ sufficiently small. For $v$ a
  vertex of~$G$, let $N_v^* \subset K^*$ be
  the $\epsilon$-neighborhood of~$f(v)$. Make sure $\epsilon$ is
  small enough so that the $N_v^*$ are disjoint regular
  neighborhoods. Let
  $N^* = \bigcup_v N_v^*$, $N_v = \kappa^{-1}(N_v^*) \subset K$,
  $N = \bigcup_v N_v$, $M_v = g^{-1}(N_v^*)\subset G$, and
  $M = \bigcup_v M_v$. We will use $M$ and~$N$ for our neighborhoods
  in the definition of a weak map. On the restriction of $\wt g$ to
  each neighborhood, set up a marked local
  model, analogously to Definition~\ref{def:local-model}, except that
  the target model is not necessarily a star. Using
  Theorem~\ref{thm:maxflow-mincut}, choose the local lift $\wt g$ so
  each local model is
  taut as a map from~$W_g$.

  By assumption, $\wt g$ is not taut, and
  thus is not locally taut in a neighborhood of some point~$y$;
  necessarily $y \in \bdy N_v$ for some vertex~$v$.
  Let $U$ be a small neighborhood of~$y$, and let $Z$ be a component
  of $\wt g^{-1}(U)$ on which $\wt g$ can be homotoped to reduce
  $n_{\wt g}$ (as in Definition~\ref{def:taut} on locally taut). The
  exterior edges of~$Z$ (those that meet $\wt g^{-1}(U)$) can be
  divided into \emph{inside} edges, whose images under $\wt g$ point
  away towards~$v$, and \emph{outside} edges, whose images point away
  from~$v$.

  Since the marked local models for $\wt g$ are all taut, the total
  tension of the inside edges cannot be reduced; in particular,
  the total tension of the inside edges is less than or equal to the
  total tension of the outside edges. In fact, since $\wt g$ as a
  whole is not locally taut at~$y$, this inequality is strict:
  \begin{equation}\label{eq:inside-outside}
    \sum_{e\text{ inside}} \abs{g'(e)} < \sum_{e\text{ outside}} \abs{g'(e)}.
  \end{equation}
  
  Let $(\wt h, h)$ be the constant-derivative small modification of
  $(\wt g, g)$ so that $h$ maps each vertex in~$Z$ to~$y$ and agrees
  with $g$ on all other
  vertices.
  In~$h$ the length of the image of each inside edge was
  increased by~$\epsilon$, and the length of the image of each outside
  edge was decreased by~$\epsilon$. Then
  \begin{align*}
    \Dir(h) - \Dir(g)
    &= \sum_{e\text{ inside}} \frac{\ell(g(e))\epsilon + \epsilon^2}{\alpha(e)}
      + \sum_{e\text{ outside}} \frac{-\ell(g(e))\epsilon + \epsilon^2}{\alpha(e)}\\
    &= \epsilon\biggl(\sum_{e\text{ inside}} \abs{f'(e)}
      - \sum_{e\text{ outside}} \abs{f'(e)}\biggr) + O(\epsilon^2),
  \end{align*}
  which is negative for $\epsilon$ sufficiently small by
  Equation~\eqref{eq:inside-outside}. Thus $\Dir(h) < \Dir(g)$ and
  $g$~is not a local minimum for $\Dir$.

  The rest of the proof is almost unchanged: (\ref{item:weak-harm-global})
  implies (\ref{item:weak-harm-local}) by definition, and
  (\ref{item:weak-harm-curve}) implies (\ref{item:weak-harm-global})
  by properties of tight sequences.  For
  (\ref{item:weak-harm}) implies (\ref{item:weak-harm-curve}), if $f$
  is harmonic, by
  Theorem~\ref{thm:maxflow-mincut} there is a weighted multi-curve $(C,c)$
  on~$G$ so that $n_c = \abs{f'}$. Then $\shortseq{C}{c}{G}{f}{K}$ is
  tight.
  Finally, existence follows from by bounding the
  combinatorial types as before.
\end{proof}

We can improve the local lifts in a weak harmonic map a little.

\begin{definition}
  In a weak map $(\wt f, f) \co \Gamma_1 \to \Gamma_2$, we say that
  the local lift~$\wt f$ is \emph{vertex-precise} if,
  for every
  vertex~$v$ of $\Gamma_1$, $\kappa_2(\wt f(v)) = f(\kappa_1(v))$ on
  the nose (with no need for homotopy).
\end{definition}

\begin{proposition}\label{prop:weak-harmonic-vp}
  If $(\wt f, f) \co G \to K$ is a harmonic weak map, then $\wt f$ can
  be chosen to be vertex-precise and taut as a map from $W_f$ to~$K$.
\end{proposition}

\begin{proof}
  By definition of harmonic weak maps, we can find a taut initial
  lift~$\wt f$. If $\wt f$ does not map vertices to vertices, pick some
  vertex~$v$ so
  that $f(v) \ne \kappa(\wt f(v))$. Then $\wt f(v)$ lies on an
  edge~$e$ incident to $\kappa^{-1}(f(v))$. Since $\wt f$ is taut,
  $n_{\wt f}$ is constant on~$e$. We can thus  push $\wt f(v)$ into
  $\kappa^{-1}(f(v))$ without changing $f$ or~$n_{\wt f}$ (and thus
  keeping $\wt f$ taut). Repeat until
  $\wt f$~maps vertices to vertices.
\end{proof}

\subsection{Uniqueness and continuity}
\label{sec:relaxed}

Harmonic maps are not in general unique in their homotopy class. For
instance, if the target
length graph is a circle, then composing a harmonic map with any
rotation of the circle gives another harmonic map.  However, the length of
the image of an edge in a harmonic map is
unique. In fact, the lengths are unique in a larger set.

\begin{definition}\label{def:relaxed}
  For $\Gamma$ a marked graph, $K$ a weak marked length graph, and
  $[f] \co \Gamma \to K$ a homotopy class of maps, a
  \emph{relaxed
    map}~$r$ with respect to~$[f]$ is an assignment of a length $r(e)$ to each
  edge~$e$ of~$\Gamma$, so that, for any taut weighted marked
  multi-curve $(C,c)$
  on~$\Gamma$,
  \begin{equation}\label{eq:relaxed}
    \sum_{e \in \Edges(\Gamma)} n_c(e) r(e) \ge \ell[f \circ c].
  \end{equation}
  A relaxed map~$r$ naturally gives a weak length metric on~$\Gamma$.
  Let $\Rlx[f] \subset \Len(\Gamma)$ be the set of relaxed maps. We
  write $\Rlx_\ell[f]$ if we
  need to make precise the dependence on
  $\ell \in \Len(K)$.
\end{definition}

Although a relaxed map is not, in fact, any sort of map, the next
three lemmas show how relaxed maps are related to actual maps.

\begin{lemma}\label{lem:relax-lip}
  If $[f] \co \Gamma_1 \to \Gamma_2$ is a homotopy class,
  $r \in \Len(\Gamma_1)$, and $\ell \in \Len(\Gamma_2)$, then $r \in \Rlx_\ell[f]$ iff there is a map
  $g \in [f]$ with $\Lip^r_\ell(g) \le 1$.
\end{lemma}

\begin{proof}
  This is Theorem~\ref{thm:weak-stretch}.
\end{proof}

\begin{lemma}
  If $f \co \Gamma \to K$ is a weak PL map, then
  $f^*\ell \in \Rlx_\ell[f]$.
\end{lemma}

\begin{proof}
  By definition, $\Lip^{f^*\ell}_{\ell}(f) = 1$. Apply
  Lemma~\ref{lem:relax-lip}.
\end{proof}

\begin{lemma}
  If $[f] \co \Gamma \to K$ is a homotopy class of maps from a marked
  graph~$\Gamma$ to a non-trivial weak marked length graph~$K$ and $r \in \Rlx[f]$, then
  there is a PL map $g \in [f]$ so that $r = g^* \ell$.
\end{lemma}

\begin{proof}
  Lemma~\ref{lem:relax-lip} gives a PL map $g_0 \co (\Gamma,r) \to K$
  in~$[f]$ with
  $\Lip(g_0) \le 1$. That is, $\ell(g_0(e)) \le r(e)$ for each
  edge~$e$
  of~$\Gamma$. Define~$g$ by modifying~$g_0$: for each
  edge~$e$ on which $\ell(g_0(e)) < r(e)$, make the length of the
  image of~$e$ longer by introducing some zigzag folds.
\end{proof}

\begin{lemma}\label{lem:relaxed}
  Definition~\ref{def:relaxed} does not change if we let $(C,c)$ be
  \begin{enumerate}
  \item\label{item:relaxed-tt} a marked balanced train track,
  \item\label{item:relaxed-curve} a marked weighted multi-curve (as in
    Definition~\ref{def:relaxed}),
  \item\label{item:relaxed-conn} a marked curve with weight~$1$, or
  \item\label{item:relaxed-simple} a marked curve with weight~$1$
    that crosses each edge at most twice.
  \end{enumerate}
\end{lemma}

\begin{proof}
  The types of curve-like objects are progressively more restrictive,
  so we need to show that the existence of a structure of one type
  violating Equation~\eqref{eq:relaxed}
  implies the next. Condition~(\ref{item:relaxed-tt}) implies
  condition~(\ref{item:relaxed-curve}) by Proposition~\ref{prop:tt-curve}.
  Condition~(\ref{item:relaxed-curve}) implies
  condition~(\ref{item:relaxed-conn}) by the additivity of
  Eq.~\eqref{eq:relaxed} over connected components.
  Condition~(\ref{item:relaxed-conn}) implies
  condition~(\ref{item:relaxed-simple}) by Lemma~\ref{lem:relax-lip} and
  Theorem~\ref{thm:weak-stretch} or, more simply, by taking any
  curve, looking for a maximal segment with no repeated
  oriented edges, doing cut-and-paste if necessary, and then using
  additivity over the connected components.
\end{proof}

\begin{lemma}\label{lem:relaxed-convex}
  For $[f] \co \Gamma_1 \to \Gamma_2$ a homotopy class of maps between marked
  graphs and $\ell \in \Len(\Gamma_2)$, $\Rlx_\ell[f]$
  is a closed, non-compact, convex polytope defined by finitely many
  inequalities, each inequality depending linearly on~$\ell$.
\end{lemma}

\begin{proof}
  This follows from condition~(\ref{item:relaxed-simple}) of
  Lemma~\ref{lem:relaxed}, as there are only finitely many curves
  crossing each edge at most twice, and Equation~\eqref{eq:relaxed}
  cuts out a closed half-space for each such curve. Scaling a relaxed
  map by a factor $\lambda > 1$ gives another relaxed map, so
  $\Rlx[f]$ is not compact.
\end{proof}

If $r \in \Rlx[f]$ is a relaxed map where the domain is an elastic
graph $(\Gamma,\alpha)$,
we can define the Dirichlet energy of~$r$:
\begin{equation}\label{eq:dir-rel}
\Dir^\alpha(r) \coloneqq \sum_{e \in \Edges(\Gamma)} \frac{r(e)^2}{\alpha(e)}.
\end{equation}
In fact, Equation~\eqref{eq:dir-rel} makes sense for any $r \in \Len(\Gamma)$.

We can now give the uniqueness statement.

\begin{proposition}\label{prop:harmonic-unique-min}
  If $f \co G \to K$ is a harmonic map from a marked elastic
  graph to a weak marked length graph, then
  $f^*\ell$ uniquely minimizes Dirichlet energy on $\Rlx_\ell[f]$.
\end{proposition}

\begin{proof}
  Let $G = (\Gamma, \alpha)$. The function $\Dir^\alpha$ is strictly
  convex on $\Len(\Gamma)$, and its sub-level sets are compact. As
  such, $\Dir^\alpha$ achieves a unique minimum on the closed convex
  set $\Rlx[f]$. Since $f$ was harmonic, it minimizes $\Dir(f)$ in
  $[f]$; since every point in $\Rlx[f]$ gives the lengths of an actual
  map in $[f]$, the minimizer in $\Rlx[f]$ must be~$f^*\ell$.
\end{proof}

In light of Proposition~\ref{prop:harmonic-unique-min}, we can think
of Dirichlet energy and the length of edges of the harmonic minimizer as
functions on $\Len(K)$.
\begin{definition}\label{def:Hf}
  For $[f]$ a homotopy class of maps from a marked elastic
  graph $(\Gamma_1,\alpha)$ to a marked graph~$\Gamma_2$, define
  \begin{align*}
    \Dir_{[f]} \co \Len(\Gamma_2) &\to \RR\\
    H_{[f]} \co \Len(\Gamma_2) &\to \Len(\Gamma_1)
  \end{align*}
  by setting $\Dir_{[f]}(\ell)$ to $\Dir^\alpha_\ell[f]$ and
  $H_{[f]}(\ell)$ to the relaxed map in $\Rlx_\ell[f]$
  minimizing~$\Dir^\alpha$.
\end{definition}

\begin{proposition}\label{prop:dirichlet-continuous}
  Let $[f] \co G \to \Gamma_2$ be a homotopy class of maps from a marked
  elastic graph to a marked graph. Then $\Dir_{[f]}$ and $H_{[f]}$ are
  continuous functions on $\Len(\Gamma_2)$, with $\Dir_{[f]}$
  piecewise-quadratic and $H_{[f]}$ piecewise-linear.
\end{proposition}

\begin{proof}
  As $\ell \in \Len(\Gamma_2)$ varies,
  the closed convex set $\Rlx_\ell[f]$ varies continuously by
  Lemma~\ref{lem:relaxed-convex}. Since $\Dir^\alpha$ is strictly convex, both
  the value and location of the minimum of $\Dir^\alpha$ on
  $\Rlx_\ell[f]$ depend continuously on~$\ell$.
  
  Furthermore, since $\Dir^\alpha$ is quadratic on $\Len(\Gamma)$ and
  $\Rlx_\ell[f]$ depends piecewise-linearly on~$\ell$, the
  value of the minimum of $\Dir^\alpha$ on $\Rlx_\ell[f]$ is a piecewise-quadratic
  function of~$\ell$ and the location of the minimum is a
  piecewise-linear function of~$\ell$. (The particular quadratic or
  linear function depends on the face of $\Rlx_\ell[f]$ containing the
  minimum.)
\end{proof}

See also Remark~\ref{rem:Dir-deriv}.

\begin{remark}
  An alternate way to see that $\Dir_{[f]}$ is piecewise-quadratic and
  that $H_{[f]}$ is piecewise-linear is to note that they are
  respectively quadratic and linear for a fixed combinatorial type of
  a harmonic representative, and only finitely many combinatorial
  types of maps appear by Theorem~\ref{thm:harmonic-min-weak} and
  Proposition~\ref{prop:strong-red-finite}. The combinatorial type
  is related to the face of $\Rlx_\ell[f]$ containing
  $H_{[f]}(\ell)$.
\end{remark}


\section{Minimizing embedding energy}
\label{sec:filling}

\subsection{Characterizing minimizers: $\lambda$-filling maps}
\label{sec:lambda-filling}

We now turn to the proof of Theorem~\ref{thm:emb-sf}, starting with a
characterization of which maps can appear as minimizers of
$\Emb[\phi]$. Since we will show that a map that
minimizes $\Emb[\phi]$ is a
simultaneous minimizer of the ratio of Dirichlet energy (for maps to a
length graph) and of the ratio of extremal length (from maps from a
weighted multi-curve), it is helpful to have a
notion that encapsulates both weights and lengths to prove both
tightness statements uniformly.

\begin{definition}\label{def:strip-graph}
  A \emph{strip graph} $S = (\Gamma,w,\ell)$ is a
  marked graph~$\Gamma$ with
  weights $w \in \Wgt(\Gamma)$ and
  lengths $\ell \in \Len(\Gamma)$, so that $w(e) \ne 0$ iff
  $\ell(e) \ne 0$. The strip graph is \emph{positive} if all lengths
  and weights are positive, and is \emph{balanced} if $w$ is balanced
  (Definition~\ref{def:tt}).
 
  There is an associated marked weighted
  graph $W_S = (\Gamma,w)$
  and weak marked length graph $K_S = (\Gamma, \ell)$. In addition to
  these, there is also an associated elastic structure: We say
  that $S$ is
  \emph{compatible} with an elastic graph $G_S = (\Gamma, \alpha)$ if
  $\ell(e) = w(e) \alpha(e)$ for each edge~$e$. (If $S$ is positive,
  then $G_S$ is unique.)
  A strip graph also has an \emph{area}
  \[
  \Area(S) \coloneqq \sum_{e \in \Edges(\Gamma)} \ell(e)w(e).
  \]
\end{definition}

To build intuition for strip graphs, we relate strip graphs to
tight sequences of maps.
\begin{proposition}\label{prop:strip-tight}
  A balanced strip graph $S$ gives a tight
  sequence of maps
  \[
  \longseq{W_S}{}{G_S}{}{K_S}{}{K^*_S}.
  \]
\end{proposition}
(Recall from Definition~\ref{def:weak-graph} that $K_S^*$ is $K_S$
with the null edges collapsed.)

\begin{proof}
  The fact that the map $W_S \to K_S$ is taut (hence
  energy-minimizing) is the condition that $w$ is balanced. The map
  $W_s \to K_S^*$ is then taut by Lemmas~\ref{lem:taut-subdomain}
  and~\ref{lem:taut-subrange}.
  We also have
  \begin{align*}
    \ell(W_S \to K_S^*)
      = \EL(W_S \to G_S)
        = \Dir(G_S \to K_S)
          &= \Area(S)\\
    \Lip(K_S \to K^*_S) &= 1,
  \end{align*}
  so the energies are multiplicative, as desired.
\end{proof}

\begin{proposition}
  Let $f \co G \to K$ be a harmonic map, with
  $G = (\Gamma, \alpha)$. Then there is a unique
  balanced strip structure $S = (\Gamma, w, \ell)$ on~$\Gamma$,
  compatible with~$\alpha$, so that $f$ is an isometry from
  $(\Gamma,\ell)$ to~$K$.

  Similarly, let $f \co W \to G$ be a taut map from a weighted graph
  to an elastic graph. Then there is a unique balanced strip structure
  $S = (\Gamma,w,\ell)$ compatible with~$G$ so that $n_f(y) = w(y)$.
\end{proposition}
\begin{proof}
  For the first statement, set $\ell(e)$ to be the length of the
  image, $f(e)$,
  and set $w(e)$ to be the tension,
  $\abs{f'(e)}$. For the second statement, set $w(e)$ to be the
  multiplicity covered~$n_f(e)$ and $\ell(e) = \alpha(e)\cdot w(e)$.
  Verifying the required properties is straightforward. In particular,
  ``balanced'' follows from Corollary~\ref{cor:harmonic-triangle} in
  the harmonic map case.
\end{proof}

\begin{definition}\label{def:lambda-filling}
  Let $S_1 = (\Gamma_1, w_1, \ell_1)$ and
  $S_2 = (\Gamma_2, w_2,\ell_2)$ be two marked balanced strip
  graphs. Write $G_1$ for the elastic graph $G_{S_1}$ associated to
  $S_1$, and so forth. (This involves choices if the $S_i$ are not
  positive.) Let $\phi \co \Gamma_1 \to \Gamma_2$ be a PL map. Since
  there are many different structures on the $\Gamma_i$, we give names
  to $\phi$ considered as a map between different structures: we have
  \begin{align*}
    \phi^W_W &\co W_1 \to W_2 &
      \phi^W_G &\co W_1 \to G_2 &
        \phi^W_K &\co W_1 \to K_2 \\
    &&\phi^G_G &\co G_1 \to G_2 &
        \phi^G_K &\co G_1 \to K_2 \\
    &&&&\phi^K_K &\co K_1 \to K_2.
  \end{align*}
  Then if $\lambda > 0$ is a real number and $S_2$ is positive, $\phi$ is
  \emph{$\lambda$-filling} if
  \begin{enumerate}
  \item\label{item:taut} $\phi^W_K$ is taut;
  \item\label{item:length-pres} lengths are preserved: $\phi^K_K$ is a
    local isometry; and
  \item\label{item:weight-scale} weights are scaled by a factor
    of~$\lambda$: for every
    regular value
    $y \in \Gamma_2$,
    \begin{equation}\label{eq:lambda-weights}
    \sum_{x \in \phi^{-1}(y)} w_1(x) = \lambda w_2(y).
    \end{equation}
    In other words, $n_\phi^{w_1} = \lambda w_2$ and in particular
    $\WR(\phi^W_W) = \lambda$.
  \end{enumerate}
  We say that a map~$\phi$ between elastic graphs is $\lambda$-filling if
  there are compatible strip structures so that $\phi$ is $\lambda$-filling as a
  map between the strip structures.
\end{definition}

\begin{lemma}\label{lem:filling-fill}
  Suppose $\phi \co S_1 \to S_2$ is a $\lambda$-filling map.  Then $\Fill_{\phi^G_G}$
  is identically equal to~$\lambda$. In particular,
  $\Emb(\phi^G_G) = \lambda$.
\end{lemma}

\begin{proof}
  Since $\phi$ is constant on the $0$-length edges of~$S_1$, we may
  delete these edges and
  assume that $S_1$ is positive as well.
  Since $\phi^K_K$ is length-preserving and
  $\alpha_i(e)=\ell_i(e)/w_i(e)$ for $i=1,2$, it
  follows that, for regular points $x \in G_1$,
  \[
  \left\lvert\bigl(\phi^G_G\bigr)'(x)\right\lvert = w_1(x)/w_2(\phi(x)).
  \]
  The result follows from Equation~\eqref{eq:lambda-weights}.
\end{proof}

\begin{lemma}\label{lem:filling-tight}
  A $\lambda$-filling map
  $\phi \co S_1 \to S_2$ gives
  a tight sequence of maps
  \[
  \longseq{W_1}{\id_{\Gamma_1}}{G_1}{\phi^G_G}{G_2}{\id_{\Gamma_2}}{K_2}.
  \]
\end{lemma}

\begin{proof}
  The composite $\phi^W_K \co W_1 \to K_2$ is taut by assumption.
  The energies of the various maps are
  \begin{align*}
    \EL(\id_{\Gamma_1}) &= \Area(S_1) && \text{by Proposition~\ref{prop:strip-tight}}\\
    \Emb(\phi^G_G) &= \lambda &&\text{by Lemma~\ref{lem:filling-fill}}\\
    \Dir(\id_{\Gamma_2}) &= \Area(S_2) = \Area(S_1)/\lambda
       &&\text{by Proposition~\ref{prop:strip-tight} and Equation~\eqref{eq:lambda-weights}.}
  \end{align*}
  Then $\ell(W_1 \to K_2) = \Area(S_1) =
  \sqrt{\EL(\id_{\Gamma_1})\Emb(\phi^G_G)\Dir(\id_{\Gamma_2})}$. which
  is multiplicative, as desired.
\end{proof}

\begin{proposition}\label{prop:lambda-filling-emb}
  Let $\phi \co G_1 \to G_2$ be a map between marked elastic graphs
  $G_1$ and~$G_2$. If there is a $\lambda$-filling map
  $\psi \in [\phi]$, then
  $\Emb[\phi] = \SF_{\Dir}[\phi] = \SF_{\EL}[\phi] = \Emb(\psi) = \lambda$.
\end{proposition}

\begin{proof}
  Immediate from Lemmas~\ref{lem:filling-fill},
  \ref{lem:filling-tight} and~\ref{lem:tight}.
\end{proof}

Thus, $\lambda$-filling maps are optimizers for $\Emb$. If there were
always a $\lambda$-filling map in~$[\phi]$, then we would be done
with Theorem~\ref{thm:emb-sf}. Unfortunately this is not
true.

\begin{example}\label{examp:partial-filling}
  Let $G_1$ and $G_2$ both be the join of two circles, with elastic
  constants $(2,2)$ and $(1,4)$, respectively, and let
  $\phi \co G_1 \to G_2$ be the constant-derivative map which is the identity on~$\pi_1$
  and maps the vertex to the vertex:
  \[
  \mfigb{graphs-61} \quad\overset{\phi}{\longrightarrow}\quad \mfigb{graphs-62}.
  \]
  Then
  $\Emb(\phi) = 2$. Furthermore, $\SF_{\EL}[\phi] \ge 2$, by
  considering the right-hand loop of~$G_1$. Thus
  $\Emb[\phi] = \Emb(\phi) = \SF_{\EL}[\phi]= 2$.
  There is no
  $\lambda$-filling map in~$[\phi]$.
\end{example}

\begin{definition}\label{def:partial-filling}
  Let $S_1 = (\Gamma_1, w_1, \ell_1)$ and
  $S_2 = (\Gamma_2, w_2,\ell_2)$ be two marked balanced strip graphs,
  with $S_2$ positive.
  Then, for $\lambda > 0$, a map
  $\phi \co S_1 \to S_2$ is \emph{partially $\lambda$-filling} if
  there are complementary subgraphs
  $\Delta_i, \Sigma_i \subset \Gamma_i$ (i.e., with complementary sets
  of edges), so that
  \begin{enumerate}
  \item\label{item:Gamma-0-image} $\phi(\Delta_1) \subset \Delta_2$
    (so we can write $\phi \vert_\Delta$ for the restriction;
  \item\label{item:Sigma-image} $\phi(\Sigma_1) \subset \Sigma_2$ (so
    we can write $\phi \vert_\Sigma$ for the restriction);
  \item the strip graph structures induced on~$\Delta_i$
    and~$\Delta_2$ are balanced;
  \item the restriction $\phi\vert_\Delta \co \Delta_1 \to \Delta_2$ is
    $\lambda$-filling with respect to the induced strip graph structures;
  \item $\phi$ is everywhere length-preserving; and
  \item\label{item:scale-less} $\phi \vert_\Sigma$ scales weights by
    strictly less
    than~$\lambda$: for every regular value $y \in \Sigma_2$,
    \[
    \sum_{x \in \phi^{-1}(y)} w_1(x) < \lambda w_2(y).
    \]
  \end{enumerate}
  We call $\Delta_1$ and $\Delta_2$ the \emph{filling
    subgraphs} of $S_1$ and~$S_2$.
  Set $W_1^0$ to be the marked balanced weighted graph
  $(\Delta_1,w_1)$ and set $K_2^{0*}$ to be the marked length graph
  $(\Gamma_2, \ell_2)$ with the edges in~$\Sigma_2$
  collapsed to points.
\end{definition}

\begin{lemma}\label{lem:partial-tight}
  If $\phi \co S_1 \to S_2$ is a partially $\lambda$-filling map then,
  with notation as in Definition~\ref{def:partial-filling}, there is a
  tight sequence
  \[
  W_1^0 \longrightarrow G_1
    \overset{\phi^G_G}{\longrightarrow} G_2
    \longrightarrow K_2^{0*}.
  \]
\end{lemma}

Here, we are thinking about $K_2^{0*}$ as an ordinary graph. (Later
there will be an associated weak length graph.)
\begin{proof}
  The version of Lemma~\ref{lem:filling-fill} in this context says
  that $\Fill_{\phi^G_G}(y) = \lambda$ if $y$ is a regular point in~$\Delta_2$ and
  $\Fill_{\phi^G_G}(y) < \lambda$ if $y$ is a regular point in $\Sigma_2$. We then
  have
  \begin{align*}
    \EL(W_1^0 \to G_1) &= \Area(\Delta_1)\\
    \Emb(G_1 \to G_2) &= \lambda\\
    \Dir(G_2 \to K_2^{0*}) &= \Area(\Delta_2) = \Area(\Delta_1)/\lambda\\
    \ell[W_1^0 \to K_2^{0*}] = \ell(W_1^0 \to K_2^{0*}) &= \Area(\Delta_1),
  \end{align*}
  which is multiplicative.

  Since $\phi\vert_\Delta$ is $\lambda$-filling, the map
  $W_1^0 \to \Delta_2$ is taut. By Lemma~\ref{lem:taut-subrange}
  applied to $W_1^0 \to \Gamma_2$, this implies that
  $W_1^0 \to K_2^{0*}$ is taut, completing the proof the sequence is tight.
\end{proof}

Thus partially $\lambda$-filling maps are also global minimizers for
$\Emb$.
The bulk of this section will be devoted to proving they exist:

\begin{proposition}\label{prop:partial-filling}
  For $[\phi] \co G_1 \to G_2$ a homotopy class of maps
  between marked elastic graphs, there are compatible strip structures on
  $G_1$ and $G_2$ and map $\psi \in [\phi]$ so that $\psi$ is
  partially $\lambda$-filling.
\end{proposition}

Proposition~\ref{prop:partial-filling} suffices to prove
Theorem~\ref{thm:emb-sf}. We spell out some further consequences.

\begin{proposition}\label{prop:emb-sf-detail}
  For $[\phi] \co G_1 \to G_2$ a homotopy class of maps between
  elastic graphs, there is a tight sequence
  \[
  C \overset{c}{\longrightarrow} T \overset{t}{\longrightarrow} G_1
    \overset{\psi}{\longrightarrow} G_2 \overset{f}{\longrightarrow} K
  \]
  where $\psi \in [\phi]$, $T$ is a balanced train track whose
  underlying graph is a subgraph of~$G_1$, $t$ is the inclusion map,
  $K$ is a length graph whose underlying graph is obtained by
  collapsing some edges of~$G_2$, $f$ is the collapsing map,
  $(C,c)$ is a weighted multi-curve saturating~$T$, and $\psi \circ t$ is a
  train-track map. Furthermore,  the image of $\psi \circ t$ is the
  union of the edges that are not collapsed
  by~$f$, and $\Fill_\psi$ is constant
  and maximal on those edges.
\end{proposition}

\begin{proof}[Proof of Proposition~\ref{prop:emb-sf-detail}, assuming
  Proposition~\ref{prop:partial-filling}]
  By Proposition~\ref{prop:partial-filling}, there is a partially
  $\lambda$-filling map $\psi \in [\phi]$.
  Lemma~\ref{lem:partial-tight} gives a tight sequence
  \[
  \longseq{W_1^0}{i}{G_1}{\psi}{G_2}{f}{K_2^{0*}}.
  \]
  Since $i$ is taut, Theorem~\ref{thm:maxflow-mincut} gives a tight
  sequence
  of maps
  \[
    C \overset{c}{\longrightarrow}
      T \overset{t}{\longrightarrow}
        W_1^0 \overset{i}{\longrightarrow}
          G_1 \overset{\psi}{\longrightarrow}
            G_2 \overset{f}{\longrightarrow}
              K_2^{0*}
  \]
  where $C$ is a weighted multi-curve and $T$ is a balanced train track,
  whose underlying graph is a subgraph of $W_1^0$ (and so $G_1$).
  Dropping $W_1^0$ gives the desired sequence.
  The edges not collapsed by~$f$,
  the image of $\psi \circ i \circ t$, and the edges where $\Fill_\psi$ is
  constant are all equal to the filling subgraph of~$G_2$.
\end{proof}

\begin{proof}[Proof of Theorem~\ref{thm:emb-sf}, assuming
  Proposition~\ref{prop:partial-filling}]
  Immediate consequence of Proposition~\ref{prop:emb-sf-detail} and
  Lemma~\ref{lem:tight}.
\end{proof}

\subsection{Iterating to optimize embedding energy}
\label{sec:iteration}

One approach to proving Proposition~\ref{prop:partial-filling} would
be to study those maps that locally minimize the embedding energy,
analogously to the proof of Theorem~\ref{thm:lipschitz-stretch}. From
Lagrange multipliers at the
local minimum, you can extract lengths and weights to form the
desired strip structure.
This is a little involved,
  since the relevant ambient space in which to apply the theory of
  Lagrange multipliers is not obvious; there is no obvious analogue of
  the space of constant-derivative maps used in
  Section~\ref{sec:harmonic}.
We will therefore take
a different approach, one that also suggests an algorithm to actually
compute the embedding energy.

From a homotopy class $[\phi] \co G_1 \to G_2$ of maps
between marked elastic graphs, we will give an iteration that
has a fixed point at a partially $\lambda$-filling map. To motivate the
iteration, we give an analogue in the setting of vector spaces.
First, recall that a strictly convex norm on a finite-dimensional
vector space~$V$
  defines a \emph{dual
  norm} on the dual space~$V^*$. This is the minimal norm that satisfies,
  for all $v \in V$ and $v^* \in V^*$,
  \begin{equation}\label{eq:norm-submul}
    \langle v^*,v \rangle \le \norm{v}\norm{v^*}.
  \end{equation}
  Equation~\eqref{eq:norm-submul} is tight in the sense of
  Proposition~\ref{prop:sub-mult}, namely, for every non-zero $v \in V$ there
  is a non-zero $v^* \in V^*$, unique up to scale, so that
  Equation~\eqref{eq:norm-submul}
  is an equality. If in addition $\norm{v^*} = \norm{v}$, we say that $v$ and $v^*$ 
  \emph{support} each other. (The hyperplane corresponding to $v^*$ is
  parallel to a supporting hyperplane for a norm ball in~$V$ at~$v$.)

\begin{example}\label{examp:bilin}
  If $\norm{v} = \sqrt{\langle v, v\rangle}$ for an inner product
  $\langle\cdot,\cdot\rangle$, the map from a vector to its supporting
  vector is the canonical isomorphism $V \xrightarrow{\sim} V^*$ given
  by $v \mapsto \langle v, \cdot \rangle$.
\end{example}

Now suppose we given an isomorphism $\phi \co V \to W$ between two
finite-dimensional vector spaces, with a strictly convex norm on each.
We wish to
find the operator norm $\norm{\phi}_{V,W}$ by finding a non-zero vector $v \in V$ that
maximizes the ratio of norms
\[
\NR(v) \coloneqq \frac{\norm{\phi(v)}_W}{\norm{v}_V}.
\]

\begin{algorithm}\label{alg:iter-vect}
  To attempt to maximize $\NR(v)$, pick $v_0 \in V$
  and set $i=0$.
  \begin{enumerate}
  \item \label{step:map-1} Let $w_i \in W$ be $\phi(v_i)$.
  \item \label{step:supp-2} Find the supporting
    vector $w_i^* \in W^*$ for~$w_i$.
  \item \label{step:map-2}Let $v_i^* \in V^*$ be $\phi^*(w_i^*)$.
  \item \label{step:supp-1}Find the supporting vector
    $v_{i+1} \in V$ for~$v_i^*$, increase $i$ by~$1$, and return to
    Step~(\ref{step:map-1}).
  \end{enumerate}
  This gives sequences of vectors $v_i\in V$, $v_i^* \in V^*$,
  $w_i \in W$, and $w_i^* \in W^*$. We may also consider the
  corresponding sequences $[v_i]$, etc., in the respective projective spaces. The
  candidate for maximizing
  $\NR$ on $PV$ is $\lim_{i\to \infty} [v_i]$,
  if it exists.
\end{algorithm}

Let
$\Iter_\phi\co V\righttoleftarrow$ be the composition of the steps in
Algorithm~\ref{alg:iter-vect}, and let
$P\Iter_\phi \co PV\righttoleftarrow$ be the corresponding map on
projective spaces.

  \begin{lemma}\label{lem:objective-increase}
    $\NR(v)$ does not decrease under $\Iter_\phi$:
    $\NR(v) \le \NR(\Iter_\phi(v))$. If we have equality, then
    $[v]$ is a fixed point of $P\Iter_\phi$.
  \end{lemma}

\begin{proof}
  We use notation from Algorithm~\ref{alg:iter-vect}.
  Repeatedly apply Equation~\eqref{eq:norm-submul}, as an
  equality for vectors that support each other:
  \begin{align*}
    \norm{w_i} \, \norm{w_i^*} &= \langle w_i^*, w_i \rangle
      = \langle w_i^*, \phi v_i \rangle = \langle \phi^* w_i^*, v_i \rangle
      = \langle v_{i}^*, v_i \rangle\\
    \norm{v_{i}^*}\, \norm{v_{i+1}} &= \langle v_{i}^*, v_{i+1} \rangle
      = \langle \phi^* w_i^*, v_{i+1} \rangle = \langle w_i^*, \phi v_{i+1} \rangle
      = \langle w_i^*, w_{i+1} \rangle\\
    \frac{\NR(v_i)}{\NR(v_{i+1})} = \frac{\norm{w_i}}{\norm{v_i}}\frac{\norm{v_{i+1}}}{\norm{w_{i+1}}}
      &= \frac{\langle v_{i}^*, v_i \rangle}{\norm{v_i}\,\norm{w_i^*}}
         \frac{\langle w_i^*, w_{i+1} \rangle}{\norm{w_{i+1}}\,\norm{v_{i}^*}}
      \le \frac{\norm{v_{i}^*}\norm{v_i}}{\norm{v_i}\,\norm{w_i^*}}
         \frac{\norm{w_i^*}\norm{w_{i+1}}}{\norm{w_{i+1}}\,\norm{v_{i}^*}} = 1.
  \end{align*}
  If we have equality, then
  $\langle v_i^*, v_i\rangle = \norm{v_i^*}\norm{v_i}$ and so $v_i$ is
  a multiple of the supporting vector for~$v_i^*$, namely~$v_{i+1}$.
\end{proof}

\begin{corollary}
  A vector~$[v_0] \in PV$ that maximizes $\NR(v_0)$ is a fixed
  point for $P\Iter_\phi$. If $v_0$ is an attracting fixed point for the
  $P\Iter_\phi$, then $\norm{\phi v}/\norm{v}$ has a local maximum
  at~$v_0$.
\end{corollary}

\begin{example}\label{examp:iter-bilin}
  In the setting of Example~\ref{examp:bilin}, if the norms on $V$
  and~$W$ come from inner products, then $\Iter_\phi$ is $\phi^*\phi$
  and the iteration reduces to power iteration: find
  the maximum eigenvector of $\phi^*\phi$ by repeatedly
  applying it. This almost always converges to an eigenvector of
  maximal eigenvalue, with convergence rate determined by the ratio
  between the two largest distinct eigenvalues.
\end{example}

\begin{example}
  Consider the case when $\phi$ is the
  identity on $\RR^n$ and $\norm{\cdot}_1$ is the standard inner
  product. Then $\norm{\cdot}_2$ is defined by its unit norm ball
  $B_2 \subset \RR^n$.  The supporting vector of $v \in \bdy B_1$ is
  the tangent hyperplane to~$B_2$. Up to scale,
  Algorithm~\ref{alg:iter-vect} alternates between taking a
  tangent hyperplane to~$B_2$, finding the closest point to the origin
  on the
  tangent hyperplane, and projecting to $B_2$, as in
  Figure~\ref{fig:iter-vect}.
\end{example}

\begin{figure}
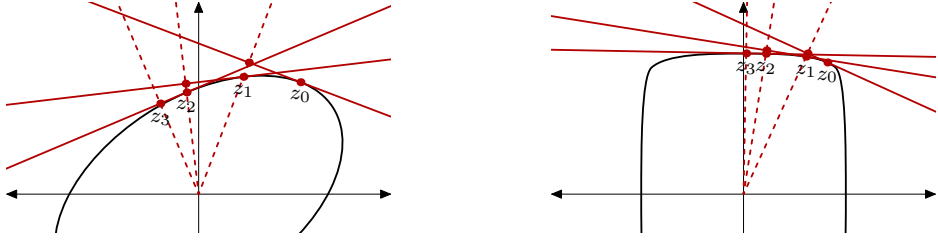

  \begin{align*}
    \mfigb{iter-11} && \mfigb{iter-14}
  \end{align*}
  \caption{Examples for iteration on vector spaces (up to scale). On
    the left, $B_1$
    comes from a quadratic norm as in
    Example~\ref{examp:iter-bilin}. The right shows an example where
    Algorithm~\ref{alg:iter-vect} does not converge to a global
    maximum of~$\NR$.}
  \label{fig:iter-vect}
\end{figure}

We now return to the actual case of interest of elastic graphs. For a marked
elastic graph~$G = (\Gamma, \alpha)$, we think of a ``duality''
between homotopy classes of maps $f \co G \to K$ to a length graph and maps
$c \co W \to G$ from a weighted graph. The pairing is given by
$\langle [c], [f]\rangle = \ell[f \circ c]$, and the two norms are
$\sqrt{\Dir[f]}$ and $\sqrt{\EL[c]}$. The analogue of
Equation~\eqref{eq:norm-submul} is Equation~\eqref{eq:sm-CGK}.
More concretely, we work with $\Len(\Gamma)$
and $\Wgt(\Gamma)$.
\begin{definition}\label{def:duality}
  For $G$ an elastic graph, the \emph{pairing} $\langle
  \cdot,\cdot\rangle \co \Wgt(\Gamma) \times \Len(\Gamma) \to \RR$ is
  defined by
  \[
    \langle w, \ell \rangle \coloneqq \sum_{e \in \Edge(\Gamma)} w(e)\ell(e)
  \]
  and the \emph{duality map}
  $D_G \co \Wgt(\Gamma) \to \Len(\Gamma)$ is defined by
  \[
    D_G(w)(e) \coloneqq \alpha(e)w(e).
  \]
\end{definition}
For $w \in \Wgt(\Gamma)$ a balanced set of weights and
$\ell \in \Len(\Gamma)$ any set of lengths,
$\langle w, \ell\rangle = \ell[f \circ c]$ where
$c \co (\Gamma, w) \to G$ and $f \co G \to (\Gamma, \ell)$ are both
the identity map.

Fix $[\phi] \co G_1 \to G_2$ a homotopy class of maps between marked
elastic graphs $G_1 = (\Gamma_1,\alpha_1)$ and $G_2 = (\Gamma_2,\alpha_2)$.
For $\ell \in \Len(\Gamma_2)$, let $\DR(\ell)$ be the ratio of
Dirichlet energies
$\Dir^{\alpha_1}_\ell[\phi]/\Dir^{\alpha_2}_\ell[\id_{\Gamma_2}]$, and for
$w \in \Wgt(\Gamma_1)$, let $\ER(w)$ be the ratio of extremal lengths
$\EL^w_{\alpha_2}[\phi]/\EL^w_{\alpha_1}[\id_{\Gamma_1}]$.

\begin{algorithm}\label{alg:iter}
  To attempt to maximize $\DR$ and~$\ER$,
  pick a generic initial set of lengths
  $\ell_0 \in \Len(\Gamma_2)$, and let $K_0$ be the marked
  length graph $(\Gamma_2,\ell_0)$. We will write $D_i$ for $D_{G_i}$.
  Let
  $f_0 \co G_2 \to K_0$ be the identity on the level of graphs. Set
  $i=0$ and consider the following iteration.
\begin{enumerate}
\item\label{item:step-harmonic} Find a harmonic representative~$g_i$ of the
  composite map $[\phi \circ f_i] \co G_1 \to K_i$. For each edge $e$
  of~$\Gamma_1$, let $m_i(e) = \ell(g_i(e))$. Thus $m_i =
  H_{[f]}(\ell_i)$, with
  $H_{[f]}$ (Definition~\ref{def:Hf}).
\item\label{item:step-W} Set $w_i = D_1^{-1}(m_i) \in \Wgt(\Gamma_1)$,
  so that $w_i(e)$ is $\abs{g_i'(e)}$. Let
  $W_i$ be
  the tension-weighted graph $(\Gamma_1, w_i)$, and let $d_i \co W_i \to G_1$
  be the identity on the level of graphs. $W_i$ is
  balanced by
  Proposition~\ref{prop:harmonic-tt}.
\item\label{item:step-push} Set $v_i = N_{[\phi]}(w_i)$, the
  push-forward of~$w_i$. Since $g_i$ is harmonic and thus taut from
  the tension-weighted graph, we have
  \begin{equation}\label{eq:weights-sum}
  v_i(y) = \sum_{x \in g_i^{-1}(y)} w_i(x).
  \end{equation}
  The weights~$v_i$ are also balanced, by
  Proposition~\ref{prop:taut-balanced}.
  Let $c_i \co W_i \to G_2$ be $g_i \circ d_i$.
\item\label{item:step-K} Set $\ell_{i+1} = D_2(v_i)\in\Len(\Gamma_2)$, let
  $K_{i+1}=(\Gamma_2,\ell_{i+1})$, and let
  $f_{i+1} \co G_2 \to K_{i+1}$ be the identity on the level of
  graphs. Increase $i$ by~$1$ and return to
  Step~(\ref{item:step-harmonic}).
\end{enumerate}
\end{algorithm}

Schematically, Algorithm~\ref{alg:iter} iterates around the following loop.
\begin{equation}\label{eq:iteration}
  \begin{tikzpicture}[baseline]
    \matrix[row sep=20pt, column sep=2em]{
      \node (L2) {$\mathllap{\ell_i \in{}} \Len(\Gamma_2)$};
        & \node (W2) {$\Wgt(\Gamma_2)\mathrlap{{} \ni v_i}$};\\
      \node (L1) {$\mathllap{m_i \in{}} \Len(\Gamma_1)$};
        & \node (W1) {$\Wgt(\Gamma_1)\mathrlap{{} \ni w_i.}$};\\
    };
    \draw [->] (L2) -- node[left,cdlabel]{H_{[\phi]}} (L1);
    \draw [->] (L1) -- node[below,cdlabel]{D_1^{-1}} (W1);
    \draw [->] (W1) -- node[right,cdlabel]{N_{[\phi]}} (W2);
    \draw [->] (W2) -- node[above,cdlabel]{D_2} (L2);
  \end{tikzpicture}
\end{equation}
  If all lengths and weights remain positive, we have a
  diagram of spaces and maps, in which
  \begin{itemize}
  \item rows are tight sequences (see below);
  \item the dashed lines are only defined up to homotopy; and
  \item the regions commute up to homotopy:
  \end{itemize}
  \begin{equation}
  \mathcenter{\begin{tikzpicture}[x=4em,y=35pt]
    \node(W1) at (0,0) {$W_{i-1}$};
    \node (G2) at (2,0) {$G_2$};
    \node (K1) at (3,0) {$K_i$};
    \node (K1a) at (3,-1) {$K_i$};
    \node (W2) at (0,-1) {$W_i$};
    \node (W2a) at (0,-2) {$W_i$};
    \node (G1) at (1,-1) {$G_1$};
    \node (G2a) at (2,-2) {$G_2$};
    \node (K2) at (3,-2) {$K_{i+1}$.};
    \draw [->] (W1) -- node[above,cdlabel]{c_{i-1}} (G2);
    \draw [->] (G2) -- node[above,cdlabel]{f_i} (K1);
    \draw [->] (W2) -- node[above,cdlabel]{d_i} (G1);
    \draw [->] (G1) -- node[above,cdlabel]{g_i} (K1a);
    \draw [->] (W2a) -- node[above,cdlabel]{c_i} (G2a);
    \draw [->] (G2a) -- node[above,cdlabel]{f_{i+1}} (K2);
    \draw [double equal sign distance] (K1) -- (K1a);
    \draw [double equal sign distance] (W2) -- (W2a);
    \draw [dashed,->] (G1) -- node[above left=-3pt,cdlabel]{[\phi]} (G2);
    \draw [dashed,->] (G1) -- node[above right=-3pt,cdlabel]{[\phi]} (G2a);
  \end{tikzpicture}}\label{eq:tight-iteration}
  \end{equation}
  The row $\shortseq{W_i}{d_i}{G_1}{g_i}{K_i}$ is tight since $g_i$ is
  harmonic and $w_i = \abs{g_i'}$, while the row
  $\shortseq{W_i}{c_i}{G_2}{f_{i+1}}{K_{i+1}}$ is tight because $c_i$
  is taut and $\ell_{i+1} = D_2(v_i)$.

Let $\Iter_\phi \co \Len(\Gamma_2)\righttoleftarrow$ be the
composition 
\[
\Iter_\phi = D_2 \circ N_{[\phi]} \circ D_1^{-1} \circ H_{[\phi]}.
\]
All of
the maps involved are piecewise-linear and linear on rays by
Propositions~\ref{prop:nf-cont-pl} and~\ref{prop:dirichlet-continuous}, so
$\Iter_\phi$ is as well. 
We allow some of the lengths or weights to vanish (i.e., we include
the boundary of $\Len(\Gamma_i)$ and $\Wgt(\Gamma_i)$ in the
maps).

\begin{lemma}\label{lem:iter-nonzero}
  If $\phi \co G_1 \to G_2$ is an essentially surjective map between
  marked elastic graphs and $0 \ne \ell \in \Len(\Gamma_2)$, then
  $\Iter_\phi(\ell) \ne 0$.
\end{lemma}

\begin{proof}
  Let $e_2$ be an edge of $G_2$ with $\ell(e_2) \ne 0$ and let
  $f \co G_1 \to (\Gamma_2, \ell)$ be the harmonic representative.
  Since $\phi$ is essentially surjective, there must be an edge $e_1$
  of $G_1$ with so that $f$ is not constant on~$e_1$ and $f(e_1)$
  intersects~$e_2$. Set $v = N_{[\phi]}(D_1^{-1}(H_{[\phi]}(\ell)))$
  as in Algorithm~\ref{alg:iter}. By Equation~\eqref{eq:weights-sum},
  $v(e_2) \ge \abs{f'(e_1)} > 0$.
\end{proof}

As a result of linearity on rays and Lemma~\ref{lem:iter-nonzero},
$\Iter_\phi$ descends to a projective map
$P\Iter_\phi \co P\Len(\Gamma_2)\righttoleftarrow$. Since
$P\Len(\Gamma_2)$ is a compact ball, there a fixed point.

The iteration is parallel to the
iteration on vector spaces, with
\begin{align*}
  V &\leftrightarrow \Len(G_2) & W &\leftrightarrow \Len(G_1)\\
  v_i &\leftrightarrow f_i\text{ or }\ell_i &
     w_i &\leftrightarrow g_i\text{ or }m_i\\
  v_i^* &\leftrightarrow c_i\text{ or }v_i &
     w_i^* &\leftrightarrow d_i\text{ or }w_i.
\end{align*}

\begin{remark}
  The only computationally-expensive step in Algorithm~\ref{alg:iter} is
  Step~(\ref{item:step-harmonic}), finding the harmonic equilibrium.
\end{remark}

\begin{lemma}\label{lem:DR-ER-increase}
  Let $[\phi] \co G_1 \to G_2$ be an essentially surjective homotopy
  class of maps between marked elastic graphs.
  $\Iter_\phi$ does not decrease the objective
  functions for
  both Dirichlet energy and extremal length: with the notation from
  Algorithm~\ref{alg:iter}, for $i \ge 0$,
  \begin{align*}
   \DR(\ell_i) &\le \DR(\ell_{i+1}) & \ER(w_i) &\le \ER(w_{i+1})
  \end{align*}
  If either inequality is an equality, then $[\ell_i]$ is a
  fixed point of $P\Iter_\phi$.
\end{lemma}

\begin{proof}
  This is parallel to Lemma~\ref{lem:objective-increase}.
  Tightness of the rows in Diagram~\eqref{eq:tight-iteration} tells us
  \begin{align*}
    \Dir[g_i]\EL[d_i] &= \ell[g_i \circ d_i]^2
       = \ell[f_i \circ\phi\circ d_i]^2 = \ell[f_i\circ c_i]^2\\
    \Dir[f_{i+1}]\EL[c_i] &= \ell[f_{i+1} \circ c_i]^2
       = \ell[f_{i+1}\circ\phi\circ d_{i}]^2 = \ell[g_{i+1} \circ d_{i}]^2\\
    \frac{\DR(\ell_i)}{\DR(\ell_{i+1})}
      =\frac{\Dir[g_i]}{\Dir[f_i]} \frac{\Dir[f_{i+1}]}{\Dir[g_{i+1}]}
      &= \frac{\ell[f_i\circ c_{i}]^2}{\Dir[f_i]\EL[d_{i}]}
          \frac{\ell[g_{i+1}\circ d_{i}]^2}{\Dir[g_{i+1}]\EL[c_{i}]}\\[1pt]
      &\le \frac{\Dir[f_i]\EL[c_{i}]}{\Dir[f_i]\EL[d_{i}]}
          \frac{\Dir[g_{i+1}]\EL[d_{i}]}{\Dir[g_{i+1}]\EL[c_{i}]}=1.
  \end{align*}
  If we have equality, then
  $\ell[f_i \circ c_i]^2 = \Dir[f_i]\EL[c_i]$ and so
  $\shortseq{W_i}{c_i}{G_2}{f_i}{K_i}$ is a tight sequence, in
  addition to $\shortseq{W_i}{c_i}{G_2}{f_{i+1}}{K_{i+1}}$.
  In a tight sequence $\shortseq{W}{c}{G}{f}{K}$,
  Equation~\eqref{eq:sm-CGK} is an equality.
  The Cauchy-Schwarz
  inequality used in its proof then implies that $\abs{f'}$ is proportional
  to~$n_c$, which in the present case says that
  $\ell_i$ and $\ell_{i+1}$ must be multiples of each other, as
  desired for the last statement.

  The proof for the inequality on $\ER$ is parallel:
  \begin{align*}
    \frac{\ER(w_i)}{\ER(w_{i+1})} =
    \frac{\EL[c_i]}{\EL[d_i]}\frac{\EL[d_{i+1}]}{\EL[c_{i+1}]} &=
      \frac{\ell[g_{i+1}\circ d_i]^2}{\EL[d_i]\Dir[f_{i+1}]}
      \frac{\ell[f_{i+1}\circ c_{i+1}]}{\EL[c_{i+1}]\Dir[g_{i+1}]} \le 1.
   \qedhere
  \end{align*}
\end{proof}

Now consider the lengths in $[\ell] \in P\Len(\Gamma_2)$ that maximize
$\DR(\ell)$. By
Lemma~\ref{lem:DR-ER-increase}, $[\ell]$ is a fixed point for
$P\Iter_\phi$. We have a partial converse for interior fixed points.

\begin{proposition}\label{prop:fixed-lambda}
  Let $[\phi] \co G_1 \to G_2$ be an essentially surjective homotopy
  class of maps between marked elastic graphs and suppose
  $[\ell]\in P\Len^+(\Gamma_2)$ is an interior fixed point of
  $P\Iter_\phi$ with multiplier~$\lambda$. Let $K = (\Gamma_2,\ell)$, let $f \co G_2 \to K$
  be the identity on the level of graphs, and let $g \co G_1 \to K$ be
  a harmonic representative of $\phi \circ f$. Then
  $f^{-1} \circ g \co G_1 \to G_2$ is $\lambda$-filling.
\end{proposition}

\begin{proof}
  At the fixed point, the tight sequences from
  Diagram~\eqref{eq:tight-iteration} collapse (up to scale) to
  \begin{equation}\label{eq:tight-fixed}
  \mathcenter{\begin{tikzpicture}[x=4em, y=35pt]
    \node (W) at (0,-0.5) {$W$};
    \node (G1) at (1,-1) {$G_1$};
    \node (G2) at (2,0) {$G_2$};
    \node (K) at (3,-0.5) {$K$};
    \draw [->] (W) -- node[above,cdlabel]{c} (G2);
    \draw [->] (G2) -- node[above,cdlabel]{f} (K);
    \draw [->] (W) -- node[above,pos=0.7,cdlabel]{d} (G1);
    \draw [->] (G1) -- node[above=-1pt,cdlabel]{g} (K);
    \draw [dashed, ->] (G1) -- node[left,cdlabel]{[\phi]} (G2);
  \end{tikzpicture}}
  \end{equation}
  where the top and bottom are tight sequences and the two triangles
  commute up to homotopy.
  By hypothesis, $K$ is a positive length graph and $f$ is
  invertible. Set $\psi = f^{-1} \circ g$,
  so $\psi \in [\phi]$. We will show that $\psi$ is
  $\lambda$-filling.

  We have chosen $\ell \in \Len(\Gamma_2)$. Choose $m$,
  $w$, and $v$ so that
  \begin{equation}
  \begin{aligned}
    m &= H_{[\phi]}(\ell) \in \Len(\Gamma_1)\\
    w &= D_1^{-1}(m) \in \Wgt(\Gamma_1)\\
    \lambda v &= N_{[\phi]}(w) \in \Wgt(\Gamma_2)\\
    \ell &= D_2(v) \in \Len(\Gamma_2).
  \end{aligned}\label{eq:fixed-pt-setup}
  \end{equation}
  The weights~$w$ and~$v$ are balanced as in
  Algorithm~\ref{alg:iter},
  so we have balanced strip graphs
  $S_1 = (\Gamma_1,w,m)$ and $S_2 = (\Gamma_2,v,\ell)$ compatible with
  $\alpha_1$ and~$\alpha_2$. The definition
  of $H_{[\phi]}$ ensures that $\psi$ preserves lengths, and the fact
  that $\psi$ is taut ensures that it scales weights by a
  factor of~$\lambda$.
\end{proof}

\begin{corollary}
  Any fixed point $[\ell]$ for $P\Iter_\phi$ in the interior of
  $P\Len(\Gamma_2)$ is a global maximum for $\DR(\ell)$.
\end{corollary}

\begin{proof}
  Immediate from Propositions~\ref{prop:fixed-lambda}
  and~\ref{prop:lambda-filling-emb}.
\end{proof}

\begin{remark}\label{rem:Dir-deriv}
  The maps in $\Iter_\phi$ can be given an interpretation in terms of
  derivatives. Specifically, let
  $\Dir_{[\phi]} \co \Len(\Gamma_2) \to \RR$ be Dirichlet energy as a function of lengths from
  Proposition~\ref{prop:dirichlet-continuous}.
  Then $\Dir_{[\phi]}$ is $C^1$ with derivative given by
  \[
  d \Dir_{[\phi]}(\ell) = 2\cdot N_{[\phi]}(D_{G_1}^{-1}(H_{[\phi]}(\ell)))
    \in \Wgt(\Gamma_2) = \Len(\Gamma_2)^*
  \]
  Recall from the
  introduction that, physically, tension in a spring is the
  derivative of energy as the length varies.
  Thus $d\Dir_{[\phi]}(\ell)$ gives the total tension in each edge
  of~$\Gamma_2$.
\end{remark}

\subsection{Behavior on the boundary}
\label{sec:partial-lambda-filling}
We now turn to boundary fixed points of
$P\Iter_\phi$.
Boundary fixed points need not be a global maxima of $\DR$.

\begin{example}
  In Example~\ref{examp:partial-filling}, $P\Len(\Gamma_1)$ is an
  interval, and both endpoints of the interval are fixed points for
  $P\Iter_\phi$. One endpoint is attracting (and the maximum of~$\DR$)
  and the other is repelling (and the minimum of~$\DR$).
\end{example}

However, we
can still extract much useful
structure from a boundary fixed point.

\begin{proposition}\label{prop:fixed-boundary}
  Let $[\phi] \co (\Gamma_1,\alpha_1) \to (\Gamma_2,\alpha_2)$ be an
  essentially surjective
  homotopy class of maps between marked elastic graphs, and suppose
  that $[\ell] \in \partial P\Len(\Gamma_2)$ is a fixed point of
  $P\Iter_\phi$ with multiplier~$\lambda > 0$. Then there are
  decompositions of $\Gamma_i$ into complementary
  subgraphs $\Gamma_i = \Delta_i \cup \Sigma_i$, along with a
  map $\psi \co \Gamma_1 \to \Gamma_2$ so that
  \begin{enumerate}
  \item $\psi \in [\phi]$;
  \item the edges of $\Sigma_2$ are those edges~$e$ with
    $\ell(e) = 0$;
  \item\label{item:psi-edges-map} $\psi(\Delta_1) \subset \Delta_2$ and $\psi(\Sigma_1) \subset \Sigma_2$; and
  \item the restriction of $\psi$ to a map from $(\Delta_1,\alpha_1)$
    to $(\Delta_2,\alpha_2)$ is $\lambda$-filling.
  \end{enumerate}
\end{proposition}

The key thing missing from Proposition~\ref{prop:fixed-boundary} in
asserting that $\psi$ is partially $\lambda$-filling is
condition~(\ref{item:scale-less}) of
Definition~\ref{def:partial-filling}. (Indeed
Example~\ref{examp:partial-filling}, shows that we cannot get a
partially $\lambda$-filling map from an arbitrary boundary fixed
point.)

\begin{proof}
Define $m$,
$w$, and $v$ from~$\ell$ by Equation~\eqref{eq:fixed-pt-setup}.
There are tight sequences similar to Diagram~\eqref{eq:tight-fixed},
with \emph{weak} harmonic maps $(\wt f, f)$ and
$(\wt g, g)$:
  \begin{equation}\label{eq:tight-fixed-bdy}
  \mathcenter{\begin{tikzpicture}[x=4em, y=35pt]
    \node (W) at (0,-0.5) {$W$};
    \node (G1) at (1,-1) {$\Gamma_1$};
    \node (G2) at (2,0) {$\Gamma_2$};
    \node (K) at (4,0) {$K$};
    \node (Ks) at (4,-1) {$K^*$.};
    \draw [->] (W) -- node[above,cdlabel]{c} (G2);
    \draw [->] (G2) -- node[pos=0.5,above=-1pt,cdlabel]{\wt f} (K);
    \draw [->] (W) -- node[above,cdlabel,pos=0.7]{d} (G1);
    \draw [->] (G1) -- node[pos=0.35,above=-1pt,cdlabel]{\wt g} (K);
    \draw [dashed, ->] (G1) -- node[left,cdlabel]{[\phi]} (G2);
    \draw [->] (K) -- node[right,cdlabel]{\kappa} (Ks);
    \draw [->] (G2) -- node[pos=0.85,above=-1pt,cdlabel]{f} (Ks);
    \draw [->] (G1) -- node[pos=0.6,above=-1pt,cdlabel]{g} (Ks);
  \end{tikzpicture}}
  \end{equation}
Choose $\wt f$ to be the identity (since $\Gamma_2$ and $K$ have the
same underlying graph).
Choose $\wt g$ so that it is vertex precise and $\wt g \circ d$ is
taut, as
guaranteed by Proposition~\ref{prop:weak-harmonic-vp}.

Let $\Sigma_2 \subset \Gamma_2$ be the null subgraph of~$K$.
Let
$\Delta_2$ be the complementary subgraph. Similarly let $\Sigma_1$ be
the subgraph of $\Gamma_1$ whose edges are those on which $w$ and~$m$
are~$0$ and let $\Delta_1$ be the complementary subgraph. Set
$\psi = \wt f^{-1} \circ \wt g$; we must check $\psi$ satisfies the
conditions.

The first two conditions are immediate. To check
condition~\eqref{item:psi-edges-map}, pick a regular value~$y$
for~$\psi$ on an edge~$e_2$
of~$\Gamma_2$ and a preimage $x$ on an edge~$e_1$ of~$\Gamma_1$. We
must show that
$y \in \Sigma_2 \Leftrightarrow x \in \Sigma_1$. If $y \in \Sigma_2$,
then since $\wt g \circ d$ is taut, we have
$n^w_{\wt g}(y) = n^w_{[\phi]}(y) = \lambda v(e_2) = 0$. We therefore must
have $w(e_1) = 0$, so $x \in \Sigma_1$. Conversely, if $x \in \Sigma_1$,
then by definition $\ell(e_1) = 0$. Since $g$ is length-preserving as a
map from $(\Gamma_1, \ell)$ to $K^*$ and the lift $\wt g$ is
vertex-precise, $\wt g(e_1)$ lies inside the null subgraph
of~$K$ (since by assumption $\wt g$ is not constant on~$e_1$), so $y \in \Sigma_2$.

Let $\psi|_\Delta$ be the restriction of $\psi$ to a map from
$\Delta_1$ to $\Delta_2$. It remains to show that
\[
\psi|_\Delta \co (\Delta_1, w|_{\Delta_1}, m|_{\Delta_1}) \longrightarrow (\Delta_2, v|_{\Delta_2}, \ell|_{\Delta_2})
\]
is $\lambda$-filling as a map between strip graphs. As in
Proposition~\ref{prop:fixed-lambda}, the
definition of~$m$ as $m(e) = \ell(g(e)) = \ell(\psi(e))$ ensures that
$\psi$ is length-preserving. We chose $\wt g \circ d$ to be taut, and by
restricting first the domain to $\Delta_1$ by
Lemma~\ref{lem:taut-subdomain} and then the range to~$\Delta_2$ by
Lemma~\ref{lem:taut-subrange}, we see
that $\psi|_\Delta$ is taut. The
definition of~$v$ then tells us that $\psi|_\Delta$ scales weights
by~$\lambda$.
\end{proof}

In the setting of Proposition~\ref{prop:fixed-boundary}, let
$\psi|_\Sigma$ be the
restriction of~$\psi$ to a map between the collapsed subgraphs $\Sigma_1$
and~$\Sigma_2$. For $i=1,2$, let $P'_i = \Sigma_i \cap \Delta_i$ be
the vertices shared between the two complementary subgraphs, and let
$P_i$ be the
union of $P_i'$ with the vertices of $\Sigma_i$ that were already
marked. We view $\psi|_\Sigma$ as a map of marked elastic graphs
$(\Sigma_1,P_1,\alpha_1) \to (\Sigma_2, P_2,\alpha_2)$, and consider
the problem of finding $\Emb[\psi \vert_\Sigma]$ in its own right.

\begin{proposition}\label{prop:DR-boundary-max}
  In the setting of Proposition~\ref{prop:fixed-boundary},
  suppose that there
  is a fixed point $[\ell_0]$ of $P\Iter_{\psi\vert_\Sigma}$ with
  multiplier~$\lambda_0 > \lambda$. Then $\ell$ is not a local maximum
  of $\DR(\ell)$.
\end{proposition}

\begin{proof}
  We will show that $\DR(\ell + \epsilon \ell_0) > \DR(\ell)$ for
  sufficiently small~$\epsilon$.
  From $\ell_0 \in \Len(\Sigma_2)$,
  construct $m_0 \in \Len(\Sigma_1)$, $w_0 \in \Wgt(\Sigma_1)$,
  and $v_0 \in \Wgt(\Sigma_2)$ by
  Equation~\eqref{eq:fixed-pt-setup}.  Extend $\ell_0$, $m_0$,
  $w_0$, and $v_0$ to $\Gamma_i$ by setting them to be zero on edges in $\Delta_i$.
  Let
  $\psi_0 \co (\Sigma_1, P_1) \to (\Sigma_2, P_2)$ be the map
  constructed from~$\ell_0$ by
  Proposition~\ref{prop:fixed-boundary}. In particular, $\psi_0$ is
  harmonic as a map from $(\Sigma_1, P_1, \alpha_1)$ to
  $(\Sigma_2, P_2, \ell_0)$.  Define a new map
  $h \co \Gamma_1 \to \Gamma_2$ by
  \begin{equation}\label{eq:assemble-map}
  h(x) \coloneqq
  \begin{cases}
    \psi(x) & x \in \Delta_1\\
    \psi_0(x) & x \in \Sigma_1.
  \end{cases}
  \end{equation}
  Since we pinned
  $\Delta_i \cap \Sigma_i$ in $\psi \vert_\Sigma$, the map $h$ is continuous. By
  construction, $h$ is weakly harmonic as a map from
  $(\Gamma_1, \alpha_1)$ to $(\Gamma_2, \ell)$. We claim that $h$ is
  also harmonic if we perturb~$\ell$.
  For small $\epsilon$, consider the modified lengths
  $\ell_\epsilon = \ell + \epsilon \ell_0 \in \Len(\Gamma_2)$. Let
  $h_\epsilon$ be $h$ considered as a map from
  $(\Gamma_1,\alpha_1)$ to $(\Gamma_2, \ell_\epsilon)$.

  \begin{claim}
    For $\epsilon$ sufficiently small, $h_\epsilon$ is a harmonic map.
  \end{claim}

  \begin{proof}
    For simplicity, we suppose that $\ell_0$ is non-zero on every
    edge, so that $(\Gamma_2, \ell_\epsilon)$ is a length graph and
    $h_\epsilon$ is an ordinary map (not weak). The case when $\ell_0$
    has some zeroes can be treated by induction.

    The tension weight of~$h$ is $w \in \Wgt(\Gamma_1)$. Let
    $w_\epsilon$ be the tension weight
    of~$h_\epsilon$. Concretely,
    \[
    w_\epsilon(e) =
    \begin{cases}
      w(e) & e \in \Edges(\Delta_1)\\
      \epsilon w_0(e) & e \in \Edges(\Sigma_1).
    \end{cases}
    \]
    We must check that $h_\epsilon$ is still taut as a map from
    $(\Gamma_1, w_\epsilon)$ to~$\Gamma_2$. By
    Proposition~\ref{prop:harmonic-tt},
    we must check that $w_\epsilon$ satisfies the train-track triangle
    inequalities at the vertices of~$\Gamma_1$.

    For vertices in $\Delta_1 \setminus \Sigma_1$ and
    $\Sigma_1 \setminus \Delta_1$, the triangle inequalities follow
    from the fact that $\psi$ and $\psi_0$ are harmonic,
    respectively. For vertices in $\Delta_1 \cap \Sigma_1$, the
    triangle inequalities follow from Lemma~\ref{lem:eps-triangle}
    below, where the $a_i$ are the weights of the incident groups of
    edges of~$\Delta_1$ and the $b_i$ are the weights of the
    incident groups of edges of~$\Sigma_1$.
  \end{proof}

  \begin{lemma}\label{lem:eps-triangle}
    If $(a_1,\dots,a_n) \in \RR_+^n$ satisfies the triangle inequalities
    and $(b_1,\dots,b_m) \in \RR_+^m$ is another vector, then, for
    all $\epsilon$ sufficiently small,
    \[
    (a_1,\dots,a_n,\epsilon b_1, \dots, \epsilon b_m)
    \]
    satisfies the triangle inequalities.
  \end{lemma}
  \begin{proof}
    Elementary.
  \end{proof}

  Returning to the proof of Proposition~\ref{prop:DR-boundary-max},
  since $h_\epsilon$ is harmonic, if $f$ and~$g$ are the
  harmonic maps to $(\Gamma_2, \ell)$ and
  $f_0$ and~$g_0$ are the harmonic maps to $(\Sigma_2, \ell_0)$,
  we have
  \[
    \DR(\ell_\epsilon)
    = \frac{\Dir(h_\epsilon)}{\Dir(f) + \epsilon^2 \Dir(f_0)}
     = \frac{\Dir(g)+\epsilon^2\Dir(g_0)}{\Dir(f) + \epsilon^2\Dir(f_0)}
     > \frac{\Dir(g)}{\Dir(f)} = \lambda = \DR(\ell)
  \]
  using the assumption that $\lambda_0 = \Dir(g_0)/\Dir(f_0) > \lambda$.
\end{proof}

We can now prove the promised existence of a partially $\lambda$-filling map.

\begin{proof}[Proof of Proposition~\ref{prop:partial-filling}]
  First, if $[\phi]$ is not essentially surjective, we can restrict to
  the essential image of $[\phi]$, the image of a taut
  representative of $[\phi]$ with respect to any weight structure
  on~$G_1$. After doing this, $P\Iter_\phi$ is defined by
  Lemma~\ref{lem:iter-nonzero}.

  We proceed by induction on the size of $\Gamma_1$.  Given the
  homotopy class $[\phi]$, let $\ell \in \Len(\Gamma_1)$ be a global
  maximum of $\DR$ (which exists since $P\Len(\Gamma_1)$ is
  compact). 
  By Lemma~\ref{lem:DR-ER-increase}, $\ell$ is a fixed point
  of $P\Iter_\phi$; let $\lambda$ be its multiplier. If $\ell$ is in
  $\Len^+(\Gamma_1)$, we are done by
  Proposition~\ref{prop:fixed-lambda}. Otherwise, consider the
  subgraphs $\Delta_i$ and $\Sigma_i$ given by
  Proposition~\ref{prop:fixed-boundary}, with restricted maps
  $\psi\vert_\Delta$ (which is $\lambda$-filling) and
  $\phi \vert_\Sigma$. Since $\Sigma_1$ is a proper subgraph of~$\Gamma_1$,
  by induction we can find a partially $\lambda_1$-filling map
  $\psi\vert_\Sigma \in [\phi_\Sigma]$ for some $\lambda_1 \ge 0$.

  Now assemble $\psi\vert_\Delta$ and $\psi\vert_\Sigma$ to a single map
  $\psi \co \Gamma_1 \to \Gamma_2$ by Equation~\eqref{eq:assemble-map}. Then $\psi$ is $\lambda$-filling
  on~$\Delta_1$ and partially $\lambda_1$-filling on~$\Sigma_1$. By
  Proposition~\ref{prop:DR-boundary-max},
  $\lambda_1 \le \lambda$, so $\psi$ is
  partially $\lambda$-filling on all of~$\Gamma_1$.
\end{proof}

\begin{remark}
  The map constructed above has a stronger ``layered'' structure,
  where $\Gamma_1$ and $\Gamma_2$ are divided into layers, each with
  its own filling constant. Specifically, there are properly nested
  subgraphs
  \begin{align*}
    \Gamma_1 &= \Sigma_1^0 \supsetneq \Sigma_1^1 \supsetneq \dots
               \supsetneq \Sigma_1^n\\
    \Gamma_2 &= \Sigma_2^0 \supsetneq \Sigma_2^1 \supsetneq \dots
               \supsetneq \Sigma_2^n,
  \end{align*}
  a map $\psi \co \Gamma_1 \to \Gamma_2$ in $[\phi]$, and a sequence
  of filling constants $\lambda_1 > \dots > \lambda_n$, with the
  following properties.
  \begin{enumerate}
  \item $\psi$ preserves $\Sigma^i$: for $1 \le i \le n$,
    $\psi(\Sigma_1^i) \subset \Sigma_2^i$.
  \item $\psi$ preserves $\Sigma^{i-1} \setminus \Sigma^i$: for
    $k\in\{1,2\}$ and $0 < i \le n$, let $\Delta_k^i$ be the
    complementary subgraph to $\Sigma^i_{k}$ in $\Sigma^{i-1}_k$. Then
    $\psi(\Delta_1^i) \subset \Delta_2^i$.
  \item $\psi$ is $\lambda_i$-filling on $\Delta^i$: for
    $0 \le i \le n$, let $\psi_i$ be the restriction of $\psi$ to a
    map from $\Delta_1^i$ to $\Delta_2^i$, where for $i > 0$ we
    additionally mark the vertices in
    $\Delta_1^i \cap \Delta_1^{i-1}$ and in
    $\Delta_2^i \cap \Delta_2^{i-1}$. Then $\psi_i$ is
    $\lambda_i$-filling.
  \item $\psi|_{\Sigma^n_1}$ is constant.
  \end{enumerate}
\end{remark}

\subsection{General targets}
\label{sec:general-targets}

We turn to Theorem~\ref{thm:emb-sf-gen}, allowing more general
length space targets. First we need to generalize
Equation~\eqref{eq:sm-GGK} to this setting.

\begin{lemma}
  Let $G_1$ and $G_2$ be elastic graphs, let $X$ be a length space,
  let $\phi \co G_1 \to G_2$ be a piecewise-linear map, and
  let $f \co G_2 \to X$ be a Lipschitz map. Then
  \[
  \Dir(f \circ \phi) \le \Emb(\phi) \Dir(f).
  \]
\end{lemma}

\begin{proof} We compute
  \begin{align*}
    \Dir(f \circ \phi) &= \int_{x\in G_1} \abs{(f \circ \phi)'(x)}^2\,dx\\
      &= \int_{y\in G_2} \biggl( \sum_{x \in \phi^{-1}(y)} \abs{\phi'(x)}\biggr)\,
        \abs{f'(y)}^2\,dy\\
      &\le \Emb(\phi) \Dir(f),
  \end{align*}
  In the second line we do a change of variables from $G_1$
  to~$G_2$, using $dx = \abs{\phi'(x)}\,dy$. (Any portions of $G_1$
  where $\phi$ is constant and so $\phi^{-1}(y)$ is uncountable do not
  contribute to the integrals.)
  In the last line
  we use
  $\int \abs{a}\cdot \abs{b}\,dy \le \esssup \abs{a}\cdot \int \abs{b}\,dy$.
\end{proof}

\begin{proof}[Proof of Theorem~\ref{thm:emb-sf-gen}]
  Suppose $\Emb(\phi)$ is minimal within the homotopy class $[\phi]$
  and that $\Dir(f)$ is within a multiplicative factor of~$\epsilon$ of the infimum. Then
  \[
  \Dir[f \circ \phi] \le \Dir(f \circ \phi) \le \Emb(\phi) \Dir(f) \le
    \Emb[\phi] \Dir[f] (1+\epsilon).
  \]
  Since we can choose $\epsilon$ as small as we like, this gives one
  inequality of the desired equality. The other direction comes
  from Theorem~\ref{thm:emb-sf}.
\end{proof}

\subsection{Algorithmic questions}
\label{sec:convergence}

Given a homotopy class $[\phi] \co G_1 \to G_2$ of maps between marked
elastic graphs, we have proved that $\Iter_\phi$ has a projectively
fixed set of lengths $\ell \in \Len(G_2)$ maximizing $\DR(\ell)$ and
giving a partially $\lambda$-filling map in~$[\phi]$. The lengths
maximizing $\DR(\ell)$ need not be unique.

\begin{example}
  Let $G_1$ and $G_2$ both be the join of two circles,
  as in Example~\ref{examp:partial-filling}, with all elastic constants
  equal to~$1$, and let $\phi$ be the identity map. Then $\Iter_\phi$
  is the identity and $\DR$ is constant on $\Len(G_2)$.
\end{example}

\begin{question}
  What is the structure of the subset of $\Len(G_2)$ on which $\DR$
  reaches its maximum value? For instance, is it a convex subset of
  $\Len(G_2)$? Can it be a proper subset of the interior of $\Len(G_2)$?
\end{question}

As for convergence, we would like to say that
Algorithm~\ref{alg:iter} works, in the sense that iterating
$P\Iter_\phi$ always converges to a maximum of
$\DR$ (which also computes $\Emb[\phi]$). The presence of
extra fixed points of
$\Iter_\phi$ means that this does not always happen. (In dynamical
terms, $\DR$ is not decreased by $\Iter_\phi$, and not necessarily
increased; so $\DR$ is not quite a Lyapunov function for this discrete
dynamical system.) But we can make some statements.

\begin{proposition}\label{prop:iter-converge-fixed}
  Algorithm~\ref{alg:iter} gives a sequence of lengths
  $\ell_i \in \Len(G_2)$ that converge projectively to a fixed point
  for $P\Iter_\phi$.
\end{proposition}

\begin{proof}
  By Lemma~\ref{lem:DR-ER-increase}, $\DR(\ell_i)$ does not decrease
  and has an upper bound; thus $\DR(\ell_i)$ has a limit, and the
  $[\ell_i]$ have an accumulation point $[\ell_\infty]$ with
  $\DR(\Iter_\phi(\ell_\infty)) = \DR(\ell_\infty)$. By
  Lemma~\ref{lem:DR-ER-increase} again,
  $[\ell_\infty]$ is a fixed point of $P\Iter_\phi$, and
  therefore the $[\ell_i]$ limit to $[\ell_\infty]$, without the need
  to pass to a subsequence.
\end{proof}

\begin{lemma}\label{lem:iter-crit-discrete}
  For $[\phi] \co G_1 \to G_2$ a homotopy class of maps between marked
  elastic graphs,
  the set $\{\,\DR(\ell) \mid \text{$\ell$ a fixed point of $P\Iter_\phi$}\,\}$
  is finite.
\end{lemma}

\begin{proof}
  Let $\ell \in \Len(G_2)$ be a projective fixed point for
  $\Iter_\phi$ with multiplier~$\lambda$.
  Proposition~\ref{prop:fixed-boundary} gives a
  $\lambda$-filling map on subgraphs
  $\phi\vert_\Delta \co \Delta_1 \to \Delta_2$.
  By Proposition~\ref{prop:lambda-filling-emb},
  $\lambda=\Emb[\phi \vert_\Delta]$ and thus depends only
  on the subgraphs~$\Delta_i$. Since there are only finitely many
  subgraphs, we are done.
\end{proof}

\begin{proposition}
  For $\phi \co G_1 \to G_2$ a homotopy class of maps between marked
  elastic graphs, there is an open subset of $P\Len(G_2)$ on which
  $P\Iter_\phi$ converges to a maximum of~$\DR$.
\end{proposition}

\begin{proof}
  Let $\lambda$ be the maximum value of $\DR$ on $P\Len(G_2)$. By
  Lemma~\ref{lem:iter-crit-discrete}, there is an $\epsilon > 0$ so
  that there are no fixed points of $P\Iter_\phi$ in
  $\DR^{-1}(\lambda-\epsilon,\lambda)$. Then by
  Proposition~\ref{prop:iter-converge-fixed}, $P\Iter_\phi$ converges
  to a maximum of $\DR$ on $\DR^{-1}(\lambda-\epsilon,\lambda]$.
\end{proof}

\begin{question}\label{quest:converge}
  Does Algorithm~\ref{alg:iter} converge to a maximum of~$\DR$
  for an open dense set of initial data?
\end{question}

A few words are in order on why
Question~\ref{quest:converge} is not as easy as it may appear. If
$[\ell_1]$ is a fixed
point of $P\Iter_\phi$ with $\DR(\ell_1) < \Emb[\phi]$, then by
Proposition~\ref{prop:DR-boundary-max}
it is not a local maximum of~$\DR$. If $\Iter_\phi$ were
linear, that would imply the set attracted to~$[\ell_1]$
has empty interior. Since $\Iter_\phi$ is only PL, the
situation is more complicated. For instance, $\Iter_\phi$ can map an
open subset of $\Len^+(G_2)$ to a subset
$S \subset \bdy\Len(G_2)$,
since harmonic maps can generically map vertices to vertices.
Then $S$ could potentially be attracted to~$[\ell_1]$.

We can nevertheless improve Algorithm~\ref{alg:iter} to always find
$\Emb[\phi]$.

\begin{algorithm}\label{alg:Emb}
  Given an essentially surjective homotopy class
  $[\phi] \co G_1 \to G_2$ of maps between marked elastic graphs, to
  find $\Emb[\phi]$, pick arbitrary non-zero initial lengths
  $\ell_0 \in \Len(G_2)$ and iterate Algorithm~\ref{alg:iter} to get a
  sequence of lengths $\ell_i$. Since $\Iter_\phi$ is
  piecewise-linear, each set of lengths $\ell_i$ is in a closed cone
  of linearity $R_i \subset \Len(G_2)$. Since
  there are only finitely many domains of linearity, there must
  be some $i$ and~$k$ so that $R_{i+k} = R_i$.  By standard linear
  algebra techniques, we can see if $(\Iter_\phi)^k$ has a projective
  fixed point in $R_i$. By Proposition~\ref{prop:iter-converge-fixed},
  the $\ell_i$ converge to a projective fixed point, so eventually the
  linear algebra check will succeed, giving projectively fixed
  lengths~$\ell_\infty$ with multiplier $\lambda_\infty$.

  If $\ell_\infty$ is non-zero on every edge, we are done by
  Proposition~\ref{prop:fixed-lambda}. Otherwise, apply
  Proposition~\ref{prop:fixed-boundary} to extract a map
  \[
  \psi\vert_{\Sigma} \co (\Sigma_1, P_1, \alpha_1) \to
   (\Sigma_2, P_2, \alpha_2)
  \]
  between simpler graphs, with its own embedding energy
  $\lambda_\Sigma$ (which we can find by induction). If
  $\lambda_\Sigma < \lambda_\infty$, we have found a partially
  $\lambda_\infty$-filling map and are done. Otherwise, by
  Proposition~\ref{prop:DR-boundary-max}, we can find nearby lengths
  $\ell_0'$ with $\DR(\ell_0') > \DR(\ell_\infty) \ge \DR(\ell_0)$. In
  this case, repeat the algorithm, with $\ell_0'$ in place
  of~$\ell_0$. By Lemma~\ref{lem:iter-crit-discrete}, eventually we
  will find the true maximum value of $\DR(\ell)$, and thus
  the true value of $\Emb[\phi]$.
\end{algorithm}

In practice, Algorithm~\ref{alg:Emb} appears to run quickly, at least
for small examples. The additional steps to continue past a
projective fixed point have been unnecessary. Theoretically, there is
no reason to expect it to always perform well. In the closely
related case of pseudo-Anosov maps, Bell and Schleimer have
given examples where the analogous algorithm is
slow~\cite{BS17:NSDynamics}.


\appendix
\section{General graph energies}
\label{sec:energies}

As suggested by the notation in Definition~\ref{def:12inf-graphs}, the
energies of this paper fit into a more general framework. We start with a
notion of
$p$-conformal graphs, simultaneously generalizing weighted graphs,
elastic graphs,
and length graphs. There are several different perspectives. A
$p$-conformal graph can be viewed as
\begin{itemize}
\item a graph with a $p$-length $\alpha(e)$ on
  each edge;
\item an equivalence class of strip graphs $(\Gamma, w, \ell)$ under a
  rescaling operation; or
\item an equivalence class of spaces~$X$ with a length metric~$\ell$
  and measure $\mu$ under another rescaling operation.
\end{itemize}
We start with the metric
view, since it is most standard, although the formulas may
appear unmotivated.

\begin{definition}\label{def:p-graph}
  For $p \in (1,\infty]$, a \emph{$p$-conformal graph}
  $G^p = (\Gamma, \alpha)$ is a graph with a positive
  \emph{$p$-length} $\alpha(e)$ on each edge~$e$, which gives a
  metric. For $p=1$, a \emph{$1$-conformal graph} is a
  weighted graph.
\end{definition}

The weights associated to edges in a $1$-conformal graph behave
differently than the $p$-lengths for $p > 1$; for instance,
$p$-lengths add for two edges joined in series, while weights do not.
As a result, the formulas below have $p=1$ as a special case. (The
$p=1$ case where all the weights are equal to $1$ is a straightforward
limit of the $p > 1$ case; the weights $\alpha(e)$ become
irrelevant in the limit $p \to 1$.)

\begin{definition}\label{def:Epq-1}
  For $1 < p \le \infty$ and
  $f \co G^p \to K$ a PL map from a $p$-conformal graph to a
  length graph, define
  \begin{equation}\label{eq:Ep-1}
    E^p(f) \coloneqq
    \begin{cases}
      \norm{f'}_{p,G} & p > 1\\
      \int_G w(x) \abs{f'(x)}\,dx & p = 1.
    \end{cases}
  \end{equation}
  Here, we take the $L^p$ norm of $\abs{f'}$ by integration with respect
  to~$\alpha$, and use an additional
  subscript to
  make clear where the norm is being evaluated. The metrics on $G$
  and~$K$ are used to evaluate the derivative $\abs{f'}$, and the
  metric on~$G$ is also used in the integration in taking the $L^p$
  norm. Observe that the metric on~$G$ is irrelevant for $p=1$.

  If $f$ is
  constant-derivative and $1 < p < \infty$, then this becomes
  \begin{equation*}
    E^p(f) = \left(\sum_{e \in \Edges(G)} \frac{\ell(f(e))^p}{\alpha(e)^{p-1}}\right)^{1/p}.
  \end{equation*}
  (Compare Equation~\eqref{eq:dir-1}.)
  For $f \co W \to G^p$ a PL map from a weighted graph to a
  $p$-conformal graph, define
  \begin{equation}
    E_p(f) \coloneqq
    \begin{cases}
      \norm{n_f}_{\pdual, G} & p > 1\\
      \norm{n_f/w}_{\infty,G} & p = 1
    \end{cases}
  \end{equation}
  where $\pdual = p/(p-1)$ is the Hölder conjugate of~$p$. If $n_f$ is
  constant on each edge (as for taut maps, c.f.\
  Theorem~\ref{thm:maxflow-mincut}) and $1 < p < \infty$, then
  \[
    E_p(f) = \left( \sum_{e \in \Edge(G)} \alpha(e) n_f(e)^{\pdual} \right)^{1/\pdual}.
  \]
  (Compare Equation~\eqref{eq:el-curve}.)
  In general, for $f \co G^p \to H^q$ a PL map from a $p$-conformal graph to a
  $q$-conformal graph with $1 \le p \le q \le \infty$ with $1 < q$ and
  $p < \infty$, define
  \begin{align}
    \nonumber\Fill^p(f)  &\co H^q \to \RR_{\ge 0}\\
    \Fill^p(f)(y) &\coloneqq
                    \begin{cases}
                      \sum_{x \in f^{-1}(y)}\abs{f'(x)}^{p-1} & p > 1\\
                      \sum_{x \in f^{-1}(y)}w(x) & p = 1\\
                    \end{cases}\label{eq:Fillp}\\
    E^p_q(f) &\coloneqq \bigl(\norm{\Fill^p(f)}_{q/(q-p),H}\bigr)^{1/p}\label{eq:Epq}.\\
    \intertext{If $p<q$, this is}
    E^p_q(f) &= \left(\int_H \Fill^p(f)(y)^{\frac{1}{1-p/q}}\,d\alpha(y)\right)
                          ^{1/p-1/q}.
  \end{align}
  Energies of homotopy classes are defined as an infimum as usual:
  \begin{align*}
    E_p[f] &\coloneqq \inf_{g \in [f]} E_p(g)&
    E^p_q[f] &\coloneqq \inf_{g \in [f]}E^p_q(g)&
    E^p[f] &\coloneqq \inf_{g \in [f]} E^p(g).
  \end{align*}
\end{definition}

\begin{remark}
  As with the earlier energies, in each case $E^p_q(f)$ naturally
  extends to a wider class of graph maps than PL maps, with the
  regularity required depending on $p$ and~$q$.
\end{remark}

\begin{proposition}\label{prop:energy-defs-1}
  For a PL map~$f$, we have $E_p(f) = E^1_p(f)$ for $1 < p \le \infty$
  and $E^p(f) = E^p_\infty(f)$ for $1 \le p < \infty$.
\end{proposition}

\begin{proof}
  True by definition for $E^1_p$, and change of variables for $E^p_\infty$.
\end{proof}

In light of Proposition~\ref{prop:energy-defs-1}, we also define
\begin{align*}
E^\infty_\infty(f) &\coloneqq E^\infty(f)= \Lip(f)\\
E^1_1(f) &\coloneqq E_1(f)= \WR(f).
\end{align*}

\begin{proposition}\label{prop:Epq-local-sup}
  For $1 \le p \le q \le \infty$ and $\phi \co G^p \to
  H^q$ a PL map,
  \[
  E^p_q(\phi) = \sup_{f \co H \to K} \frac{E^p(f \circ \phi)}{E^q(f)}
  \]
  where the supremum runs over all metric graphs~$K$ and PL maps
  $f \co H^q \to K$ with non-zero energy. If $p < q < \infty$,
  equality holds exactly when $\abs{f'}$ is proportional to $\bigl(\Fill^p(\phi)\bigr)^{1/(q-p)}$.
\end{proposition}

Observe that Proposition~\ref{prop:Epq-local-sup} is about energies of
concrete maps, not about homotopy classes of maps.

\begin{proof} We omit the cases when $p=1$ and $q=\infty$. Otherwise, the
  energies $E^q(f)$ and $E^p(f \circ \phi)$ depend only on $\abs{f'}$, so
  we may as well assume that $K$ has the same underlying graph as
  $H$ (with varying metric) and $f$ is the identity as a graph map.
  We are thus picking $\abs{f'}$ as a piecewise\hyp constant function
  on~$H$.  We have
  \begin{align*}
    E^p(f \circ \phi)
      &= \biggl(\int_G \bigl\lvert f'(\phi(x))\bigr\rvert^p\, \abs{\phi'(x)}^p\,dx\biggr)^{1/p}\\
      &= \biggl(\int_H \bigl\lvert f'(y)\bigr\rvert^p\, \bigl(\Fill^p(\phi)(y)\bigr)\,dy\biggr)^{1/p}\\
      &\le \norm{f'}_q\,\bigl(\lVert\Fill^p(\phi)\rVert_{q/(q-p)}\bigr)^{1/p}\\
      &= E^q(f) E^p_q(\phi),
  \end{align*}
  using Hölder's inequality in the form
  $\norm{ab}_1 \le \norm{a}_{q/p}\norm{b}_{q/(q-p)}$. The equality
  statement follows from the equality condition in Hölder's inequality.
\end{proof}

\begin{remark}
  It is also true that
  \[
  E^p_q(\phi) = \sup_{c \co W \to G} \frac{E_q(\phi \circ c)}{E_p(c)}
  \]
  where the supremum runs over all maps $c \co W \to G$ from a
  weighted graph to~$G$ with
  non-zero energy (not necessarily taut). Here, equality holds when $n_c$ is
  proportional to
  $\abs{\phi'}^{p-1} \bigl(\Fill^p(\phi)\bigr)^{(p-1)/(q-p)}$.
\end{remark}

\begin{proposition}\label{prop:energy-submult}
  Given $1 \le p \le q \le r \le \infty$ and a sequence
  $\shortseq{G^p_1}{f}{G^q_2}{g}{G^r_2}$ of maps between a marked
  $p$-conformal graph $G_1$, a marked $q$-conformal graph $G_2$, and a
  marked $r$-conformal graph $G_3$,
  \begin{align*}
    E^p_r(g \circ f) &\le E^p_q(f)E^q_r(g)\\
    E^p_r[g \circ f] &\le E^p_q[f]E^q_r[g].
  \end{align*}
\end{proposition}

\begin{proof}[Proof sketch]
  The first equation is an
  immediate consequence of Proposition~\ref{prop:Epq-local-sup}. The
  second equation follows as in the proof of
  Proposition~\ref{prop:sub-mult}.
\end{proof}

Taut maps automatically minimize $E_p$ within their homotopy class, as
in Proposition~\ref{prop:taut-minimal}. We can also minimize $E^p$
within a homotopy class.

\begin{proposition}\label{prop:p-harmonic}
  Let $[f] \co G^p \to K$ be a homotopy class of maps from a marked
  $p$-conformal graph to a marked length graph. Then there is a
  constant-derivative PL map $g \in [f]$ so that $E^p(g) = E^p[f]$.
\end{proposition}

\begin{proof}[Proof sketch]
  For $p = \infty$, this is Theorem~\ref{thm:lipschitz-stretch}. For
  $p < \infty$, the form of $E^p$ guarantees that an energy minimizer
  will be constant-derivative. As in Theorem~\ref{thm:harmonic-min},
  this reduces the space of
  possibilities to minimizing over a compact space for each of
  finitely many combinatorial types.
\end{proof}

A map that minimizes $E^p(f)$ within $[f]$ is called a
\emph{$p$-harmonic map}. Theorem~\ref{thm:harmonic-min-weak} may also be
extended to give a local characterization of $p$-harmonic maps,
including cases where the target is weak. Specifically, $[f] \co G^p \to K$ is
$p$-harmonic iff the map $W_f \to K$
is taut, where $W_f$ is $G$ with $p$-tension weights $w(x) =
\abs{f'(x)}^{p-1}$.

There are also versions of stretch factors: for
  $[\phi] \co G^p \to H^q$ a homotopy class of marked maps, define
  \begin{align}
    \SFfrom{}^p_q[\phi] &\coloneqq \sup_{[f] \co H \to K} \frac{E^p[f \circ\phi]}{E^q[f]}
      \label{SFfrom}\\
    \SFto{}^p_q[\phi] &\coloneqq \sup_{[c] \co W \to G} \frac{E_q[\phi\circ c]}{E_p[c]},
      \label{SFto}
  \end{align}
  where in \eqref{SFfrom} we maximize over marked length graphs~$K$ and
  PL maps $f \co H^q \to K$ and in \eqref{SFto} we maximize over all
  marked weighted graphs (or multi-curves) on~$G^p$.

\begin{theorem}\label{thm:energy-sf}
  For $1 \le p \le q \le \infty$ and $[\phi] \co G^p \to H^q$ a homotopy
  class of maps from a marked $p$-conformal graph to a marked
  $q$-conformal graph, there is a map $\psi \in [\phi]$, a marked weighted
  graph~$W$, a marked weak length graph~$K$, and a tight sequence of
  marked maps
  \[
  \longseq{W}{c}{G^p}{\psi}{H^q}{f}{K}.
  \]
  In particular,
  \begin{align*}
  E^p_q(\psi) = E^p_q[\phi]
    &= \frac{E^p(f \circ \psi)}{E^q(f)}
      = \SFfrom{}^p_q[\phi]
    = \frac{E_q(\psi \circ c)}{E_p(c)}
      = \SFto{}^p_q[\phi].
  \end{align*}
\end{theorem}

Theorem~\ref{thm:energy-sf} is much harder than
Proposition~\ref{prop:Epq-local-sup}.

\begin{proof}[Proof sketch]
The proof is quite similar to the proof of
Theorem~\ref{thm:emb-sf} in Section~\ref{sec:filling}.
For $p=1$, the tautness results of Section~\ref{sec:taut} give
the result, while for $q=\infty$ this is essentially
Proposition~\ref{prop:p-harmonic}. So assume that $1 < p \le q < \infty$.

For $G^p$ a
$p$-conformal graph with $1 < p < \infty$, there is an invertible
``duality'' map $D_G^p \co \Wgt(G) \to \Len(G)$, defined by setting
\begin{equation}\label{eq:dual-p}
\bigl(D_G^p(w)\bigr)(e) = \alpha(e)w(e)^{1/(p-1)}.
\end{equation}

Then, for a homotopy class as in the statement of
Theorem~\ref{thm:energy-sf}, define an iteration
$\Iter_\phi \co \Len(\Gamma_2) \to \Len(\Gamma_2)$ as follows.
\begin{enumerate}
\item For $\ell \in \Len(\Gamma_2)$, set
  $K = (\Gamma_2, \ell)$, find a $p$-harmonic
  representative~$g$ of $[\id \circ \phi] \co G^p \to K$, and set
  $m(e) = \ell(g(e)) \in \Len(\Gamma_1)$.
\item Set $w = (D^p_1)^{-1}(m) \in \Wgt(\Gamma_1)$, so $(\Gamma_1,w)$
  is the tension-weighted graph of~$g$.
\item Set $v = N_{[\phi]}(w) \in \Wgt(\Gamma_2)$.
\item Set $\Iter_\phi(\ell) = D^q_2(v) \in \Len(\Gamma_2)$.
\end{enumerate}
If $p = q$, $\Iter_\phi$ descends to a map on projective spaces
$P\Iter_\phi \co P\Len(\Gamma_2) \to P\Len(\Gamma_2)$, and one has to
do an analysis of the possible boundary fixed points, entirely
parallel to Section~\ref{sec:partial-lambda-filling}.

If $p < q$, then
$\Iter_\phi$ is not linear on rays, and we do the iteration on
$\Len(\Gamma_2)$ itself. More specifically, on a ray we have
\[
\Iter_\phi(\lambda \ell) = \lambda^{\frac{p-1}{q-1}}\Iter_\phi(\ell).
\]
Since $(p-1)/(q-1) < 1$, if a ray in $\Len(\Gamma_2)$ is mapped to
itself then there is a finite fixed point on the ray. A little more
analysis shows that there are never attracting
fixed points on the boundary of $\Len(\Gamma_2)$, so there must be a
fixed point in the
interior. This gives strip graphs
compatible with $G^p$ and $H^q$ in the sense of
Definition~\ref{def:p-rescale} below, along with a $1$-filling map
between them.
\end{proof}

We now turn to the first alternate definition of $p$-conformal graphs
and $E^p_q$. Recall from Definition~\ref{def:strip-graph} that a strip
graph is tuple $(\Gamma,w,\ell)$ of a graph~$\Gamma$ and a set of
weights~$w$ and lengths~$\ell$ on~$\Gamma$.

\begin{definition}\label{def:p-rescale}
  For $p \in [1,\infty)$, a \emph{$p$-conformal rescaling} of a
  positive strip
  graph $(\Gamma, w, \ell)$ changes the weight, length, and area by
  \[
  (w, \ell, \Area) \mapsto (\lambda^{p-1}w, \lambda\ell, \lambda^p\Area)
  \]
  where $\lambda \co \Edges(\Gamma) \to \RR_+$ is a positive rescaling
  factor on each edge. (The identity $\Area = w\cdot \ell$ is
  preserved.) An \emph{$\infty$-conformal rescaling} instead
  acts by $(w, \ell, \Area) \mapsto (\lambda w, \ell, \lambda \Area)$.
  We write
  $(\Gamma, w_1, \ell_1) \equiv_p (\Gamma, w_2, \ell_2)$ if the two
  strip structures are related by a $p$-conformal rescaling.

  For $p \in (1,\infty]$, we say that a $p$-conformal graph~$(\Gamma, \alpha)$ is
  \emph{compatible} with a positive strip structure $(w, \ell)$
  on~$\Gamma$ if
  \[
  (\Gamma, w, \ell) \equiv_p (\Gamma, 1, \alpha),
  \]
  or equivalently if $\ell(e) = \alpha(e)w(e)^{1/(p-1)}$.
  Thus we may think of a $p$-conformal graph as an equivalence
  class under~$\equiv_p$.
  We say that $(\Gamma, \alpha)$ is compatible with an arbitrary (not
  necessarily positive) strip
  structure if, for each edge~$e$,
  \[
  \ell(e)^{p-1} = \alpha(e)^{p-1} w(e).
  \]
  A $1$-conformal graph is compatible
  with a strip structure if the weights agree (ignoring the lengths).
\end{definition}

To better understand the duality map $D^p$ from
Equation~\eqref{eq:dual-p}, suppose we have a $p$-conformal graph
$(\Gamma,\alpha)$. Then for $w \in \Wgt^+(G)$, the lengths $D^p(w)$
are the unique values so that
$(\Gamma, w, D^p(w)) \equiv_p (\Gamma, 1, \alpha)$.

\begin{definition}
  Let $S_1 = (\Gamma_1, w_1, \ell_1)$ and
  $S_2 = (\Gamma_2, w_2,\ell_2)$ be two marked strip graphs (not
  necessarily balanced), with $S_2$ positive.  For $\lambda > 0$, a
  map $f \co S_1 \to S_2$ is \emph{weakly $\lambda$-filling}
  if it satisfies Conditions~\eqref{item:length-pres}
  and~\eqref{item:weight-scale} of
  Definition~\ref{def:lambda-filling}, dropping the condition that $f$
  is taut as a map between weighted graphs.
\end{definition}

\begin{definition}\label{def:Epq-2}
  For $1 \le p \le q \le \infty$, let $G_1^p$ be a $p$-conformal graph,
  $G_2^q$ be a $q$-conformal graph, and $f \co G_1^p \to G_2^q$ be a PL
  map. Then $E^p_{q,\mathrm{strip}}(f)$ is defined in the following way.
  \begin{enumerate}
  \item If $p < q$, then, as we prove in Proposition~\ref{prop:Epq}
    below, there are strip graphs $S_1$ and $S_2$,
    compatible with subdivisions of $G_1^p$ and $G_2^q$ respectively,
    so that $f$ is weakly $1$-filling as a map from $S_1$ to~$S_2$.
    Then
    \begin{equation}
      E^p_{q,\mathrm{strip}}(f) = \Area(S_2)^{1/p-1/q}.
    \end{equation}
  \item If $p = q<\infty$, there are usually no strip structures
    that make $f$ weakly $1$-filling. Instead,
    take any strip structures $(\Gamma_1,w_1,\ell_1)$ and
    $(\Gamma_2,w_2,\ell_2)$ so that $f$ is length-preserving, and
    define the energy to be the maximum ratio of weights:
    \begin{equation}
      E^p_{p,\mathrm{strip}}(f) = \esssup_{y \in \Gamma_2} \frac{n_f^{w_1}(y)}{w_2(y)}.
    \end{equation}
    This is independent of the choice of strip structures.
  \end{enumerate}
\end{definition}

\begin{proposition}\label{prop:Epq}
  For $\phi \co G_1^p \to G_2^q$ a PL map from a $p$-conformal graph to a
  $q$-conformal graph, $E^p_q(\phi) = E^p_{q,\mathrm{strip}}(\phi)$.
\end{proposition}

\begin{proof}[Proof sketch]
  This is closely related to Proposition~\ref{prop:Epq-local-sup}. For
  $p < q$, subdivide $G_2$ so that $\Fill^p(\phi)$ is constant on each
  edge. Then for $e$ an edge of~$G_2$, construct the strip structure
  $(\Gamma_2, w_2, \ell_2)$ compatible with~$G_2$ so that
  \[
  \ell_2(e) = \alpha_2(e) \cdot \bigl(\Fill^p(\phi)(e)\bigr)^{1/(q-p)}.
  \]
  This determines $\ell_1$ by the condition that $\phi$ be
  length-preserving, and $w_1$ and~$w_2$ by the compatibility
  condition. It is elementary to check that $\phi$ is $1$-filling with
  respect to these strip structures and then verify
  that $E^p_q(\phi) = E^p_{q,\mathrm{strip}}(\phi)$.

  The
  case $p=q$ is easier, as you can choose $\ell_2$ arbitrarily.
\end{proof}

For the final variation on the definition of $E^p_q$, we allow more
general spaces than graphs.

\begin{definition}\label{def:p-space}
  For $1 \le p \le \infty$, a \emph{$p$-conformal space} is (loosely)
  a tuple $(X,\ell,\mu)$ of a space~$X$, a length metric $\ell$ on~$X$, and a
  measure~$\mu$ on~$X$, up to rescaling by
  \[
  (X, \ell,\mu) \equiv_p (X, \lambda\ell, \lambda^p\mu)
  \]
  for a suitable rescaling function $\lambda \co X \to \RR_{>0}$. Write
  $[(X,\ell,\mu)]_p$ for an equivalence class of~$\equiv_p$.
\end{definition}

There are analytic subtleties in Definition~\ref{def:p-space} in, e.g.,
how to define the rescaling and exactly which metrics are allowed; we
do not attempt to address them in this paper. But note that oriented
conformal $n$-manifolds~$M^n$ give
examples of $n$-conformal spaces: given a conformal class of
(Riemannian) metrics
on~$M$, pick a base metric~$g$ in the conformal class, and set $\ell$
and $\mu$ to be distance with respect to~$g$ and the Lebesgue measure
of~$g$, respectively. Picking a different metric in the conformal
class changes $\ell$ and $\mu$ by an $n$-conformal rescaling.

Suppose $X = \Gamma$ is a graph with a base metric and associated measure and
\begin{itemize}
\item $\ell$ is a piecewise-constant multiple of the base metric;
\item $\mu$ is a piecewise-constant multiple of the base Lebesgue
  measure; and
\item the rescaling functions~$\lambda$ are piecewise-constant.
\end{itemize}
Definition~\ref{def:p-space} is then almost identical to
Definition~\ref{def:p-rescale}, if we define the weight at a generic
point $x\in \Gamma$ by
\[
w(x) = \frac{\mu(\Delta x)}{\ell(\Delta x)}
\]
where $\Delta x$ is a small interval
centered on~$x$.

\begin{definition}\label{def:Epq-3}
  For $1 \le p \le q \le \infty$ with $p < \infty$ and $\phi \co X_1^p \to X_2^q$ a
  suitable map from a $p$-conformal space to
  a $q$-conformal space, with $X_i = [(X_i, \ell_i,\mu_i)]$, define
  \begin{align*}
    \Fill^p_{\mathrm{conf}}(\phi) &\co Y^q \to \RR_{\ge 0}\\
    \Fill^p_{\mathrm{conf}}(\phi) &\coloneqq \phi_*\bigl(\bigl(\Lip^{\ell_1}_{\ell_2}(\phi)\bigr)^p\cdot \mu_1\bigr) / \mu_2\\
    E^p_{q,\mathrm{conf}}(\phi) &\coloneqq \bigl(\lVert \Fill^p_{\mathrm{conf}}(\phi) \rVert_{q/(q-p),X_2}\bigr)^{1/p}\\
  \end{align*}
  To take the definition of $E^p_{q,\mathrm{conf}}$ step-by-step:
  \begin{itemize}
  \item $\Lip^{\ell_1}_{\ell_2}(\phi) \co X_1 \to \RR_+$ is the
    local Lipschitz constant of~$\phi$.
  \item Next, $\phi_* \bigl(\bigl(\Lip^{\ell_1}_{\ell_2}(\phi)\bigr)^p\cdot\mu_1\bigr)$ is the
    push-forward of measures.
  \item $\Fill^p_{\mathrm{conf}}(\phi) = \phi_*(\Lip(\phi)^p\cdot \mu_1) / \mu_2$ is the
    Radon-Nikodym derivative of the two measures.
  \item Finally, $E^p_{q,\mathrm{conf}}(\phi)$ is (up to a power) the
    $L^{q/(q-p)}$-norm of $\Fill^p_{\mathrm{conf}}$.
  \end{itemize}
  We will not define which maps $\phi$ are ``suitable'' (or indeed
  which tuples $(X,\ell,\mu)$ are allowed), but
    it should include cases where $\phi$ is Lipschitz and the
    Radon-Nikodym derivative exists, i.e., $\phi_*(\Lip(\phi)^p\cdot\mu_1)$
    is absolutely continuous with respect to $\mu_2$. This includes
    non-constant PL
    maps between graphs.

  For $q = \infty$ (so that $X_2$ is a length space),
  $E^p_{\infty,\mathrm{conf}}$ can be rewritten
  \begin{equation}\label{eq:Ep-3}
  E^p_{\infty,\mathrm{conf}}(\phi)
    = \bigl\lVert \Lip^{\ell_1}_{\ell_2}(\phi)\bigr\rVert_{p,X_1}.
  \end{equation}
  In this case we do not need the Radon-Nikodym derivative.
\end{definition}

The motivation for the exponents in
Definition~\ref{def:Epq-3} is that, up to an overall power,
$E^p_{q,\mathrm{conf}}$ is
the unique expression constructed with this data and these operations
that is invariant under both $p$-conformal rescaling on~$X_1^p$ and
$q$-conformal rescaling on~$X_2^q$.

\begin{proposition}
  For $f \co G^p \to H^q$ a PL map from a $p$-conformal graph to a
  $q$-conformal graph, $E^p_q(f) = E^p_{q,\mathrm{conf}}(f)$.
\end{proposition}

\begin{proof}
  Expand the definitions.
\end{proof}

Definitions~\ref{def:p-space} and~\ref{def:Epq-3} point to a
considerably more general setting, likely with
substantial new
difficulties. As mentioned in Section~\ref{sec:prior-work}, much
prior attention has been devoted to proving the existence of harmonic
maps between various types of spaces \cite{EF01:HarmonicPolyhedra}
(related to minimizing
$E^2_\infty$), and the general case is likely
to be harder.

\begin{warning}
  The energy $E^2_\infty$ from Definition~\ref{def:Epq-3} does not agree with
  other definitions
  of Dirichlet energy. For instance, suppose $X_1$ is a Riemann
  surface~$\Sigma$ (with its natural $2$-conformal
  structure), and $X_2$ is a Riemannian $n$-manifold~$M$.
  Pick a base metric~$g$ on~$\Sigma$ in the given conformal
  class. Then, given a smooth map $f \co \Sigma \to M$, we can
  consider the Jacobian $df_x\co T_x\Sigma \to TM$, with singular values
  $\lambda_1, \lambda_2 \co \Sigma \to \RR_{\ge 0}$, the eigenvalues of
  $\sqrt{(df_x)^T(df_x)}$, chosen so that
  $\lambda_1(x) \ge \lambda_2(x)$. Thus $df_x$ maps the unit circle
  in $T_x\Sigma$ to an ellipse whose major and minor axis have length
  $\lambda_1(x)$ and $\lambda_2(x)$. The local Lipschitz constant
  of~$f$ at~$x$ is $\lambda_1(x)$, so we have
  \begin{align*}
    \bigl(E^2_{\infty,\textrm{conf}}(f)\bigr)^2 &= \int_\Sigma \lambda_1(x)^2\,\mu(x)\\
    \intertext{(where $\mu$ is Lebesgue measure on~$\Sigma$) while
    the standard Dirichlet energy is}
    \Dir(f) &= \int_\Sigma \bigl(\lambda_1(x)^2 + \lambda_2(x)^2\bigr)\,\mu(x).
  \end{align*}
  These energies are both conformally invariant, but are not the same.
  They do agree if the target space is
  a graph, as in that case $\lambda_2(x) = 0$.
\end{warning}


\section{Electrical networks}
\label{sec:electrical}

As mentioned in the introduction, the elastic graphs of this paper are
closely related to the much
better studied theory of \emph{resistor networks}. Suppose we are
given an elastic graph~$G$ with $k$ marked vertices $x_1, \dots, x_k$,
called \emph{nodes}. Turn it into a network of resistors, where the
elastic constants
$\alpha(e)$ become resistances. If we attach external voltage sources at
voltages $V_1, \dots, V_k$ to the nodes, then the remainder of the
circuit will reach an electrical equilibrium, which has several pieces
of data:
\begin{itemize}
\item a voltage $V(v)$ for each vertex~$v$ of~$G$ (agreeing
  with $V_1,\dots,V_k$ on the nodes);
\item an internal current $I(\vec e)$ flowing through each oriented edge $\vec
  e$ of~$G$ (with $I(-\vec e) = -I(\vec e)$;
\item the total current $I_1, \dots, I_k$ flowing out of the nodes; and
\item the total energy $E$ dissipated by the system per unit time.
\end{itemize}
At equilibrium, these are related by Kirchhoff's current laws.
\begin{itemize}
\item The current on an edge is
  related to the voltage difference. If $\vec e$ has source~$s$ and
  target~$t$, let $\Delta V(\vec e) = V(t) - V(s)$; then
  \[I(\vec e) = \frac{\Delta V(\vec e)}{\alpha(e)}.\]
\item For each internal (unmarked) vertex~$v$ of~$G$, the total
  current flowing in is~$0$
  \[ \sum_{\vec e\text{ incident to }v} I(\vec e) = 0, \]
  while at the node $x_i$,
  \[
  I_i = \sum_{\vec e\text{ incident to }x_i} I(\vec e).
  \]
\item The energy dissipated is
  \[
  E = \sum_{e\in\Edges(G)} \alpha(e) I(\vec e)^2 =
   \sum_{e \in \Edges(G)} \frac{(\Delta V(\vec e))^2}{\alpha(e)}.
  \]
\end{itemize}
The energy dissipated is identical to
Equation~\eqref{eq:dir-1} for the Dirichlet energy of a map~$f$, in
the special case that the target of~$f$ is $\RR$ with $k$ marked
points at $V_1,\dots,V_k$.

For resistor networks, the equations for the internal
voltages and currents are linear, so $V(v)$, $I(\vec e)$, and $I_i$
are linear functions of the~$V_i$, while $E$ is a quadratic function
of the~$V_i$.  (By contrast, in the more general case considered in
the bulk of this paper, the energy $\Dir_{[f]}$ as a function of
lengths is only
piecewise-quadratic.) The \emph{response matrix} $\Lambda_{ij}(G)$ of
a resistor
network~$G$ is the matrix that gives the external currents~$I_i$ as a
function of the external voltages~$V_j$. The matrix $\Lambda$ is symmetric and
determines $E$ as a quadratic function of the~$V_i$:
\begin{equation}\label{eq:energy-response}
  E = \sum_{i,j=1}^k V_i V_j \Lambda_{ij}.
\end{equation}

Much attention has been devoted to the question of when two resistor
networks (with the same number of nodes) are \emph{electrically
  equivalent}, in the sense that the response matrices are the
same. Series and parallel reduction of resistors are examples of
electrical equivalence. A more substantial example
\cite{Kennelly99:TriangleStar} is the \emph{$Y$--$\Delta$ transform},
that relates a 3-node network with the topology of a~$Y$ to one with
the topology of a~$\Delta$:
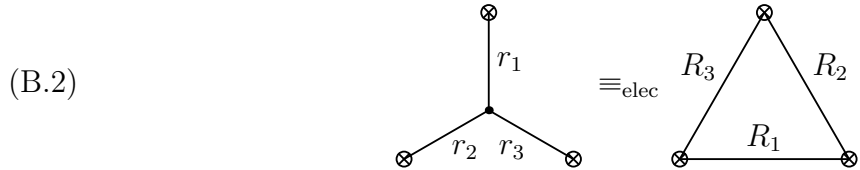
\begin{equation}\label{eq:Y-Delta}
\mfigb{energies-6} \equiv_{\mathrm{elec}} \mfigb{energies-21}
\end{equation}
where
\begin{align*}
  r_1 &= \frac{R_2 R_3}{R_1+R_2+R_3} &
        R_1 &= \frac{r_2 r_3 + r_1 r_3 + r_1 r_2}{r_1} \\
  r_2 &= \frac{R_1 R_3}{R_1+R_2+R_3} &
        R_2 &= \frac{r_2 r_3 + r_1 r_3 + r_1 r_2}{r_2} \\
  r_3 &= \frac{R_1 R_2}{R_1+R_2+R_3} &
        R_3 &= \frac{r_2 r_3 + r_1 r_3 + r_1 r_2}{r_3}
\end{align*}
It turns out that the $Y$--$\Delta$ transform and
series and parallel reduction are sufficient to relate any two
electrically equivalent planar resistor networks with
nodes on the external face
\cite{CdV94:ResElecI,CGV96:ResElecII,CIM98:CircNetworks}.

We can also ask when one resistor network~$G_1$ \emph{dominates}
another network~$G_2$, in the sense that the energy dissipated
by~$G_1$ is greater than the energy dissipated by~$G_2$, for any
choice of external voltages~$V_i$. From
Equation~\eqref{eq:energy-response} we see that this happens exactly
when $\Lambda(G_1) \preceq \Lambda(G_2)$, in the usual Löwner ordering
on quadratic forms.

Elastic networks with targets more general than $\RR$ are
almost never equivalent, so instead we ask about domination, as in
Theorem~\ref{thm:emb-sf}. If $G_1$, $G_2$ are elastic networks with
$k$ corresponding marked nodes, say that
$G_1 \preceq_{\mathrm{elast}} G_2$ if, for any maps $\phi_i$ from
$G_i$ to a marked tree~$K$ mapping corresponding nodes to the same
point,
\[
\Dir[\phi_1] \le \Dir[\phi_2].
\]
Then, for
instance, we have the following inequalities of energies:
\begin{equation}
  \label{eq:Y-Delta-examp}
\mfigb{energies-23} \preceq_{\mathrm{elast}} \mfigb{energies-1}
   \preceq_{\mathrm{elast}} \mfigb{energies-22},
\end{equation}
The second two graphs are electrically equivalent
by~\eqref{eq:Y-Delta}. The first inequality is a simple application of
Theorem~\ref{thm:emb-sf}, while the second requires a little more
argument.
Since the second and third graphs are electrically equivalent, and
those two graphs have equal
energies when $K$ is $\RR^n$ (with three marked points) for any~$n$.
(The case $n=1$ is exactly
electrical equivalence, and for $n \ge 2$ the Dirichlet
energy is the sum of the energies of the projections to the different
coordinates.)
In fact, the inequalities in~\eqref{eq:Y-Delta-examp} hold more
generally when $K$ is any CAT(0)
space
with three marked points.

More generally, a connected resistor network~$G$ with three nodes
is always
electrically equivalent to a tripod with some weights and a triangle
with some other weights, related to each other by the $Y$--$\Delta$
transform. The tripod is always elastically
dominated by the electrically equivalent triangle. In general, in
forthcoming joint work with Dejean and
Gorski, we have the following theorem.

\begin{citethm}[Dejean-Gorski-Thurston \cite{DGT:DirichletEnergy}]
  Let $G$ be an elastic network with three nodes, and let $Y_G$ and
  $\Delta_G$ be the electrically equivalent tripod and triangle,
  respectively. Then
  \[
    Y_G \preceq_{\mathrm{elast}} G \preceq_{\mathrm{elast}} \Delta_G.
  \]
\end{citethm}

\begin{figure}
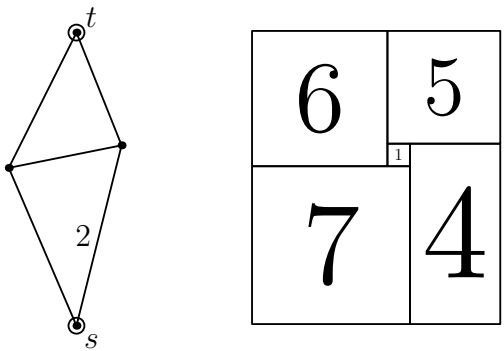

  \[
  \mfig{squares-6}
  \qquad\qquad
  \mfig{squares-5}
  \]
  \caption{A simple electric network and an associated tiling by
    rectangles. All resistances on the graph are~$1$ except for
    one edge, which has resistance~$2$. On the rectangle tiling, we
    have shown the total current flowing
    through the edge, which in the picture is the
    width of the rectangles.}
  \label{fig:rect-tiling}
\end{figure}

We close by reminding the reader of the connection between electrical
networks at equilibrium and rectangle tilings
\cite{BSST40:DissectionRects}: loosely, if you assign each edge a
rectangle of length equal to the voltage difference between the
endpoints and width equal to the total current, then Kirchhoff's laws
say that the rectangles may be assembled into a single tiling, in
which the aspect ratios are equal to the resistances. See
Figure~\ref{fig:rect-tiling} for a simple example. In the more general
setting of elastic graphs, the ``weights'' throughout this paper can
also be reinterpreted as ``widths'', giving similar tiling
pictures on a surface.

Let us spell this out a bit more, although without proofs, and
omitting background on quadratic differentials. Suppose that
$f \co G \to K$ is a
surjective harmonic map from an elastic graph to a length graph, and
suppose that $G$ and~$K$ both have a ribbon structure and $f$ is a
ribbon map, i.e., extends to an embedding
$Nf \co NG \hookrightarrow NK$ between surface thickenings of $G$
and~$K$
\cite[Def.~\ref*{Char:def:ribbon-map}]{Thurston16:Characterize}. Let
$\Sigma = NK$ be the thickening
of~$K$. Recall that a \emph{flip-translation structure} on~$\Sigma$ is
an atlas with overlap maps given by translations or rotations
by~$\pi$, with certain allowed singularities; equivalently, it is a
pair of a conformal structure~$\omega$ on~$\Sigma$ and a quadratic
differential~$q$ with respect to~$\omega$. Then there is a
flip-translation structure on~$\Sigma$ with the following properties.
\begin{enumerate}
\item The boundaries of~$\Sigma$ are horizontal.
\item Corresponding to each vertex of $K$ of valence~$v$, there is a
  zero of~$q$ of order $v-2$. (Note that $K$ cannot have vertices of
  valence~$1$.)
\item The quotient of $q$ by the vertical foliation is the length
  graph~$K$, with length induced by the measure on the vertical
  foliation.
\item $\Sigma$ is tiled by rectangles (axis-aligned in the
  flip-translation structure), with each edge~$e$ corresponding to a
  rectangle $R(e)$. The aspect ratio of $R(e)$ is~$\alpha(e)$, while
  the length (horizontal extent) of $R(e)$
  is equal to the length of $f(e)$ in~$K$.
\item The measure on the horizontal foliation of~$q$ gives weights
  related to a weighted graph~$W_f$ forming a tight sequence
  \[
    \shortseq{W}{}{G}{f}{K}
  \]
  in the sense of Theorem~\ref{thm:harmonic-min}. It is tempting to
  think that the leaves of the horizontal foliation itself gives the
  weighted multi-curve~$C$ in Theorem~\ref{thm:harmonic-min}, but this is
  not generally true, as the leaves of the horizontal foliation will
  typically not be closed when $K$ is a sufficiently complicated
  graph.
\end{enumerate}
The fact that a harmonic map gives a quadratic (not holomorphic)
differential on~$\Sigma$, or equivalently the fact that we have a
flip-translation structure rather than a translation structure
on~$\Sigma$, comes from the distinction between electrical and elastic
stretching mentioned in the introduction: an electrical flow through a
network gives an orientation on each edge, while a stretched elastic
graph has unsigned tensions on each edge.


\bibliographystyle{hamsalpha}
\bibliography{dylan,conformal,topo,graphs}

\providecommand{\bysame}{\leavevmode\hbox to3em{\hrulefill}\thinspace}
\providecommand{\href}[2]{#2}
\providecommand{\arXiv}[1]{\eprint{arXiv:#1}}
\providecommand{\eprint}{\begingroup \urlstyle{rm}\Url}
\begin{thebibliography}{CdVGV96}

\bibitem[BS17]{BS17:NSDynamics}
Mark~C. Bell and Saul Schleimer, \emph{Slow north-south dynamics on
  {$\mathcal{PML}$}}, Groups Geom. Dyn. \textbf{11} (2017), no.~3, 1103--1112,
  \arXiv{1512.00829}.

\bibitem[Bes11]{Bestvina11:Bers}
Mladen Bestvina, \emph{A {B}ers-like proof of the existence of train tracks for
  free group automorphisms}, Fund. Math. \textbf{214} (2011), no.~1, 1--12,
  \arXiv{1001.0325}.

\bibitem[BH92]{BH92:TrainTracks}
Mladen Bestvina and Michael Handel, \emph{Train tracks and automorphisms of
  free groups}, Ann. of Math. (2) \textbf{135} (1992), no.~1, 1--51.

\bibitem[BSST40]{BSST40:DissectionRects}
R.~L. Brooks, C.~A.~B. Smith, A.~H. Stone, and W.~T. Tutte, \emph{The
  dissection of rectangles into squares}, Duke Math. J. \textbf{7} (1940),
  312--340.

\bibitem[CFP94]{CFP94:SquaringRectangles}
J.~W. Cannon, W.~J. Floyd, and W.~R. Parry, \emph{Squaring rectangles: the
  finite {R}iemann mapping theorem}, The Mathematical Legacy of {W}ilhelm
  {M}agnus: Groups, Geometry and Special Functions ({B}rooklyn, {NY}, 1992),
  Contemp. Math., vol. 169, Amer. Math. Soc., Providence, RI, 1994,
  pp.~133--212.

\bibitem[CS11]{CS11:DiscreteIsoradial}
Dmitry Chelkak and Stanislav Smirnov, \emph{Discrete complex analysis on
  isoradial graphs}, Adv. Math. \textbf{228} (2011), no.~3, 1590--1630.

\bibitem[CdV94]{CdV94:ResElecI}
Yves Colin~de Verdi{\`e}re, \emph{R\'eseaux \'electriques planaires. {I}},
  Comment. Math. Helv. \textbf{69} (1994), no.~3, 351--374.

\bibitem[CdVGV96]{CGV96:ResElecII}
Yves Colin~de Verdi{\`e}re, Isidoro Gitler, and Dirk Vertigan, \emph{R\'eseaux
  \'electriques planaires. {II}}, Comment. Math. Helv. \textbf{71} (1996),
  no.~1, 144--167.

\bibitem[CIM98]{CIM98:CircNetworks}
E.~B. Curtis, D.~Ingerman, and J.~A. Morrow, \emph{Circular planar graphs and
  resistor networks}, Linear Algebra Appl. \textbf{283} (1998), no.~1-3,
  115--150.

\bibitem[DGT]{DGT:DirichletEnergy}
Baptiste Dejean, Christian Gorski, and Dylan Thurston, \emph{Dirichlet energies
  of 3-marked elastic graphs}, in preparation, Based on an REU project,
  \url{http://www.math.indiana.edu/reu/2016/reu2016.pdf}.

\bibitem[DH93]{DH93:ThurstonChar}
Adrien Douady and John~H. Hubbard, \emph{A proof of {T}hurston's topological
  characterization of rational functions}, Acta Math. \textbf{171} (1993),
  no.~2, 263--297.

\bibitem[Duf62]{Duffin62:ELNetwork}
R.~J. Duffin, \emph{The extremal length of a network}, J. Math. Anal. Appl.
  \textbf{5} (1962), 200--215.

\bibitem[Duf68]{Duffin68:Rhombic}
\bysame, \emph{Potential theory on a rhombic lattice}, J. Combinatorial Theory
  \textbf{5} (1968), 258--272.

\bibitem[EF01]{EF01:HarmonicPolyhedra}
J.~Eells and B.~Fuglede, \emph{Harmonic maps between {R}iemannian polyhedra},
  Cambridge Tracts in Mathematics, vol. 142, Cambridge University Press,
  Cambridge, 2001.

\bibitem[Fer44]{Ferrand44:Preharmonic}
Jacqueline Ferrand, \emph{Fonctions pr\'{e}harmoniques et fonctions
  pr\'{e}holomorphes}, Bull. Sci. Math. (2) \textbf{68} (1944), 152--180.

\bibitem[FM11]{FM11:MetricOutSpace}
Stefano Francaviglia and Armando Martino, \emph{Metric properties of outer
  space}, Publ. Mat. \textbf{55} (2011), no.~2, 433--473, \arXiv{0803.0640}.

\bibitem[Fug57]{Fuglede57:ELFunctComplete}
Bent Fuglede, \emph{Extremal length and functional completion}, Acta Math.
  \textbf{98} (1957), 171--219.

\bibitem[GS92]{GS92:Harmonic}
Mikhail Gromov and Richard Schoen, \emph{Harmonic maps into singular spaces and
  {$p$}-adic superrigidity for lattices in groups of rank one}, Inst. Hautes
  \'Etudes Sci. Publ. Math. (1992), no.~76, 165--246.

\bibitem[HS96]{HS96:ConvergenceCirclePack}
Zheng-Xu He and Oded Schramm, \emph{On the convergence of circle packings to
  the {R}iemann map}, Invent. Math. \textbf{125} (1996), no.~2, 285--305.

\bibitem[Isa41]{Isaacs41:Difference}
Rufus~Philip Isaacs, \emph{A finite difference function theory}, Univ. Nac.
  Tucum\'{a}n. Revista A. \textbf{2} (1941), 177--201.

\bibitem[Kah06]{Kahn06:BoundsI}
Jeremy Kahn, \emph{A priori bounds for some infinitely renormalizable
  quadratics: {I}. {B}ounded primitive combinatorics}, Preprint ims06-05, Stony
  Brook IMS, 2006, \arXiv{math/0609045v2}.

\bibitem[KPT15]{KPT15:EmbeddingEL}
Jeremy Kahn, Kevin~M. Pilgrim, and Dylan~P. Thurston, \emph{Conformal surface
  embeddings and extremal length}, Preprint, 2015, \arXiv{1507.05294}.

\bibitem[Ken99]{Kennelly99:TriangleStar}
A.~E. Kennelly, \emph{Equivalence of triangles and stars in conducting
  networks}, Electrical World and Engineer \textbf{34} (1899), 413--414.

\bibitem[Ker80]{Kerckhoff80:AsympTeich}
Steven~P. Kerckhoff, \emph{The asymptotic geometry of {T}eichmüller space},
  Topology \textbf{19} (1980), no.~1, 23--41.

\bibitem[KMT03]{KMT03:CircleProj}
Sadayoshi Kojima, Shigeru Mizushima, and Ser~Peow Tan, \emph{Circle packings on
  surfaces with projective structures}, J. Differential Geom. \textbf{63}
  (2003), no.~3, 349--397.

\bibitem[KS93]{KS93:SobolevHarmonic}
Nicholas~J. Korevaar and Richard~M. Schoen, \emph{Sobolev spaces and harmonic
  maps for metric space targets}, Comm. Anal. Geom. \textbf{1} (1993), no.~3-4,
  561--659.

\bibitem[Lov04]{Lovasz04:DiscreteAnalytic}
László Lovász, \emph{Discrete analytic functions: an exposition}, Surveys in
  differential geometry. {V}ol. {IX}, Surv. Differ. Geom., vol.~9, Int. Press,
  Somerville, MA, 2004, pp.~241--273.

\bibitem[Mer01]{Mercat01:DiscreteRiemannIsing}
Christian Mercat, \emph{Discrete {R}iemann surfaces and the {I}sing model},
  Comm. Math. Phys. \textbf{218} (2001), no.~1, 177--216.

\bibitem[Pal15]{Palmer15:harmonic-thesis}
David~R. Palmer, \emph{Toward computing extremal quasiconformal maps via
  discrete harmonic measured foliations}, {A.B.} thesis, Harvard University,
  Cambridge, {MA}, November 2015.

\bibitem[PH92]{PH92:CombTrainTracks}
Robert~C. Penner and John~L. Harer, \emph{Combinatorics of train tracks},
  Annals of Mathematics Studies, no. 125, Princeton University Press,
  Princeton, NJ, 1992.

\bibitem[Pic05]{Picard05:MartingalesTrees}
Jean Picard, \emph{Stochastic calculus and martingales on trees}, Ann. Inst. H.
  Poincaré Probab. Statist. \textbf{41} (2005), no.~4, 631--683.

\bibitem[PP93]{PP93DiscreteMinimal}
Ulrich Pinkall and Konrad Polthier, \emph{Computing discrete minimal surfaces
  and their conjugates}, Experiment. Math. \textbf{2} (1993), no.~1, 15--36.

\bibitem[RS87]{RS87:ConvergenceCirclePack}
Burt Rodin and Dennis Sullivan, \emph{The convergence of circle packings to the
  {R}iemann mapping}, J. Differential Geom. \textbf{26} (1987), no.~2,
  349--360.

\bibitem[Smi01]{Smirnov01:PercPlane}
Stanislav Smirnov, \emph{Critical percolation in the plane}, preprint, 2001,
  \arXiv{0909.4499}.

\bibitem[Smi10]{Smirnov10:DiscreteComplexAnalysis}
Stanislav Smirnov, \emph{Discrete complex analysis and probability},
  Proceedings of the {I}nternational {C}ongress of {M}athematicians. {V}olume
  {I}, Hindustan Book Agency, New Delhi, 2010, pp.~595--621.

\bibitem[Ste05]{Stephenson05:IntroCirclePack}
Kenneth Stephenson, \emph{Introduction to circle packing}, Cambridge University
  Press, Cambridge, 2005.

\bibitem[Thu16a]{Thurston16:RubberBands}
Dylan~P. Thurston, \emph{From rubber bands to rational maps: A research
  report}, Res. Math. Sci. \textbf{3} (2016), Art.\ 15, \arXiv{1502.02561}.

\bibitem[Thu16b]{Thurston16:Characterize}
\bysame, \emph{A positive characterization of rational maps}, Preprint, 2016,
  \arXiv{1612.04424}.

\bibitem[Thu86]{Thurston86:Zippers}
William~P. Thurston, \emph{Zippers and univalent functions}, The {B}ieberbach
  conjecture ({W}est {L}afayette, {I}nd., 1985), Math. Surveys Monogr.,
  vol.~21, Amer. Math. Soc., Providence, RI, 1986, pp.~185--197.

\bibitem[Wol95]{Wolf95:JSHarmonic}
Michael Wolf, \emph{On the existence of {J}enkins-{S}trebel differentials using
  harmonic maps from surfaces to graphs}, Ann. Acad. Sci. Fenn. Ser. A I Math.
  \textbf{20} (1995), no.~2, 269--278.

\end{thebibliography}

\end{document}